\documentclass[11pt,reqno]{amsart}

\usepackage{amsmath,amsfonts,amsthm,color,amssymb, mathrsfs}
\usepackage{xcolor,extarrows}
\usepackage{dsfont}
\usepackage{stmaryrd} 
\SetSymbolFont{stmry}{bold}{U}{stmry}{m}{n}
\usepackage[all]{xy}
\usepackage[left=1in, right=1in, top=1.1in,bottom=1.1in]{geometry}
\usepackage{enumitem}
\usepackage{hyperref}
\hypersetup{
colorlinks   = true,
citecolor    = blue,
linkcolor=blue
}
\allowdisplaybreaks
\numberwithin{equation}{section}

\newtheorem{theorem}{Theorem}[section]
\newtheorem{lemma}[theorem]{Lemma}
\newtheorem{corollary}[theorem]{Corollary}
\newtheorem{definition}[theorem]{Definition}

\newtheorem{proposition}[theorem]{Proposition}

\newtheorem{remark}[theorem]{Remark}
\newtheorem{notations}[theorem]{Notations}

\def\bd{\boldsymbol{\delta}}

\def\1{\mathbf 1}
\def\bZ{{\mathbb Z}}
\def\I{\mathbf{I}}
\def\bP{{\mathbb P}}
\def\fP{\mathbf{P}}
\def\rme{{\rme}}\def\al{{\alpha}}\def\be{{\beta}}

\def\om{{\omega}}

\def\e{\varepsilon}

\def\bR{\mathbb R}
\def\R{\bR}
\def\bN{\mathbb N}
\def\bE{\mathbb E}
\def\E{\mathbb E}
\def\fE{\mathbf{E}}
\def\mE{\mathscr{E}}

\def\cH{\mathcal H}
\def\cF{\mathcal F}
\def\cS{\mathcal S}

\def\Dom{\mathrm {Dom}}

\def\dd{{\mathrm d}}
\def\zf{Z_{N,\beta}^{\omega,h}}

\def\sint{\int{\raisebox{-.82em}{\hspace{-1.18em}\scriptsize \bf o}}\hspace{.7em}}
\def\sintn{\int_{[n]}{\raisebox{-.82em}{\hspace{-1.69em}\scriptsize \bf o}}\hspace{.7em}}
\def\ssint{\int{\raisebox{-.33em}{\hspace{-.87em}\scriptsize \bf o}}\hspace{.5em}}

\newcounter{bean}
\newcommand{\benuma}{\setlength{\labelwidth}{.25in}

\begin{list}
{(\alph{bean})}{\usecounter{bean}}}
\newcommand{\eenuma}{\end{list}}

\begin{document}

\title[A fractional stochastic heat equation]{on a fractional stochastic heat equation arising from the disordered pinning model}

\author[Z. Li]{Zi'an Li}\address{Research Center for Mathematics and Interdisciplinary Sciences, Shandong University, Qingdao 266237, China}
\email{202117106@mail.sdu.edu.cn}

\author[J. Song]{Jian Song}\address{Research Center for Mathematics and Interdisciplinary Sciences, Shandong University, Qingdao 266237, China}
\email{txjsong@sdu.edu.cn}

\author[R. Wei]{Ran Wei}\address{Department of Financial and Actuarial Mathematics, Xi'an Jiaotong-Liverpool University, 111 Ren'ai Road, Suzhou 215123, China}
\email{ran.wei@xjtlu.edu.cn}

\author[H. Zhang]{Hang Zhang}\address{Research Center for Mathematics and Interdisciplinary Sciences, Shandong University, Qingdao 266237, China}
\email{202217118@mail.sdu.edu.cn}

\begin{abstract}
We study the mild Skorohod  solution to the following fractional stochastic heat equation on $\R$:
\begin{equation*}
\begin{cases}
\partial_t u(t,x)=-(-\Delta)^{\rho/2} u(t,x) +\beta u(t,x)\delta_0(x)\xi(t), \\
u(0,\cdot)=u_0(x),
\end{cases}
\end{equation*}
where $-(-\Delta)^{\rho/2}$ with $\rho\in(0,2]$ is the fractional Laplacian and $\xi$ is a Gaussian noise with covariance  $\E[\xi(t) \xi(s)]=|t-s|^{2H-2}$ for $H\in(\frac12, 1]$. This equation with $\rho\in(1,2]$ arises naturally in the study of the disordered pinning model.

We show that the equation admits a local $L^2$-solution when $\rho = 2$,  whereas, for $\rho \in (0,2)$, any solution—if it exists uniquely—cannot be $L^p$-integrable for any $p > 1$. Moreover, inspired by the recent work of Quastel et~al.~\cite{qrv}, we prove that the equation has a unique global $L^1$-solution whenever $\tfrac{1}{\rho}+1<2H$. We also establish the strict positivity of the solution.  Our work partially fills the gap in the study of the Weinrib-Halperin prediction.
\end{abstract}
\maketitle

\tableofcontents

\section{Introduction}\label{sec:intro}
 
In this paper, we  study the following fractional stochastic heat equation (fractional SHE) on $\R$:
\begin{equation}\label{e:she}
\begin{cases}
\partial_t u(t,x)=-(-\Delta)^{\rho/2} u(t,x) +\beta u(t,x) \delta_0(x) \xi(t),\\
u(0,\cdot)=u_0(x),
\end{cases}
\end{equation}
where   $-(-\Delta)^{\rho/2}$ with $\rho\in(0,2]$ is the fractional Laplacian, $\delta_0(x)$ is the Dirac delta function,    $u_0(x)$ is a bounded measurable function or a finite measure on $\R$, and $\xi$ is a Gaussian noise with covariance \begin{equation}\label{e:covariance}
\bE \big[\xi(t)\xi(s)\big]=|t-s|^{2H-2},
\end{equation}
with $H\in(\frac12, 1]$ being the  Hurst parameter. This equation arises naturally from the disordered pinning model in correlated random environments and is closely related to the Weinrib--Halperin prediction~\cite{WH83} for disorder relevance and irrelevance (see more details in  Section~\ref{sec:pinning-model}).  The equation of type \eqref{e:she} was recently studied in \cite{WY24} for a special case $\rho=2$ and $H=\tfrac12$, which is known to be marginal (see Section \ref{sec:con_dis}) and has a close connection with the stochastic Volterra equation (see e.g., \cite{BFGMS20}).

Let $X=\{X_t, t\geq0\}$  be a one-dimensional symmetric $\rho$-stable process, whose infinitesimal generator is $-(-\Delta)^{\rho/2}$,  defined on a probability space  $ (\mathfrak X,\mathcal{G},\mathbf P)$.  The Gaussian noise $\xi$ is independent of $X$ and lives  on another probability space $(\Xi,\mathcal{F},\mathbb P)$. The probability and expectation on the product probability space $\mathfrak X\times \Xi$ are denoted by $\mathsf P$ and $\mathsf E$, respectively.

Without loss of generality, throughout the paper we assume  $\beta\equiv 1$  in \eqref{e:she}.

\subsection{Main results and literature review}

Formally applying Duhamel's principle to \eqref{e:she} yields
\begin{align*}
u(t,x)&=\int_\R g_\rho(t,x-y) u_0(y) dy+\int_0^t \int_\R g_{\rho}(t-s,x-y)u(s,y)\delta_0(y)\dd y\, \xi(s)\dd s\\
&=\int_\R g_\rho(t,x-y) u_0(y) dy+\int_0^t g_{\rho}(t-s,x)u(s,0)\xi(s)\dd s. 
\end{align*} 
This observation motivates the definition of a solution to \eqref{e:she}:
\begin{definition}[Mild Skorohod solution]\label{def: solution} A random field $\{u(t,x), t\in[0,T], x\in \R\}$ is called a \emph{mild Skorohod solution} to \eqref{e:she}, if $\{u(t,\cdot), t\in[0,T]\}$ is adapted to the filtration induced by the noise $\xi$, $u(t,x)$ is $L^1$-integrable for each $(t,x)\in[0,T]\times\R$,  and the following integral equation is satisfied:
\begin{equation}\label{e:int-eq}
u(t,x)=\int_\R g_\rho(t,x-y) u_0(y) dy+\sint\1_{[0,t]}(\cdot) g_{\rho}(t-\cdot,x)u(\cdot,0)\xi,
\end{equation}
where $g_\rho$ is the transition density function  of the $\rho$-stable process $X$ and the second integral on the right-hand side is an $L^1$-Skorohod integral (see Section \ref{sec:pre} for the definition).
\end{definition}

We stress that  only the $L^1$-integrability of  $u(t,x)$ is required in Definition \ref{def: solution}, and thus the Skorohod integral in \eqref{e:int-eq} should
only be necessarily $L^1$-integrable. Classical Skorohod integrals are defined as $L^2$-integrable random variables (see \cite[Section 1.3]{Nua06}), while more recently $L^1$-integrable Skorohod integrals have been developed by Quastel, Ramirez, and Vir\'ag \cite{qrv}. Both notions will play a key role in our analysis (see Section~\ref{sec:pre} for more details on Skorohod integrals).

Note that the first integral on the right-hand side of \eqref{e:int-eq} is well-defined as long as the initial condition $u_0(x)$ is a bounded function or a bounded measure. For simplicity, throughout the paper we assume  $u_0(x)\equiv 1$. We also always assume $\rho\in(0,2]$ and $H\in(\frac12,1]$.

We now state our main theorem on the well-posedness of \eqref{e:she}.

\begin{theorem}\label{thm-well-posedness}

Concerning the mild Skorohod solution to the SHE \eqref{e:she} in the sense of Definition~\ref{def: solution}, we have the following results:
\begin{enumerate}
\item When $\rho\in(0,2)$,  there exists no $L^p$-solution for any $p\ge2$; if  the solution to \eqref{e:she} is unique, then there exists no $L^p$-solution for any $p>1$.
\item When $\rho=2$, there exists a unique local $L^p$-solution but no global $L^p$-solution, for any $p>1$.
\item When $\rho\in(\frac{1}{2H-1},2]$, there exists a unique global $L^1$-solution.
\end{enumerate}
\end{theorem}
\begin{proof}
Result (1) follows from Propositions~\ref{prop:no-L^2},~\ref{prop:p-2-ineq} and  Remark~\ref{rem:p-2}; result (2) from Propositions~\ref{prop:local-L^2}, \ref{prop:p-2-ineq} and Remark~\ref{rem:p-2}; and result (3) from Propositions~\ref{prop:exis-uniq-u-n}, \ref{prop:un-u}, and~\ref{prop:Zn-Z}.
\end{proof}

\begin{remark}
In this paper, we show that for $\rho\in(1,2]$, the solution to \eqref{e:she} is unique whenever it exists (see Sections \ref{sec:local}-\ref{sec:project_eq}). The condition $\rho>1$ is imposed to ensure the existence of the local time of the stable process $X$, which plays a crucial role in our analysis. \end{remark}


Stochastic heat equations (SHEs) with multiplicative noise are of particular interest, owing in part to their close connection with the parabolic Anderson model. In most of the existing literature, the solutions are 
$L^p$-integrable for all 
$p\ge 2$ (see, e.g., \cite{walsh, dalang99, hu02, hns11, S17, chen19, cdot21, cgs24}). In this setting, the long-time asymptotics of the 
$p$-th moments of the solutions are intimately related to the so-called \emph{intermittency phenomenon} (see, e.g., \cite{cm94, bc95, khosh14, chsx15, chss18}).

It was shown by Hu \cite{hu02} that the stochastic heat equation (SHE) on $\R^2$
 with multiplicative spatial white noise admits only a local $L^2$-solution $u(t,x)$, in the sense that $\E|u(t,x)|^2$ is finite only for small time $t$. In contrast, Quastel et al.\ \cite{qrv} recently developed a theory of $L^1$-integrable Skorohod integrals and demonstrated that the same equation admits a global solution that is $L^1$-integrable  for all time $t>0$.

Theorem~\ref{thm-well-posedness} provides a detailed characterization of $L^p$-integrable  solutions with $p\ge 1$ for the SHE~\eqref{e:she} with $\rho\in(0,2]$ and $H\in(\frac12, 1]$. In particular, it shows that \eqref{e:she} admits a ``genuine'' $L^1$-solution  when $\rho\in(\frac1{2H-1}, 2)$, in the sense that there is no $L^p$-solution for any $p>1$. To our best knowledge, this is the first equation demonstrating this property, noting that the parabolic Anderson model on $\R^2$ considered in \cite{qrv} still has a local $L^2$-solution. This phenomenon may be also related to some phase transition properties of disordered pinning models (see more discussions in Section \ref{sec:pinning-model}).

We now explain our approach to analyzing $L^p$-Skorohod solutions for $p>1$ and for $p=1$, respectively.

Clearly, an $L^p$-solution with $p>2$ implies an $L^2$-solution by Jensen’s inequality; on the other hand, by hypercontractivity, an $L^p$-solution with $p\in(1,2)$ also implies an $L^2$-solution (see Proposition~\ref{rem:p-2}). Thus, for $p>1$, the study of $L^p$-solutions essentially reduces to the study of $L^2$-solutions, which can be carried out by analyzing the Wiener chaos expansion of the solution (see, e.g., \cite{hu02, hns11, S17, cdot21}).

A formal iteration of \eqref{e:int-eq} with $u_0(x)\equiv 1$ yields
\begin{align*}
&u(t,x)\\
&=1+\int_0^tg_{\rho}(t-s_1,x)u(s_1,0)\xi(s_1)\dd s_1
\\&=1+\int_0^tg_{\rho}(t-s_1,x)\left(1+\int_0^{s_1} g_{\rho}(s_1-s_{2},0)u(s_{2},0)\xi(s_{2})\dd s_{2}\right)\xi(s_1)\dd s_1
\\&=1+\sum_{n=1}^{\infty}~~\idotsint\limits_{0<s_1<\cdots<s_n<t}g_{\rho}(t-s_n,x)g_{\rho}(s_n-s_{n-1},0)\cdots g_{\rho}(s_2-s_1,0)\xi(s_{n-1})\cdots\xi(s_1)\dd s_1\cdots\dd s_n.
\end{align*}
Then, if $u(t,x)$ is $L^2$-integrable, it admits the following unique Wiener chaos expansion:
\begin{equation}\label{e:sol-chaos}
u(t,x)=1+\sum_{n=1}^{\infty}\mathbf{I}_n(f_n(t,x)),
\end{equation}
where
\begin{equation}\label{e:fn}
\begin{aligned}
&f_n(t,x;s_1,\cdots,s_n)\\&=
\mathrm{Sym}\left[g_{\rho}(t-s_n,x)g_{\rho}(s_n-s_{n-1},0)\cdots g_{\rho}(s_2-s_1,0) \mathbf{1}_{[0,t]^n_<}(s_1,\cdots s_n)\right]\\
&=\frac1{n!}\sum_{\sigma\in \mathscr P_n} g_\rho(t-s_{\sigma(n)}, x) g_\rho(s_{\sigma(n)}-s_{\sigma(n-1)},0)\cdots g_\rho(s_{\sigma(2)}-s_{\sigma(1)},0)\mathbf{1}_{[0,t]^n_<}(s_{\sigma(1)},\cdots s_{\sigma(n)}),
\end{aligned}   
\end{equation}
with   $\mathscr P_n$ being the set of the permutations on $\{1, 2, \dots, n\}$ and $[0,t]^n_<:=\{\boldsymbol{s}=(s_1\cdots s_n)\in\bR^n,~0<s_1<\cdots s_n<t\}$, 
and
\begin{equation*}
\begin{aligned}
\mathbf{I}_n(f_n(&t,x))=\int_0^t \cdots \int_0^t f_n(t,x;s_1,\cdots,s_n)\xi(s_1)\cdots\xi(s_n)\dd s_1\cdots\dd s_n
\end{aligned}   
\end{equation*}
is an $n$-th multiple Skorohod integral.

It is well-known that the existence and uniqueness of an $L^2$-solution to~\eqref{e:she} is equivalent to
\begin{align}\label{e:second_moment}
\E\big[u(t,x)^2\big]=\sum_{n=0}^\infty n! \|f_n\|^2_{\mathcal H^{\otimes n}}<\infty,    
\end{align}
with  $\mathcal H$ being the Hilbert space associated with the Gaussian noise $\xi$ (see Section~\ref{sec:Malliavin}), which can be verified by direct computations, noting that $f_n$ is given explicitly in~\eqref{e:fn}. We refer readers to Section~\ref{Sec: Lp} for further details.

The treatment of the $L^1$-solution is fundamentally different from that of the $L^2$-solution, and our methodology is greatly inspired by the recent work \cite{qrv}.

Let $\{e_k\}_{k\in\mathbb N}$ be an orthonormal basis of $\mathcal H$. Denote 
\begin{equation}\label{e:mM}
\begin{aligned}
m_{k}(t)&:=\int_0^t\int_0^T \delta_0(X_s)e_k(r)|s-r|^{2H-2}\dd r\dd s,\\
M_{n}(t)&:=\exp\left\{\sum_{k=1}^n\left(m_{k}(t)\xi_k-\frac{1}{2}m_{k}(t)^2\right)\right\},
\end{aligned}
\end{equation}
where  $\delta_0(\cdot)$ is the Dirac delta function and 
\begin{equation}\label{e:xi-k}
\xi_k:=\int_0^T e_k(s) \xi(s)\dd s
\end{equation}
is the Wiener integral of $e_k$ (see Section \ref{sec:Malliavin}).
The term $m_k(t)$  has an alternative expression \eqref{e:mkt}, which is an integral of a continuous function against the measure induced by the local time of $X$.


Let $\mathbf P_{(t,x)}$ denote the law of a stable process on $[0,t]$ with an arbitrary initial value, conditioned to be at $x$ at time $t$. Alternatively, we can view $\mathbf P_{(t,x)}$ as the law of the ``backward stable process'' $\{\tilde{X}^{(t,x)}_s=x+X_{t-s},s\in[0,t]\}$ with the terminal value $\tilde X_t^{(t,x)}=x$. Similarly, we use $\mathbf E_{(t,x)}$ to denote the expectation under the law  $\mathbf P_{(t,x)}$. 

Denote
\begin{equation*}
Z_n=Z_n(t,x)=u_n(t,x):=\mathbf E_{(t,x)}\big[M_{n}(t)\big].
\end{equation*}
Then, $u_n(t,x)$ solves a finite-dimensional projection of \eqref{e:she}, where the noise $\xi$ is projected onto the Gaussian subspace spanned by$\{\xi_1,\dots, 
\xi_n\}$. Specifically, 
\[u_n(t,x)=1+\E\left[\sint  \1_{[0,t]}(\cdot) g_{\rho}(t-\cdot,x)u(\cdot,0)\xi\bigg| \mathcal F_n \right],\] 
where $\mathcal F_n=\sigma(\xi_1, \dots, \xi_n)$. 
Clearly, $\{Z_n\}_{n\in\mathbb N}$ is a nonnegative martingale and hence converges almost surely. Assuming
$\rho\in(\frac1{2H-1}, 2]$, we prove the uniform integrability of ${Z_n}$ and show that the $L^1$-limit 
\begin{equation}\label{e:Z}
u(t,x):=Z=\lim_{n\to \infty} Z_n
\end{equation}
solves \eqref{e:she} in the sense of Definition~\ref{def: solution}. Our analysis of uniform integrability relies on the finiteness of the \emph{self-energy} of the local time of the $\rho$-stable process $X$:
\begin{equation}\label{e:I-t}
I_t:=\sum_{k=1}^\infty m_k^2(t)=\int_0^t\int_0^t\delta_0(X_r) \delta_0(X_s)|r-s|^{2H-2}\dd r\dd s.
\end{equation}
As a byproduct, we show that $I_t$ is finite almost surely if and only if $\rho>\frac1{2H-1}$, which is of interest by itself. Readers are referred to Section~\ref{sec: L1-sol} for details.

\

One advantage of the above-mentioned martingale approach is that it allows us to establish a $0$--$1$ law for the solution (see Proposition~\ref{prop:0-1}), and subsequently deduce its strict positivity:
\begin{theorem}\label{thm:positivity}
When $\rho\in(\frac1{2H-1},2]$, the solution to \eqref{e:she} is positive almost surely. 
\end{theorem}
\begin{proof}
It follows directly from Proposition \ref{prop: positivity}.    
\end{proof}

As solutions $u(t,x)$ to SHEs can be viewed as partition functions of polymer models, it is desirable to establish the strict positivity of the solutions in order to define the free energy $\log u(t,x)$. In particular, the Cole--Hopf transformation $\log u(t,x)$ of the solution $u(t,x)$ to the SHE with multiplicative space--time white noise solves the KPZ equation. We refer to \cite{mueller91, flores14, ch19, qrv} for positivity results on SHEs.

\

Let $G\sim N(0,1)$ be a standard Gaussian random variable and $a\in\R$ be a real number.  Then for all bounded measurable functions $f$, we have $$E[Yf(G)] = E[f(G+a)] \text{ with } Y=e^{aX-\frac12a^2}.$$  
This identity is a manifestation of the Cameron–Martin theorem.
In infinite dimensions, an analogous change-of-measure formula holds for Gaussian fields: exponential tilting by a (regularized) Gaussian functional shifts the field by an element of the Cameron–Martin space. When the shift $a$ itself is randomized and is defined in a distributional, rather than pointwise, sense, the resulting change of measure gives rise to  the theory of Gaussian multiplicative chaos (GMC) (see, e.g., \cite{Sha16}).

In our setting (see equations \eqref{e:mM} and \eqref{e:xi-k}), $\xi=(\xi_1,\xi_2, \cdots)$ is an infinite-dimensional random vector consisting of i.i.d.\ standard Gaussian random variables on the probability space $(\Xi, \mathcal F, \mathbb P)$, and $m(t)=(m_1(t), m_2(t), \cdots)$ is a vector of shift randomized  by $X$ defined on $(\mathfrak X, \mathcal G, \mathbf P)$. We introduce the following definition.  

\begin{definition}[Partition function]\label{def-Par} A random variable $J$ defined on $\Xi$ is called the \emph{partition
function} or \emph{total mass} of the \emph{randomized shift} $m(t)$, if for any bounded measurable function $F:\Xi\to\R$, we have
\begin{equation}\label{e:partition}
\E[JF(\xi)]=\E\fE_{(t,x)} [F(\xi+m(t))].
\end{equation}
An equivalent formulation is that the marginal distribution of $\xi+m(t)$ is absolutely continuous with respect to that of $\xi$ with density (Radon-Nikodym derivative) $J$. 
\end{definition}

When $\rho\in(\frac1{2H-1},2]$,  the randomized shift $m(t)$ is square integrable, i.e.,  $\sum_{k=1}^\infty m_k^2(t)(\omega)<\infty$ $\mathbf P$-a.s. (see Corollary~\ref{cor:XX} and Remark~\ref{rem:self-energy}). Then $\sum_{k=1}^\infty m_k(t)(\omega)\xi_k-\frac12 m_k^2(t)(\omega)$ exists $\mathbf P$-a.s.,  and  we can define a random measure on $\mathfrak X$ parametrized by $\xi$:
\begin{equation}\label{e:M-xi}
M_{\xi,t}(\dd \omega)=\exp\left\{\sum_{k=1}^\infty \left(m_k(t)(\omega)\xi_k-\frac12 m_k^2(t)(\omega)\right) \right\}\mathbf P_{(t,x)}(\dd \omega).
\end{equation}
One can check (see, e.g.,  \cite[Example 12]{Sha16}) that the $\mathbb P$-random measure $M_{\xi,t}$ on $\mathfrak (X,\mathcal G)$ satisfies 
\begin{equation}
\begin{aligned}\label{e:GMC}
\bE M_{\xi,t}[F(\xi,\omega)]=\mathbb E\mathbf E_{(t,x)}[F(\xi+ m(t,\omega),\omega)],
\end{aligned}
\end{equation}
for any bounded measurable function $F:\Xi\times\mathfrak X\to\bR$, where we abuse $M_{\xi,t}[\cdot]$ to denote the integral against the  measure $M_{\xi,t}(\dd \omega)$. Then $M_{\xi,t}$ is referred to as the unnormalized polymer measure (or the Gaussian multiplicative chaos associated with the randomized shift \(m(t)\)) on \(\mathfrak X\), and the normalized measure \(M_{\xi,t} / M_{\xi,t}(\mathfrak X)\) is called the polymer measure.

When $F$ in \eqref{e:GMC} is independent of $\omega\in \mathfrak X$, the equality \eqref{e:GMC} reduces to \eqref{e:partition} with the partition function given by 
\begin{equation}\label{e:partition-M}
J=M_{\xi,t}(\Omega)=\mathbf E_{(t,x)}\left[ \exp\left\{\sum_{k=1}^\infty\left( m_k(t)\xi_k-\frac12 m_k^2(t)\right) \right\}\right],
\end{equation}
which in particular holds when $\rho\in(\frac1{2H-1},2]$. 

Now we are ready to present the Feynman--Kac formula. 
\begin{theorem}[Feynman--Kac formula]\label{thm-gmc} Assume $\rho\in(\frac1{2H-1},2]$. Then, we have the following Feynman--Kac formula for the solution $u(t,x)$ to \eqref{e:she}
\begin{equation}\label{e:FK}
\begin{aligned}
u(t,x)&=\lim_{n\to\infty }\mathbf E_{(t,x)}\left[ \exp\left\{\sum_{k=1}^n\left( m_k(t)\xi_k-\frac12 m_k^2(t)\right) \right\}\right]\\
&=\mathbf E_{(t,x)}\left[ \exp\left\{\sum_{k=1}^\infty \left(m_k(t)\xi_k-\frac12 m_k^2(t) \right)\right\}\right].
\end{aligned}
\end{equation}
\end{theorem}

\begin{proof}
The first equality of \eqref{e:FK} follows from Propositions~\ref{prop:exis-uniq-u-n}, \ref{prop:un-u}, and~\ref{prop:Zn-Z}.  Proposition~\ref{prop:Zn-Z}  and  Lemma~\ref{ab-con} yield that \eqref{e:GMC} holds with $J=\lim_{n\to\infty }\mathbf E_{(t,x)}\left[ \exp\left\{\sum_{k=1}^n m_k(t)\xi_k-\frac12 m_k^2(t) \right\}\right]$. Together with \eqref{e:partition-M}, this yields the desired second equality.
\end{proof}

\begin{remark}\label{rem:FK}
When $\rho\in(\frac1{2H-1},2]$, we have (see Remark \ref{rem:self-energy} and Corollary \ref{cor:XX}) $$\sum_{k=1}^\infty m_k^2(t) = \int_0^t \int_0^t \delta_0(X_r) \delta_0(X_s) |r-s|^{2H-2}\dd r\dd s <\infty, \quad\fP_{(t,x)}\text{-a.s.}$$
and thus $\delta_0(X)\in \mathcal H$, $\fP$-a.s. (see also Proposition~\ref{prop:delta-L2}). By \eqref{e:delta-u-determin}, we have 
\[ \sum_{k=1}^\infty m_k(t)\xi_k=\sint \mathbf 1_{[0,t]} \delta_0(X) \xi=:\bd\big(\mathbf 1_{[0,t]} \delta_0(X)\big).\]
Therefore, the Feynman--Kac formula \eqref{e:FK} can also be written as
\begin{equation}\label{e:FK'}
\begin{aligned}
u(t,x)&=\mathbf E_{(t,x)}\left[\exp\left\{\sint\1_{[0,t]} \delta_0(X) \xi-\frac12 \int_0^t \int_0^t \delta_0(X_r) \delta_0(X_s) |r-s|^{2H-2}\dd r\dd s\right\}\right]\\
&=\mathbf E_{(t,x)}\left[ \exp\left\{\bd\big(\mathbf 1_{[0,t]} \delta_0(X)\big)-\frac12 \big\|\1_{[0,t]}\delta_0(X)\big\|^2_{\mathcal H}\right\}\right].
\end{aligned}
\end{equation}
\end{remark}

\begin{remark}\label{rmk:mollified_she}
It is natural to approximate our SHE \eqref{e:she}  with $\beta=1$ by the mollified equation 
\begin{equation}\label{e:she-e}
\partial_t u_\e(t,x)=-(-\Delta)^{\rho/2} u_\e(t,x) + u_\e(t,x) \delta_{\e}(x) \xi(t), \text{ with } u_\e(0,x)=1,
\end{equation}
where 
\begin{equation}\label{e:delta-e}
\delta_\e(x)=\frac{1}{2\e}\mathbf{1}_{\{|x|<\e\}}
\end{equation}
is an approximation of the Dirac delta function $\delta_0(x)$, as $\e\to 0$. Following the approach in \cite{S17}, we can show that the  mild Skorohod solution to \eqref{e:she-e} exists uniquely, given by 
\begin{equation}\label{e:FK-e}
\begin{aligned}
u_\e(t,x)
&=\mathbf E_{(t,x)}\left[ \exp\left\{\bd\big(\mathbf 1_{[0,t]} \delta_\e(X)\big)-\frac12 \big\|\1_{[0,t]}\delta_\e(X)\big\|^2_{\mathcal H}\right\}\right].
\end{aligned}
\end{equation}
Although $u_\e(t,x)$ in \eqref{e:FK-e} provides a formal approximation of   $u(t,x)$ in  \eqref{e:FK'}, it remains unclear whether $u_\e(t,x)$ converges to $u(t,x)$ in $L^1$ when $\rho\in(\frac1{2H-1},2]$, due to the lack of an underlying martingale structure  or a control function.    
\end{remark}

\begin{notations}\label{notations} We adopt the following notations throughout the paper.
\begin{enumerate}
\item\emph{Notations for deterministic cases:} Let $\bN $ denote the set of natural numbers without $0$, i.e., $\bN:=\{1,2,\cdots\}$ and $\bN_{0}:=\{0,1,2,\cdots\} $; for $N\in\bN$, $\llbracket N \rrbracket:=\{1,2,\cdots,N\}$; for $a\in\bR$, $\left[a\right]$ means the greatest integer that is not greater than $a$; $\|\cdot\|$ is used for the Euclidean norm; 
let\ $\boldsymbol{k}:=(k_{1},k_{2},\cdots,k_{d}), \boldsymbol{x}:=(x_{1},x_{2},\cdots,x_{d})$ etc.\ stand for vectors in $\bZ^{d}$ or $\bR^{d}$ depending on the context; we use $C$ to denote a generic positive constant that may vary in different lines;  we write $f(x)\sim g(x)$(as $x\rightarrow\infty$) if $\lim_{x\rightarrow\infty}f(x)/g(x)=1$.
\item\emph{Notations for random cases:} We use $\overset{(\text{d})}{\longrightarrow}$ to denote the convergence in distribution (also called the \emph{weak convergence}) for random variables or random vectors; for a random variable X, $\|X\|_{L^p}:=(\mathrm{E}\left[|X|^{p}\right])^{1/p}$ for $p\geq 1$; we use $\1_A$ to denote the indicator function of some set $A$, i.e., $\1_{A}(t)=1$ if $t\in A$ and $\1_{A}(t)=0$ if $t\notin A$;  for an event $A$ or a random variable $Y$  on $(\mathfrak X, \mathcal G, \mathbf P)$, we use $A'$ and $Y'$ to denote their independent copies, where $A(\omega,\omega')=A(\omega)$ ($A'(\omega,\omega')=\Omega\times A'(\omega')$, resp.) and $Y=Y(\omega,\omega')=Y(\omega)$ ($Y'=Y'(\omega,\omega')=Y'(\omega')$, resp.) can be viewed as defined on the product space $(\mathfrak X, \mathcal G, \mathbf P)^{\otimes 2}$; we use $\int_0^T f(s)\xi(s) \dd s$ or $\bd(f)$ to denote the $L^2$-Skorohod integral and $\ssint f \xi$ to denote the $L^1$-Skorohod integral. 
\end{enumerate}

\end{notations}
\subsection{The disordered pinning model and the Weinrib--Halperin prediction}\label{sec:pinning-model}~In this section, we describe the close connection between our SHE \eqref{e:she} and the disordered pinning model. We also introduce the key notions of \emph{disorder relevance} and \emph{disorder irrelevance}, along with the related Weinrib--Halperin prediction.

The disordered pinning model is a class of polymer models used to describe the interaction of a long polymer chain with a defect line (or interface) carrying random charges or impurities, represented by random environment (or disorder).
The model is defined through a Gibbs transformation of the law of a renewal process, which encodes the return times of the polymer to the defect line, and assigns random energetic rewards or penalties to these contacts. Owing to its simple formulation and rich mathematical structure, the disordered pinning model is particularly well suited for the study of phase transitions and critical phenomena, and it has attracted widespread interest in physics, chemistry, and biology. We refer to the monographs \cite{Gia07} and \cite{Gia11} for further details.

Now we define the disordered pinning model. Let $\tau=\{\tau_{i}\text{,}~i\in\bN_{0}\}$ be a renewal process on $\bN$ with $\tau_0:=0$ and i.i.d.\ increments. We use $\bP_\tau$ to denote the probability with respect to the renewal process and $\bE_\tau$ for the expectation. Without loss of generality, we assume that $\tau$ is persistent, i.e., $\bP_{\tau}(\tau_{1}<\infty)=1$. As noted in \cite{Gia11}, the transience of $\tau$ only results in a shift of certain coefficients, which is not essential. The inter-arrival law of $\tau$ is defined by
\begin{equation}\label{eq:pin_dist}
\bP_{\tau}(\tau_{1}=n)=\frac{L(n)}{n^{1+\al}},\quad\text{for}~n\geq1, 
\end{equation}where $\alpha>0$ and $L(\cdot): \bR^+\to\bR^+$ is a slowly varying function, i.e., $\lim_{t\to\infty}L(at)/L(t)=1$ for any $a>0$ (see \cite{BGT89} for more details about slowly varying functions). 

Let $\omega=\{\om_n\}_{n\in\bN_0}$ be a sequence of random variables independent of $\tau$ with probability $\bP_\omega$ and expectation $\bE_\omega$, representing the random environment (also called \emph{disorder}) and satisfying
\begin{equation}
\bE_\om[\omega_i]=0,~\mathrm{Var}(\om_i)=1,~\text{and}~\exists\beta_0>0,~\text{such that}~\bE_\om\big[e^{\beta\omega_i}\big]<\infty,~ \forall\beta\in(-\beta_0,\beta_0).
\end{equation}

The disordered pinning model with length $N$ and free boundary condition is defined via a Gibbs transformation by
\begin{equation}\label{def:Gibbs}
\frac{\dd\bP_{N,\beta}^{\omega,h}}{\dd\bP_\tau}(\tau):=\frac{1}{\zf}\exp\Big\{\sum_{n=1}^N(\beta \omega_{n}+h)\1_{\{n\in\tau\}}\Big\},
\end{equation}
where
\begin{equation}\label{e:partition-dis}
\zf=\bE_{\tau}\bigg[\exp\bigg\{\sum_{n=1}^N(\beta \omega_{n}+h)\1_{\{n\in\tau\}}\bigg\}\bigg].
\end{equation}
is the partition function,   $\beta>0$ is the inverse temperature, and $h\in\bR$ is an external field.

To see the connection between our SHE \eqref{e:she} and the disorder pinning model, note that by the Feynman--Kac formula, the mild Skorohod solution of \eqref{e:she} can be formally written as
\begin{equation}\label{mild solution}
u(t,x)=\fE_{(t,x)}\bigg[\exp\bigg\{\beta\int_0^t\delta_0(X_{r})\xi(t-r)\dd r-\frac{\beta^2}{2}\bE \bigg[\bigg(\int_0^t\delta_0(X_{r})\xi(t-r)\dd r\bigg)^2 \bigg]\bigg\}\bigg].
\end{equation}
 If  the disorder $\omega$  in \eqref{def:Gibbs} is Gaussian, then the  \emph{Wick-ordered} partition function is defined as
\begin{equation}\label{e:wick_partition}
\tilde{Z}_{N,\be}^{\om,h}:=\bE_{\tau}\bigg[\exp\bigg\{\sum_{n=1}^N (\beta\omega_n+h)\1_{\{n\in\tau\}}-\frac{\beta^2}{2}\bE_\omega\bigg[\bigg(\sum_{n=1}^N\omega_n\1_{\{n\in\tau\}}\bigg)^2\bigg]\bigg\}\bigg].
\end{equation}

Equation \eqref{mild solution} can be viewed as the continuum analogue of \eqref{e:wick_partition} with 
$h=0$, in the sense that, after a suitable scaling, the hitting times of $0$ for $X_t$ share the same inter-arrival distribution as the renewal process $\tau$. Indeed, consider a heavy-tailed random walk $S=\{S_i, i\in\mathbb{N}_0\}$ belonging to the domain of attraction of a $\rho$-stable law for some $\rho\in(0,2]$. By Lemma~\ref{lem:random walk}, the hitting times of $0$ for $S$ generate a renewal process whose inter-arrival distribution is given by \eqref{eq:pin_dist} when $\rho=\frac{1}{1-\alpha}\in(1,2]$ for $\alpha\in(0,\tfrac12]$. This is precisely the regime in which we establish the existence and uniqueness of an $L^1$-solution to the SHE \eqref{e:she}; see Theorem~\ref{thm-well-posedness}. In contrast, for $\rho\in(0,1]$, no renewal process of the form \eqref{eq:pin_dist} arises from a random walk in the domain of attraction of a $\rho$-stable law. Therefore, the SHE \eqref{e:she} can be viewed not only as a continuum analogue but also as a genuine extension of the disordered pinning model \eqref{e:wick_partition}.

As a disordered system, a fundamental question for the disordered pinning model is whether an arbitrarily small amount of disorder can significantly alter its behavior. The random environment is coupled to the inverse temperature $\beta$, which quantifies the strength of the disorder. If, for every $\beta>0$, the model exhibits behavior that is essentially different from that of the homogeneous case $\beta=0$, then the disorder is said to be \emph{disorder relevant}. In contrast, if for sufficiently small $\beta>0$ the model remains comparable to the homogeneous one, then the disorder is said to be \emph{disorder irrelevant}.

For the disordered pinning model in i.i.d.\ random environments, the celebrated Harris criterion~\cite{Harris74} asserts that the disorder is relevant when $\alpha>\tfrac12$ and irrelevant when $\alpha<\tfrac12$. In the marginal case $\alpha=\tfrac12$, the Harris criterion is inconclusive. Over the past two decades, this criterion--—including the marginal case---has been fully confirmed through a series of works, which we briefly review below.

Referring to \cite{Gia11}, the free energy of the system is defined by
\begin{equation}\label{def:free_energy}
F(\beta,h):=\lim\limits_{N\to\infty}\frac{1}{N}\log Z_{N,\beta}^{\omega,h}\xlongequal{\text{a.s and in}~L^1}\lim\limits_{N\to\infty}\frac{1}{N}\bE_\omega\Big[\log Z_{N,\beta}^{\omega,h}\Big],
\end{equation}
where the limit is known to exist. It can be shown that $F(\beta,h)$ is non-negative and non-decreasing in $h$. Let
\begin{equation}\label{def:h_c}
h_c(\beta):=\inf\{h\in\R:, F(\beta,h)>0\}=\sup\{h\in\R:, F(\beta,h)=0\}
\end{equation}
denote the \emph{critical point} of $F(\beta,h)$. Moreover, it can be shown that $h_c(\beta)\in(-\infty,+\infty)$.

For the homogeneous model, $F(0,h)\approx(h-h_c(0))^{\theta(\alpha)}$ with $\theta(\alpha)=1\vee\frac{1}{\alpha}$, which is called the \emph{critical exponent}. For $\beta>0$, the following picture is known.
\begin{itemize}
\item For $\alpha<\frac12$ and sufficiently small $\beta>0$, $F(\beta,h)\approx(h-h_c(\beta))^{\theta(\alpha)}$ (see \cite{Lac10}).
\item For any $\alpha>0$ and any $\beta>0$, $F(\beta,h)\leq C(h-h_c(\beta))^2$ (see \cite{GT06}).
\item For $\alpha=\frac12$, disorder relevance/irrelevance depends on the finer property of the slowly varying function $L(\cdot)$ in \eqref{eq:pin_dist}, and the disorder is relevant iff  $h_c(\beta)>h_c(0)$ for any $\beta>0$ (see \cite{Lac10,BL18}). 
\end{itemize}
The above results confirm the Harris criterion by exhibiting a change in the critical exponent for $\alpha>\tfrac12$ and a shift of the critical point in the marginal case $\alpha=\tfrac12$.

Although the disordered pinning model in i.i.d.\ random environments has been well studied, much less is known about the model in correlated environments. A natural choice for the correlation of the random environment is
\begin{equation}\label{def:cor_dis_env}
\gamma(n-m):=\bE_\omega[\omega_n\omega_m]=|n-m|^{-p}\wedge1,\quad\text{for}~n,m\in\bN_0,
\end{equation}
where $p>0$. There is a counterpart to the Harris criterion for correlated random environments, known as the Weinrib--Halperin prediction~\cite{WH83}. It conjectures that, for the disordered pinning model in the setting \eqref{def:cor_dis_env}, the disorder is relevant when $\alpha>\min(p,1)/2$ and irrelevant when $\alpha<\min(p,1)/2$. Note that when $p>1$, corresponding to sufficiently weak correlations, the Weinrib--Halperin prediction coincides with the Harris criterion.

The first study on the disordered pinning model in correlated random environments was carried out in \cite{Ber13}. It shows that if $p>1$, then for any $\alpha>0$ and any $\beta>0$, $h_c(\beta)\in(-\infty,+\infty)$ and $F(\beta,h)\leq C(h-h_c(\beta))^2$, which confirms disorder relevance for $\alpha>\frac12$, as in the case of i.i.d.\ random environments; however, for $\alpha<\frac12$, whether disorder is irrelevant still remains open. Furthermore, it has been shown in \cite{Ber13} that if $p\leq1$, then for any $\beta>0$ and any $h\in\R$,  $F(\beta,h)>0$. Hence, it is impossible to characterize disorder relevance/irrelevance by comparing the critical exponent with the homogeneous model. In the analysis in \cite{Ber13}, whether the correlation $\gamma(n)$ is summable plays a key role (i.e., whether $p>1$). By relaxing the specific assumption \eqref{def:cor_dis_env} and assuming only that $\sum_{n\ge 1}\gamma(n)<+\infty$, \cite{BP15} showed that if $\mathbf{E}[\tau_1]<+\infty$ (which requires $\alpha\ge 1$), then disorder is relevant in the sense of a shift of the critical point.

For the disordered pinning model in correlated random environments with $p<1$, the critical exponent of $F(\beta,h)$ does not exist by \cite{Ber13}, necessitating alternative approaches to study disorder relevance. A natural framework is the \emph{intermediate disorder regime}, which builds on the fact that when disorder is relevant, the environment influences the model even for arbitrarily small $\beta$. By tuning $\beta=\beta_N$ down to zero as $N\to\infty$, one can obtain a non-trivial random limit.  This approach was first developed in~\cite{AKQ14} for the $1+1$ directed polymer in i.i.d.\ environments  and subsequently extended in~\cite{CSZ16} to the models including the Ising model, the long-range directed polymer, and the pinning model, in i.i.d.\ random environments. Its applicability extends to a wide range of correlated environments, including polymer models with spatially colored \cite{rang20, cg23}, temporally colored \cite{rsw24}, or space–time colored environments \cite{SSSX21}. Inspired by these works, the disordered pinning model under assumption \eqref{def:cor_dis_env} with $p<1$ was recently analyzed in \cite{lswz}, as we discuss below.

We assume that $\omega$ is a Gaussian environment, as in \cite{Ber13,BP15}, since  the model is otherwise technically intractable. Let $p=2-2H$ with $H\in(\frac12,1)$, where $H$ is the Hurst parameter; this range ensures that $p\in(0,1)$. According to the Weinrib--Halperin prediction, disorder is expected to be relevant when $\alpha+H>1$. For a simple illustration here, we set $h=0$ in \eqref{e:wick_partition} and $L(\cdot)\equiv1$ in \eqref{eq:pin_dist}. By tuning the inverse temperature as $\beta_N=\hat{\beta}N^{-(\alpha+H-1)}$, we have established the following results in \cite{lswz} under the assumption $\alpha+H>1$.
\begin{itemize}
\item When $\alpha>\frac12$,  for any $\hat{\beta}>0$, the Wick-ordered partition function $\tilde{Z}_{N,\beta_N}^{\omega}$ in \eqref{e:wick_partition} converges in distribution to a non-trivial random variable $\tilde{\mathcal{Z}}_{\hat{\beta},\alpha}$ as $N\to\infty$, where $\tilde{\mathcal{Z}}_{\hat{\beta},\alpha}$ is an $L^2$-convergent series of multiple Skorohod integrals.
\item When $\alpha=\frac12$,  there exists  $\hat{\beta}_c>0$, such that for $\hat{\beta}\in(0,\hat{\beta}_c)$, $\tilde{Z}_{N,\beta_N}^{\omega}$ converges in distribution to $\tilde{\mathcal{Z}}_{\hat{\beta},\frac12}$ which is an $L^2$-convergent series only for $\hat{\beta}\in(0,\hat{\beta}_c)$.
\item When $\alpha<\frac12$, $\sup_{N}\bE_\omega[(\tilde{Z}_{N,\beta_N}^{\omega})^2]=+\infty$.
\end{itemize}

These results suggest that  an extra condition $\alpha\geq\frac12$ is necessary in order to have an $L^2$-integrable random variable as the weak limit of the rescaled partition function. In particular, the conditions $H>\frac12$ and $\alpha\geq\frac12$ automatically imply $\alpha+H>1$, which is the condition for disorder relevance predicted by Weinrib and Halperin. Interestingly, we showed that as long as $\alpha+H>1$,  each multiple Skorohod integral in the series expansion of $\tilde{\mathcal Z}_{\hat \beta, \alpha}$ is well-defined. However, we need an extra condition $\alpha\geq\frac12$ in order to have the $L^2$-convergence of the series.

This is a new phenomenon that, to the best of our knowledge, has not been observed previously. For the above-mentioned polymer models with random environments that are independent in time, the intermediate disorder regime—which is an $L^2$-regime—appears to coincide with the disorder-relevant regime. In contrast, for the disordered pinning model in a correlated random environment, the $L^2$-regime is a proper subset of the (predicted) full disorder-relevant regime (we will discuss this in more detail in Section~\ref{sec:spatial_noise}). On the other hand, the condition $\alpha+H>1$ is necessary in order to have $\beta_N\to0$, and it provides the correct scaling under which each multiple Skorohod integral is convergent. Consequently, there is currently no strong evidence contradicting the Weinrib--Halperin prediction. Addressing the case $\alpha+H>1$ with $\alpha<\tfrac12$ therefore requires the development of new strategies.

In this paper, we partially solve this problem under the assumption $\rho\in(\frac{1}{2H-1},2]$ (see Theorem~\ref{thm-well-posedness}), which is equivalent to $\alpha+2H>2$. We show that our SHE \eqref{e:she}, which is a continuum analogue to the disorder pinning model, has an $L^1$-solution. In fact, we have already proved in \cite{lswz} that $\{\tilde{Z}_{N,\beta_N}^{\omega}\}_N$ is uniformly integrable. However, in the absence of a martingale structure, it is difficult to establish $L^1$-convergence, which is more convenient to work with in the continuum setting of SHEs.

To conclude this subsection, we recall that in our previous work \cite{lswz} we also showed that, when $\alpha+H>\frac32$, tuning $\beta_N$ in the same manner as above yields convergence in distribution of the ordinary partition function $Z_{N,\beta_N}^\omega$ in \eqref{e:partition-dis} to a non-trivial limit $\mathcal{Z}_{\hat{\beta}}$. This limit can be represented as an $L^2$-convergent series of multiple Stratonovich integrals. However, such multiple Stratonovich integrals are not well defined when $\alpha+H<\frac32$. Since throughout this paper we assume $\alpha<\frac12$ and $H<1$, the Stratonovich regime is therefore excluded from our analysis.

\subsection{Discussion on some other models}\label{sec:spatial_noise} As discussed in Section~\ref{sec:pinning-model}, for the disordered pinning model, whether the disorder-relevant regime coincides with the $L^2$-regime depends on the independence of the underlying (temporal) random environment. More generally, in other disordered models, the relationship between the disorder-relevant regime and the $L^2$-regime is likewise governed by the temporal structure of the random environment; this issue will be discussed in detail in this subsection.

We consider two other classes of models in which the random environments involve spatial noise. The first class consists of  SHEs with space–time noise and its discrete analogue, the directed polymer model in random environments, both of which are closely related to the disordered pinning model. The second class is the parabolic Anderson model with purely spatial noise, as the approach developed in this paper is inspired by a recent breakthrough \cite{qrv} on the two-dimensional parabolic Anderson model. We note that a purely spatial noise can be viewed as a space–time noise with fully correlated temporal randomness.

The SHE with time-space noise is defined by
\begin{equation}\label{e:she_time_space}
\partial_t v(t,x)=-(-\Delta)^{\rho/2}u(t,x)+\beta v(t,x)\dot{W}(t,x),
\end{equation}
where $\dot W$ is a Gaussian noise with covariance 
\begin{equation}\label{e:time_space_cov}
\bE[\dot{W}(t,x)\dot{W}(s,y)]=\gamma_1(t-s)\gamma_2(x-y).
\end{equation}
By the Feynman--Kac formula, its mild Skorohod solution is given by
\begin{equation}
v(t,x)=\mathbf{E}_X^x\bigg[\exp\bigg\{\beta\int_0^t\dot{W}(t-s,X_s)\dd s-\frac{\beta^2}{2}\bE_W\bigg[\bigg|\int_0^t\dot{W}(t-s,X_s)\dd s\bigg|^2\bigg]\bigg\}\bigg].
\end{equation}
Correspondingly, the Wick-ordered partition function of the directed polymer model is defined by
\begin{equation}\label{e:dp_partition}
\mathbf{Z}_{n,\beta}^{\omega}:=\bE_{S}\bigg[\exp\bigg\{\beta\sum_{i=1}^n \omega_{i,S_i}-\frac{\beta^2}{2}\bE_\omega\bigg[\bigg(\sum_{i=1}^n \omega_{i,S_i} \bigg)^2\bigg]\bigg\}\bigg],
\end{equation}
where $S$ is a random walk on $\mathbb{Z}^d$ with expectation $\bE_S$, and $\omega:=(\omega_{i,x})_{i\in\bN,x\in\bZ}$ is a family of random variables representing random environments with the same covariance structure as \eqref{e:time_space_cov}, that is, $\bE_\omega[\omega_{i,x}\omega_{j,y}]=\gamma_1(n-m)\gamma_2(x-y)\wedge1$.

If $\dot W$ is white noise or if $\omega$ is a family of i.i.d.\ random variables, then both the SHE \eqref{e:she_time_space} and the directed polymer model \eqref{e:dp_partition} have been extensively studied (see \cite{Z24} for a recent survey). Moreover, existing results indicate that, in the i.i.d.\ setting, the disordered pinning model and the directed polymer model are closely comparable. More precisely, second-moment computations of their partition functions show that a key object for the disordered pinning model is the renewal process $\tilde{\tau}=\tau\cap\tau'$, where $\tau'$ is an independent copy of $\tau$, while for the directed polymer model the corresponding object is the set of collision times between a random walk $S$ and its independent copy $S'$, which induces a renewal process $\bar{\tau}$. As long as $\tilde{\tau}$ and $\bar{\tau}$ have the same distribution, the two models exhibit many common features. Interestingly, the $L^2$-regime identified through second-moment analysis coincides with the disorder-relevant regime predicted by the Harris criterion. Moreover, the intermediate disorder regime exists on the entire disorder-relevant regime (see \cite{AKQ14,CSZ16}), and this remains true even in the marginally relevant case (see \cite{CSZ17}).


Next, we show that the temporal structure of the random environment---regardless of whether the spatial noise is independent---governs the relationship between the $L^2$-regime and the disorder-relevant regime. We take the SHE  \eqref{e:she_time_space} as an example, for which it is convenient to apply a \emph{renormalization group transformation} to understand the disorder relevance/irrelevance. This transformation is analogous to a \emph{coarse-graining} procedure: we rescale time and space and examine the model on progressively larger scales (also referred to as \emph{zooming out}), then observe whether the noise amplifies or diminishes. Disorder is deemed relevant if the noise amplifies under this procedure, and irrelevant if it diminishes.

In our covariance setting \eqref{e:time_space_cov}, we assume that $\gamma_i(t) = |t|^{-\theta_i}$ for $i=1,2$. If $\gamma_i(\cdot)$ is a Dirac delta function, i.e., if the noise is white on the corresponding coordinate, then we interpret $\theta_i$ as the dimension $d$ (with time having dimension $1$). This interpretation is consistent with the scaling property of the Dirac delta function. Then for any $\e>0$, we have 
\begin{equation}
\bE[\dot{W}(\e^{-1}t, \e^{-a}x)\dot{W}(\e^{-1}s, \e^{-a}y)]=\e^{\theta_1+a\theta_2}|t-s|^{-\theta_1}|x-y|^{-\theta_2},
\end{equation}
which implies that $\dot{\widetilde{W}}(t,x):=\e^{-(\theta_1+\frac12a\theta_2)/2}\dot{W}(\e^{-1}t,\e^{-a}x)\xlongequal{d}\dot{W}(t,x)$ is another noise with the same covariance structure \eqref{e:time_space_cov}. Write $\tilde{v}(t,x):=v(\e^{-1}t,\e^{-a}x)$, and we have 
\begin{equation}
\e\partial_t\tilde{v}(t,x)=-\e^{a\rho}(-\Delta)^{\rho/2}\tilde{v}(t,x)+\beta\e^{(\theta_1+a\theta_2)/2}\tilde{v}(t,x)\dot{\widetilde{W}}(t,x).
\end{equation}
For a proper scaling, we should have that $a\rho=1$, i.e., $a=\frac1\rho$. Dividing both sides by $\e$, we have a scaling factor $\e^{(\theta_1+\theta_2/\rho-2)/2}$ in front of the noise term. As $\e\to0$, if $\theta_1+\theta_2/\rho>2$, then the noise term vanishes, and the disorder is irrelevant; if $\theta_1+\theta_2/\rho<2$, then the noise term explodes, and the disorder is relevant; the case $\theta_1+\theta_2/\rho=2$ is marginal.

The $1+1$ directed polymer model with random environments that are correlated in space but independent in time was first studied in \cite{rang20} and later extended to a more general framework in \cite{cg23}. The case where the random environments are correlated in time but independent in space was studied in \cite{rsw24}. In \cite{cg23}, the parameters are $\theta_1 = 1$ and $\theta_2 = 2r - 1$ with $r=\frac32-H\in (\frac12, 1)$ (recall that $H\in(\frac 12,1)$). According to the above renormalization group transformation, the condition for disorder relevance is $\frac12<r < \frac{1+\rho}{2}$. Exactly under this condition, the authors showed
that the partition function converges in distribution to an $L^2$-integrable random variable. In \cite{rsw24}, the parameters are $\theta_1 = 2-2H$ and $\theta_2 = 1$, with $H \in (\frac12,1]$, and the condition for disorder relevance is $H > \frac{1}{2\rho}$.
Under the slightly stronger assumptions $\rho \in (1,2]$ and $H \in (\frac12,1]$, which automatically imply $H > \frac{1}{2\rho}$, the authors showed that the rescaled partition functions converge in distribution to an $L^2$-integrable random variable. The additional assumption $\rho > 1$ is imposed to ensure $L^2$-integrability, and plays a role analogous to the condition $\alpha > \frac12$ in our previous work \cite{lswz}.

Note that the Hurst parameter $H=\tfrac12$ corresponds to white noise (in accordance with the Harris criterion), whereas $H<\tfrac12$ represents substantially rougher noise, whose covariance must be interpreted as a generalized function. The case $H<\tfrac12$ is considerably more involved (see, e.g., \cite[Section~5]{Nua06}), and the frameworks developed in \cite{rang20,cg23,rsw24,lswz} no longer apply. For this reason, we restrict our attention to the case $H\ge \tfrac12$. For the directed polymer model, when the random environment is independent in time, the $L^2$-regime coincides with the disorder-relevant regime, even if the environment exhibits spatial correlations, as shown in \cite{rang20,cg23}. In contrast, when the environment is correlated in time, a gap emerges between the $L^2$-regime and the (predicted) disorder-relevant regime, even in the presence of spatial independence; see \cite{rsw24}.

Now, we discuss the parabolic Anderson model,  defined by the SHE on $\R^d$:
\begin{equation}\label{e:PAM}
\partial_t v(t,x)=\Delta v(t,x)+\beta v(t,x)\dot{W}(x),
\end{equation}
where $\dot W(x)$ is a Gaussian noise satisfying
\begin{equation}
\bE[\dot{W}(x)\dot{W}(y)]=\delta(x-y).
\end{equation}
Here the noise is purely spatial and white. A similar renormalization group transformation yields
\begin{equation}
\partial_t \tilde{v}(t,x)=\Delta \tilde{v}(t,x)+\beta\e^{\frac{d}{4}-1}\tilde{v}(t,x)\dot{\widetilde{W}}(x),
\end{equation}
which shows that the critical dimension is 
$d=4$ and that the disorder is relevant if 
$d<4$. For $d=1$, it is well known that the solution $v(t,x)$ is well defined and $L^p$-integrable for any $p>0$ (see \cite{cm94}). For $d=2$, the existence of an $L^2$-solution was established, but only locally in time (see \cite{NZ89,hu02}). A global solution on the entire space $[0,+\infty)\times \R^d$ for
$d=2,3$ was constructed later in \cite{ML15,ML18} using the theory of regularity structures. However, the integrability properties of the solution were not specified. Very recently, it was shown in \cite{qrv} that for $d=2$, a global $L^1$-solution exists by extending the classical theory of the Skorohod integral to the $L^1$-regime.

The parabolic Anderson model is indeed driven by noise with strong temporal correlation. Owing to the absence of a temporal coordinate, the noise exhibits perfect temporal correlation, in the sense that the same spatial noise is present at all time. In this setting, we also observe that the $L^2$-regime for the parabolic Anderson model forms only a proper subset of the (predicted) disorder-relevant regime.

Finally, we turn to our SHE \eqref{e:she} and to the disordered pinning model introduced in Section~\ref{sec:pinning-model}. By applying the same renormalization group transformation, we obtain that for $\tilde u(t,x) = u(\e^{-1}t, \e^{-1/\rho}x)$, 
\begin{equation}
\partial_t\tilde{u}(t,x)=-(-\Delta)^{\rho/2}\tilde{u}(t,x)+\beta\e^{1/\rho-H}\tilde{u}(t,x)\delta(x)\tilde{\xi}(t),
\end{equation}
where $\tilde \xi(t):= \e^{H-1}\xi(\e^{-1}t)$
 is another temporal noise with the same distribution as $\xi(t)$. This scaling indicates that the condition for disorder relevance is $1/\rho-H<0\Longleftrightarrow \alpha+H>1$, which coincides with the Weinrib--Halperin prediction. Recall that we have shown in \cite{lswz} that the $L^2$
-regime exists if and only if $\alpha\ge \frac12$, and therefore the $L^2$-regime cannot be extended to the entire disorder-relevant regime.

In summary, the above investigation highlights that the presence or absence of temporal correlations in the noise plays a decisive role in determining whether the $L^2$-regime coincides with the disorder-relevant regime. For models driven by noise that is white in time, the entire disorder-relevant regime lies within the $L^2$-regime. By contrast, when the noise is colored in time, the situation becomes more intricate and interesting. As shown in \cite{qrv} and in the present work, a genuine $L^1$-regime may emerge in the disorder-relevant phase. Moreover, it remains an open question whether the $L^1$-regime can be extended to the whole disorder-relevant regime. In \cite{ML18}, a solution to the three-dimensional parabolic Anderson model was constructed, but its integrability properties were not specified. For the stochastic heat equation \eqref{e:she} studied in this paper, the regime
$$\Big\{(\alpha,H):\alpha<\frac12, \alpha+H>1,\alpha+2H<2\Big\}$$
remains open, and will be further discussed in Section~\ref{sec:con_dis}. Similarly, for the directed polymer model studied in \cite{rsw24}, where the random environment is independent in space but correlated in time, an analogous non-$L^2$ regime
$$\Big\{(\rho,H):\rho\in\Big[\frac12,1\Big),H\in\Big(\frac12,1\Big],H>\frac{1}{2\rho}\Big\}$$
also remains open.

\subsection{Contributions and discussions}\label{sec:con_dis} In this subsection, we summarize our main contributions of this paper and discuss related open questions.

Below, we outline the main contributions of this paper.
\begin{enumerate}
\item We broaden the scope of the disordered pinning model by investigating its continuum analogue, the SHE \eqref{e:she}. While the special case $\rho=2$ and $H=\frac12$ was previously studied in \cite{WY24}, the general continuum formulation has not yet been fully explored. For $\rho\in(1,2]$, the SHE \eqref{e:she} serves as the exact continuum counterpart of the disordered pinning model with $\alpha=1-\frac{1}{\rho}\in(0,\frac12]$. In contrast, for $\rho\in(0,1]$, there is no corresponding disordered pinning model with renewal process \eqref{eq:pin_dist} that matches the SHE \eqref{e:she}. We establish the well-posedness of the SHE by constructing a global $L^1$-solution for $\rho\in(\frac{1}{2H-1},2]$.

\item
To the best of our knowledge, this is the first construction of a global $L^1$-solution to an SHE in the absence of a corresponding local $L^2$-solution. Our work is inspired by the methods and results of \cite{qrv}, where a global $L^1$-solution for the two-dimensional parabolic Anderson model was obtained, with techniques that crucially rely on the existence of a local $L^2$-solution. Furthermore, we establish the positivity of our $L^1$-solution, which allows us to apply the Cole–Hopf transformation, that is, to take the logarithm of the solution, in order to study the free energy of the system.

\item 

Our study advances the understanding of the Weinrib--Halperin prediction. In earlier works, Quentin \cite{Ber13} showed that for random environments with sufficiently weak correlations ($p>1$ in \eqref{def:cor_dis_env}), the Weinrib--Halperin prediction coincides with the Harris criterion. However, the classical approach to the disordered pinning model breaks down when $p\le 1$, since the free energy is always strictly positive and, consequently, no critical exponent exists.

In our prior work \cite{lswz}, we extended the analysis to the case $p<1$. In particular, we identified the intermediate disorder regime that confirms disorder relevance under the conditions $\alpha>\frac12$ and $H\in(\frac12,1)$, which is stronger than the condition $\alpha+H>1$ proposed by the Weinrib--Halperin prediction. In the present paper, we further extend the disorder-relevant regime to
\begin{equation*}
\Big\{(\alpha,H): \alpha<\frac12,\ H\in\Big(\frac12,1\Big),\ \alpha+2H>2\Big\},
\end{equation*}
thereby partially closing the gap left by previous studies.
\end{enumerate}

To conclude the introduction, we present a list of open questions concerning the Weinrib--Halperin prediction in our SHE \eqref{e:she} and the disordered pinning model in correlated random environments.

\begin{enumerate}
\item  To fully confirm that the regime $\alpha+H>1$ is disorder relevant according to the Weinrib--Halperin prediction, it remains to show that in the regime
\begin{equation*}
\Big\{(\alpha,H):\alpha+H>1, \alpha+2H\le 2, \alpha<\frac12\Big\},
\end{equation*}
the model is indeed disorder relevant. However, our method relies on the finiteness of the  self-energy of the local time (see \eqref{e:I-t}) and fails in this case, because the self-energy is almost surely finite if and only if $\alpha+2H>2$ (see Lemma~\ref{lem:infty}). As a consequence, we cannot achieve the uniform integrability of the projected solutions $\{u_n(t,x)\}_{n\geq1}$ (see Section \ref{sec: L1-sol}).

\item The disorder-irrelevant regime $\alpha+H<1$ predicted by the Weinrib--Halperin criterion remains largely unexplored. Previous works \cite{Ber13,lswz} only established that the condition $\alpha>\frac12$ ensures disorder relevance. In contrast, concerning the disorder-irrelevant regime, neither the general case $\alpha+H<1$ nor the special case with the additional assumption $\alpha<\frac12$ has been investigated. For $p>1$ in \eqref{def:cor_dis_env}, it was shown in \cite{Ber13} that the critical point $h_c(\beta)$ (see \eqref{def:h_c}) exists finitely. Therefore, it is possible  to establish disorder irrelevance for $p>1$ and $\alpha<\frac12$ by showing that the free energy (see \eqref{def:free_energy}) has the asymptotics $F(\beta,h)\approx(h-h_c(\beta))^{1/\alpha}$ for small enough $\beta>0$, as shown in \cite{Lac10}, and thus the disordered model and the homogeneous model have the same critical exponent. For $p\in(0,1)$, i.e., $H\in(\frac12,1)$, since $F(\beta,h)$ is always positive and there is no critical exponent, a new perspective is required to study disorder irrelevance.

\item The case $\alpha+H=1$ corresponds to the marginal regime. The special case $\alpha=H=\frac12$ has been studied extensively in \cite{BL18,BL17,CSZ17,CSZ23,WY24}. Note that $H=\frac12$ corresponds to white noise. Consequently, all of the aforementioned works fall within the framework of the Harris criterion, and it is therefore not surprising that the entire disorder-relevant regime they considered also coincides with the $L^2$-regime. However, existing approaches developed for marginally relevant polymer models do not appear to be suitable when disorder is marginally relevant in models with correlated noise, since the $L^2$-regime already disappears when $\alpha+2H>2$  with $\alpha<\frac12$.
\end{enumerate}

\subsection{Organization of the paper}
The remainder of this paper is organized as follows.

In section \ref{sec:pre}, we provide some necessary preliminaries, including the classical Malliavin calculus and
the theory of the $L^1$-integrable Skorohod integral, which will be utilized throughout the paper. The proof of our main result, Theorem~\ref{thm-well-posedness}, occupies Sections~\ref{Sec: Lp} and~\ref{sec: L1-sol}. Specifically, Section \ref{Sec: Lp} handles the $L^p$-solutions ($p>1$) using the classical Malliavin calculus, while Section \ref{sec: L1-sol} extends the analysis to the $L^1$-solution using the $L^1$-Skorohod integral developed in~\cite{qrv}.
 Section \ref{sec:0-1} is devoted to the proof of Theorem \ref{thm:positivity}, concerning the strict positivity of the solution. Finally, technical lemmas are collected in Appendix~\ref{sec:appendix}.


\section{Preliminaries}\label{sec:pre}

In this section, we first review some preliminaries of the classical Malliavin calculus and refer readers to~\cite{Nua06,Nua18} for further details. We then recall the theory of the $L^1$-integrable Skorohod integral developed in~\cite{qrv}.

\subsection{Malliavin calculus} \label{sec:Malliavin}

Let $\xi$ be the Gaussian noise with the covariance function given by~\eqref{e:covariance}.
Denote by $\mathcal H$ the Hilbert space associated with $\xi$, which is the completion of the space of  smooth functions with compact support under the inner product
\begin{equation}\label{e:innerproduct}
\langle f, g\rangle_{\mathcal H} =\int_0^T\int_0^T f(s) g(t) |s-t|^{2H-2} dsdt.
\end{equation}
To see that $\langle\cdot,\cdot\rangle_{\cH}$ is indeed an inner product, we refer to \cite[Example 3.5.3]{Nua18}. The norm on $\cH$ is denoted by $\|\cdot\|_{\cH}$. Suppose $\{e_k(t), t\in[0,T]\}_{k\geq1}$ is an orthonormal basis of $\mathcal{H}$ consisting of  bounded continuous functions. Thus, for each $f\in \mathcal H$, we have $f=\sum_{k=1}^\infty \langle f, e_k\rangle_{\mathcal H} e_k$.

Let $\{\xi(f),f \in \cH \}$ be an isonormal Gaussian process with covariance  $\E[\xi(f) \xi(g)] = \langle f, g\rangle_\mathcal H.$  For each $f \in \mathcal H$, the random variable $\xi(f)$ is called the \emph{Wiener integral} of $f$, and is usually written in the form of a stochastic integral:
\[
\int_0^Tf(s)\xi(s)\dd s:= \xi(f). 
\]
For  $0\not\equiv f\in \cH$,  the $n$-th \emph{multiple Wiener integral} of $f^{\otimes n}$ is defined by  
\begin{equation}\label{Wiener}
    \I_n(f^{\otimes n}):= \|f\|_{\mathcal H}^n
H_n\big(\xi(f)/\|f\|_{\mathcal H}\big),
\end{equation} 
where 
\begin{equation*}
H_n(x):= (-1)^n\mathrm{e}^{x^2/ 2}\frac{\dd^n}{\dd x^n}\mathrm{e}^{-x^2/ 2}, \quad  x\in \bR,
\end{equation*}
is the $n$-th \emph{Hermite polynomial}. Let ${\cH}^{\hat \otimes n}$ be the subspace of ${\cH}^{\otimes n}$ containing only symmetric functions. For  $f\in {\cH}^{\hat\otimes n}$,  the $n$-th \emph{multiple Wiener integral} $\I_n(f)$ can be defined by \eqref{Wiener} via a limiting argument, noting that linear combinations of the functions of the form $f^{\otimes n}$ are dense in $\mathcal H^{\hat \otimes n}$. Moreover, 
$$
{\mathbb E}[\I_m(f)\I_n(g)]=n!\langle f,g\rangle_{\cH^{\otimes n}}\mathbf1_{\{m=n\}}, ~ {\rm for }~ f \in {\cH}^{\hat\otimes m}, g\in {\cH}^{\hat\otimes n}.
$$
For general $f\in {\cH}^{\otimes n}$, we set $\I_n(f):= \I_n(\hat f),$ where $\hat  f$ is the symmetrization of $f$ (see \eqref{e:fn}).  

Let $\mathcal{S}$ denote the class of smooth random variables of the form
\begin{equation}\label{smooth-rv}
F=f(\xi(h_1),\cdots ,\xi(h_n)), \, \text{ for } h_1,\dots,h_2\in\mathcal{H},
\end{equation} 
where $f:\bR^n\to\bR$ is a smooth function such that $f$ and all its partial derivatives have  at most polynomial growth. 

\begin{definition}[Malliavin derivative]\label{def-D}
Let $F$ be a smooth random variable of the form \eqref{smooth-rv}. The Malliavin derivative of $F$ is the $\mathcal{H}$-valued random variable given by
\begin{equation}\label{smooth-D}
DF=\sum_{i=1}^n\frac{\partial f}{\partial x_i}(\xi(h_1),\cdots ,\xi(h_n))h_i.
\end{equation} 
\end{definition}
Note that $D$ is a linear operator from $\cS\subset L^2(\Xi)$ into $L^2(\Xi;\cH)$. Let $\mathbb{D}^{1,p}$ with $p>1$ denote the closure of
$\mathcal{S}$ under the norm
$$\Vert F\Vert_{1,p}=\Big(\bE[|F|^p]+\bE[\Vert DF\Vert^p_{\mathcal{H}}]\Big)^{1/p}.$$

\begin{definition}[$L^2$-Skorohod integral/divergence]\label{def: Skorohod-l2}
Suppose $u\in L^2(\Xi;\mathcal{H})$, that is, $u$ is an $\mathcal{H}$-valued square-integrable random variable. We say that  $u$ is ($L^2$-)Skorohod integrable  if there exists an $L^2$-integrable random variable denoted by $\bd(u)$ such that for any $F\in \mathbb{D}^{1,2}$, we have
\begin{align}\label{e:L2-dual}
\mathbb{E}\left[\bd(u)F\right]=\mathbb{E}\left[\langle DF, u\rangle_{\mathcal{H}}\right] 
\end{align}
with $|\mathbb{E}\left[\bd(u)F\right]|\leq c\Vert F\Vert_{L^2(\bP)}$, where $c$ is a constant depending on u. We call $\bd(u)$ the $L^2$-Skorohod integral of $u$ or the divergence of $u$. 
\end{definition}
In particular, we have that $\bd(h)=\xi(h)$ is the Wiener integral for $h\in\mathcal H$. The set of $L^2$-Skorohod integrable processes is denoted by $\Dom(\bd)$.   For each $u\in \Dom(\bd)$, we also write its divergence as an integral: 
\[\int_0^T u(s) \xi(s) \dd s :=\bd(u).\]
  
\subsection{\texorpdfstring{$L^1$}{}-Skorohod integral}\label{sec:Malliavin-L1}
We now collect some essential elements of the theory of the 
$L^1$-Skorohod integral developed in~\cite{qrv}. Recall \eqref{e:xi-k}:
\begin{equation*}
\xi_k:=\xi(e_k)=\int_0^T e_k(s) \xi(s) \dd s,
\end{equation*}
where  $\{e_k\}_{k\in\mathbb N}$ is an orthonormal basis of $\mathcal H$. 
Then $\{\xi_k, k\in \mathbb N\}$ is a family of  i.i.d.\ standard Gaussian random variables.
Let $\mathcal{F}_n=\sigma(\xi_1,\dots,\xi_n)$ denote the $\sigma$-field generated by $\{\xi_1, \dots, \xi_n\}$. 
\begin{definition}[$L^1$-Skorohod integral for random sequences]\label{def-Skorohod}
Let $D=\{1,\cdots,\mathrm{max(D)}\}$ be finite or let $D=\bN$, and $U=\left(u_k\in L^1, k\in D\right)$ be a sequence of random variables measurable with respect to $\sigma(\xi_k:k\in D)$. We say that $U$ is $L^1$-Skorohod integrable if there exists a random variable $S \in L^1$
 such that for any $n\in D$ and any bounded differentiable $f:\mathbb{R}^n\to\mathbb{R}$ with bounded $\nabla f$, we have
$$ \bE\left[Sf(\xi_1,\cdots,\xi_n)\right]=\bE\left[U\cdot \nabla f(\xi_1,\cdots,\xi_n)\right],$$ then $S$ is called the $L^1$-Skorohod integral of $U$ and is denoted by 
\[S:=\sint U\xi.\]
\end{definition}

\begin{definition}[$L^1$-Skorohod integral for random processes]\label{def:Skorohod-integral}

Let $D=\{1,\cdots,\mathrm{max(D)}\}$ be finite or let $D=\bN$, and  $u$ be a random process on $[0,T]$ measurable with respect to $\sigma(\xi_k:k\in D)$. Denote
$$u_k:=\langle u, e_k\rangle_{\mathcal H}=\int_0^T\int_0^T e_k(t)u(s) |t-s|^{2H-2}\dd t\dd s.$$   
We say that $u$ is $L^1$-Skorohod integrable with Skorohod integral $S$,  if $u_k\in L^1$ for each $k\in D$  and $U=(u_k,k\in D)$ is $L^1$-Skorohod integrable with Skorohod integral $S$ in the sense of Definition~\ref{def-Skorohod}. We  denote the $L^1$-Skorohod integral $S$ for the random  process $u$ by 
\[S:=\sint u \,\xi.\]
\end{definition}

\begin{remark}
For $u\in\Dom(\bd)\subset L^2(\Xi;\mathcal{H})$, we have 
$$u=\sum_{k=1}^\infty \langle u,e_k\rangle_{\mathcal{H}}e_k=:\sum_{k=1}^\infty u_k e_k. $$
For any $F=f(\xi_1,\cdots,\xi_n)$ of the form \eqref{smooth-rv} with $\xi_k:=\xi(e_k)$,  we have $$DF=\sum_{k=1}^n \frac{\partial f}{\partial x_k}(\xi_1,\dots, \xi_n) e_k.$$
By \eqref{e:L2-dual}, we have, by denoting $U=(u_1,u_2,\cdots)$,
\begin{align*}
\bE[\bd(u)F]&=\bE\left[\langle u, DF\rangle_{\mathcal H}\right]=\E\left[U\cdot \nabla f(\xi_1, \dots, \xi_n)\right].
\end{align*}
This shows that for $u\in\Dom(\bd)$, the divergence $ \bd(u)=\int_0^T u(s) \xi(s) \dd s$, defined in Definition~\ref{def: Skorohod-l2}, coincides with the $L^1$-Skorohod integral $\ssint u\, \xi=\ssint U\xi$ defined in Definition~\ref{def:Skorohod-integral}.    
In particular, for a deterministic process $u\in \mathcal H$, we have 
\begin{equation}\label{e:delta-u-determin}
\bd(u)=\sint u\xi =\sum_{k=1}^\infty u_k \xi_k. 
\end{equation}
\end{remark}

\begin{definition}[Projection]\label{def:projection}  Let $\mathcal{P}_n U:=\left(U_1,\cdots,U_n,0,0,\cdots\right)$, where $\mathcal{P}_n$ is the projection operator that sets all coordinates in $U$ beyond $n$ to $0$. Then the $n$-dimensional projection of the $L^1$-Skorohod intrgral $\ssint U\xi$ is defined by
$$\sintn U \xi: =\sint \left(\mathcal{P}_n U\right)\xi.$$
\end{definition}





The results below are borrowed from Proposition 13 and  Corollary 14 in
\cite{qrv}, respectively.

\begin{proposition}\label{prop-conditional expectation}

If $U$ is $L^1$-Skorohod integrable, then $\bE\left[\mathcal{P}_n U\big|\mathcal{F}_n\right]$ is also $L^1$-Skorohod integrable, and 
$$\bE\left[\sint U\xi \Big|\mathcal{F}_n\right]=\sint \bE\left[\mathcal{P}_n U\big|\mathcal{F}_n\right]\xi =\sintn\bE\left[U|\mathcal{F}_n\right]\xi.$$
\end{proposition}

\begin{proposition}\label{uniformly-Skorohod}

Let $U=(u_1, u_2,\dots)$ be an $\bR^{\bN}$-valued random vector such that $u_k\in L^1$ for all $k\in\bN$. Denote $H_n:=\bE[\mathcal{P}_n U|\mathcal{F}_n]$. Then the following two conditions are equivalent:
\begin{enumerate}
\item For every $n$, $S_n:=\ssint H_n\xi$ exists and  $\{S_n\}_{n\in\mathbb N}$ is  uniformly integrable.

\item $U$ is $L^1$-Skorohod integrable.

\end{enumerate}
If any of these conditions holds, then 
$S_n$ is a martingale that converges almost surely and in $L^1$ to $\ssint U\dd \xi$.
\end{proposition}

\section{On \texorpdfstring{$L^p$}{}-solutions with \texorpdfstring{$p>1$}{}}\label{Sec: Lp}

In this section, we apply the Malliavin calculus \cite{Nua06} and follow the approach in \cite{hu02,S17} to determine whether equation \eqref{e:she} admits an $L^p$-solution with $p>1$. We show that,  for any $p>1$, there is no $L^p$-solution when $\rho\in(0,2)$  if the solution exists uniquely (Propositions~\ref{prop:no-L^2} and \ref{prop:p-2-ineq} and Remark \ref{rem:p-2}), whereas for $\rho=2$ the equation admits a local $L^p$-solution but no global $L^p$-solution (Propositions~\ref{prop:local-L^2}, \ref{prop:p-2-ineq} and Remark \ref{rem:p-2}).

As mentioned in Section \ref{sec:intro}, we assume $\beta=1$ in \eqref{e:she}. Recall the Wiener chaos expansion of $u$ defined in \eqref{e:sol-chaos} 
\begin{equation*}
u(t,x)=1+\sum_{n=1}^{\infty}\mathbf{I}_n(f_n(t,x;s_1,\cdots,s_n))
\end{equation*}
and its second moment 
\begin{equation}\label{e:2nd-moment}
\E[u(t,x)^2]=1+\sum_{n=1}^\infty n!\Vert f_n(t,x)\Vert_{\mathcal{H}^{\otimes n}}^2,
\end{equation}
where $f_n(t,x)$ is given in \eqref{e:fn}. 
\begin{proposition}\label{prop:no-L^2}
When $\rho\in(0, 2)$, there is no $L^2$-solution for equation \eqref{e:she}.
\end{proposition}
\begin{proof} In view of \eqref{e:2nd-moment}, it suffices to show $\sum_{n=1}^\infty n!\Vert f_n(t,x)\Vert_{\mathcal{H}^{\otimes n}}^2=\infty$ when $\rho\in(0,2)$.  Note that
\begin{equation}\label{eq:lower bound}
\begin{aligned}
n!\Vert f_n(t,x)\Vert_{\mathcal{H}^{\otimes n}}^2&=n!\int_{[0,t]^n}\int_{[0,t]^n}f_n(t,x;s_1,\cdots,s_n)f_n(t,x;s_1',\cdots,s_n')\prod_{i=1}^n|s_i-s_i'|^{2H-2}\dd \boldsymbol{s}\dd \boldsymbol{s}'\\
&\geq t^{(2H-2)n}n!\left(\int_{\boldsymbol{s}\in [0,t]^n_<}g_\rho(t-s_n,x)\prod_{i=1}^{n-1}g_\rho(s_{i+1}-s_i,0)\dd \boldsymbol{s}\right)^2.
\end{aligned}   
\end{equation}

By  Lemma~\ref{lem:g-bound}, we have that
\begin{equation}\label{I_n-moment-ge}
\begin{aligned}
n!\Vert f_n(t,x)\Vert_{\mathcal{H}^{\otimes n}}^2&\geq t^{(2H-2)n}n!C^{2n} \left(\int_{s\in [0,t]^n_<}(s_2-s_1)^{-\frac{1}{\rho }}\cdots(s_n-s_{n-1})^{-\frac{1}{\rho }}(t-s_n)\dd \boldsymbol{s}\right)^2
\\&=C^{2n} n!\frac{\Gamma(1-\frac{1}{\rho})^{2(n-1)}\Gamma(2)^2}{\Gamma(n(1-\frac{1}{\rho})+\frac{1}{\rho}+2)^2}t^{2n(H-\frac{1}{\rho})+\frac{2}{\rho}+2}
\\&\geq C^{2n}~n^{-4}n^{n(\frac{2}{\rho}-1)}t^{2n(H-\frac{1}{\rho})+\frac{2}{\rho}+2},
\end{aligned}   
\end{equation}
where the equality follows from Lemma \ref{lem:a direct calculate} and the last inequality follows from Stirling's approximation. Thus, for $\rho\in(0,2)$,
$\sum_{n=1}^{\infty}n!\|f_n\|^2_{\mathcal H^{\otimes n}}=\infty$, and this proves the desired result.
\end{proof}


\begin{proposition}\label{prop:local-L^2}

When $\rho=2$ (equivalently, $\alpha=\tfrac12$), equation \eqref{e:she} admits a local $L^2$-solution, that is, there exists a constant $t_c\in(0,\infty)$ such that an $L^2$-solution exists on the time interval $[0,t_c)$, while $\sum_{n=1}^{\infty}n!\|f_n(t,x)\|^2_{\mathcal H^{\otimes n}}=\infty$ for all $t>t_c$.
\end{proposition}
\begin{proof} 
We first prove the non-existence of global $L^2$-solutions when $\rho=2$. By the same argument as in the proof of Proposition~\ref{prop:no-L^2}, we can get the same estimate  \eqref{I_n-moment-ge}. Thus, noting $H>\frac12$, we have that for sufficiently large $t$,
\begin{equation}\label{2-moment-diverge}
\sum_{n=1}^{\infty}n!\|f_n(t,x)\|^2_{\mathcal H^{\otimes n}}\geq\sum_{n=1}^\infty C^{2n} n^{-4}t^{2n(H-\frac{1}{2})+3}=\infty.
\end{equation}

It remains to show that a $L^2$-solution exists locally. By sequentially applying Lemmas \ref{lem:1/H norm}, \ref{lem:a direct calculate} and \ref{lem:g-bound},
we obtain that
\begin{equation}\label{e:2-moment-leq}
\begin{aligned}
&\quad n!\Vert f_n(t,x)\Vert_{\mathcal{H}^{\otimes n}}^2\leq C^{2n}\Vert f_n\Vert^2_{L^{1/H}(\bR^n)} 
\\&\leq C^{2n} (n!)^{2H-1}\left(\int_{s\in [0,t]^n_<}(s_2-s_1)^{-\frac{1}{2H }}\cdots(s_n-s_{n-1})^{-\frac{1}{2H}}(t-s_n)^{-\frac{1}{2H}}\dd s\right)^{2H}
\\&=C^{2n} (n!)^{2H-1}\frac{\Gamma(1-\frac{1}{2 H})^{2nH}}{\Gamma(n(1-\frac{1}{2 H})+1)^{2H}}t^{2nH(1-\frac{1}{2H})}.
\end{aligned}   
\end{equation}

Note that by Stirling's approximation, for large enough $n$, we have that
$$\frac{(n!)^{2H-1}}{\Gamma(n(1-\frac{1}{2 H})+1)^{2H}}\leq C^n{\sqrt{n}},$$
which implies
\begin{equation}
n!\Vert f_n(t,x)\Vert_{\mathcal{H}^{\otimes n}}^2\leq C^{n}{\sqrt{n}}t^{n(2H-1)}. 
\end{equation}
Therefore, for sufficiently small $t>0$,
\begin{align}\label{2-moment-converge}
\bE[u(t,x)^2]=1+\sum_{n=1}^{\infty}n!\Vert f_n(t,x)\Vert_{\mathcal{H}^{\otimes n}}^2<\infty.
\end{align}

Finally, note that $\|f_n(t,x)\|^2_{\cH^{\otimes n}}$ is increasing in $t$  since  $f_n$ is nonnegative, and hence $\bE[u(t,x)^2]$ is increasing in $t$. Set $t_c:=\sup\{t>0: \bE[u(t,x)^2]<\infty\}$ and the proof is completed.
\end{proof}

\begin{remark}
From the arguments above, one readily sees that, for general $\beta>0$ in \eqref{e:she}, the only effect is to introduce an additional factor $\beta^{2n}$ in the second moment of the $n$-th order chaos. This factor has no effect in the regime $\rho\in(0,2)$, while in the critical case $\rho=2$ it merely induces a quantitative shift of the critical point $t_c$.
\end{remark}

To show that an $L^p$-solution does not exist for any $\rho\in(0,2)$ and $p>1$ when the solution is unique, we show that the existence of an $L^p$-solution implies the existence of an $L^2$-solution for a possibly different $\beta$ in \eqref{e:she}. Let $u_\lambda$ be the mild solution of \eqref{e:she}  with $\beta=\sqrt\lambda$. The following result is inspired by \cite{hln17} (see also   \cite[Lemma~2]{le16}). 

\begin{proposition}\label{prop:p-2-ineq}
Assume the solution to \eqref{e:she} is unique.  For every $1<p\leq q$, we have
\begin{align*}
\left\Vert u_{\frac{p-1}{q-1}\lambda}(t,x)\right\Vert_{L^q(\bP)}\leq\left\Vert u_{\lambda}(t,x)\right\Vert_{L^p(\bP)}.
\end{align*}
\end{proposition}
\begin{proof} 
Let $\{T_\tau\}_{\tau\geq0}$ denote the Ornstein–Uhlenbeck semigroup in the Gaussian space associated with the noise $\xi$ (see \cite[Section 1.4.1]{Nua06}).  For a bounded measurable function $f$ on $\R^{\cH}$

we have that by Mehler’s formula (see \cite[Equation (1.67)]{Nua06}),
\begin{equation}\label{e:Mehler}
T_\tau f(\xi) = \E'\Big[f\big(e^{-\tau}\xi(s)+\sqrt{1-e^{-2\tau}}\xi'(s)\big)\Big],
\end{equation}
where $\xi'$ is an independent copy of $\xi$ with probability $\bP'$ and expectation $\bE'$ respectively.  Note that $e^{-\tau}\xi(s)+\sqrt{1-e^{-2\tau}}\xi'(s)$ is identically distributed as $\xi$ under the joint probability $\bP\times\bP'$.

For each $\tau\geq0$, let $u_{\tau,\lambda}$ be the mild solution of \eqref{e:she} with $\beta=\sqrt{\lambda}$ driven by the noise $e^{-\tau}\xi(s)+\sqrt{1-e^{-2\tau}}\xi'(s)$, that is,
\begin{align*}
u_{\tau,\lambda}(t,x)&=1+\sqrt{\lambda}\int_0^t g_{\rho}(t-s,x)u_{\tau,\lambda}(s,0)\big(e^{-\tau}\xi(s)+\sqrt{1-e^{-2\tau}}\xi'(s)\big)\dd s. 
\end{align*} 
By Mehler’s formula \eqref{e:Mehler}, we have that $T_\tau u_{\lambda}(t,x)=\bE'[u_{\tau,\lambda}(t,x)]$. Then,
\begin{align*}
T_\tau u_{\lambda}(t,x)&=1+\sqrt{\lambda}\bE'\left[\int_0^t g_{\rho}(t-s,x)u_{\tau,\lambda}(s,0)\big(e^{-\tau}\xi(s)+\sqrt{1-e^{-2\tau}}\xi'(s)\big)\dd s\right]
\\&=1+\sqrt{\lambda}\int_0^t g_{\rho}(t-s,x)\bE'[u_{\tau,\lambda}(s,0)]e^{-\tau}\xi(s)\dd s
\\&=1+\sqrt{\lambda}e^{-\tau}\int_0^t g_{\rho}(t-s,x)T_\tau u_{\lambda}(s,0)\xi(s)\dd s.
\end{align*}
Thus, $T_\tau u_{\lambda}$ is a solution of \eqref{e:she} with $\beta=\sqrt{\lambda}e^{-\tau}$, and hence $T_\tau u_{\lambda}=u_{e^{-2\tau}\lambda}$ by the assumption that the solution is unique.  By the hypercontractivity of the Ornstein-Uhlenbeck semigroup (see \cite[Theorem 1.4.1]{Nua06}), we have 
\begin{align*}
\Vert u_{e^{-2\tau}\lambda}(t,x)\Vert_{L^{q(\tau)}(\bP)}=\Vert T_\tau u_{\lambda}(t,x)\Vert_{L^{q(\tau)}(\bP)}\leq\Vert u_{\lambda}(t,x)\Vert_{L^p(\bP)}
\end{align*}
for all $1<p<\infty$ and $\tau\geq0$ with $q(\tau)=1+e^{2\tau}(p-1)$. Then the desired result follows by choosing $e^{-2\tau}=\frac{p-1}{q-1}$.
\end{proof}
\begin{remark}\label{rem:p-2}
For $\rho=2$ and $\beta=1$, Proposition \ref{prop:local-L^2} shows that an $L^2$-solution exists for $t\in[0,t_c)$. Proposition \ref{prop:p-2-ineq} then allows us to establish the $L^r$-integrability of the solution for any $r>1$. If $r\in(1,2)$, then clearly $\|u(t,x)\|_{L^r(\bP)} \le \|u(t,x)\|_{L^2(\bP)}<\infty$ for $t\in[0,t_c)$ by Jensen's inequality; on the other hand, by applying Proposition~\ref{prop:p-2-ineq} with $p=r, q=2$ and $\lambda=1$, we have that $\Vert u(t,x)\Vert_{L^r(\bP)}\geq\Vert u_{r-1}(t,x)\Vert_{L^2(\bP)}$, which implies that $\|u(t,x)\|_{L^r(\bP)}=\infty$ for sufficiently large $t>0$. Similarly, if $r> 2$, then $\|u(t,x)\|_{L^r(\bP)} \ge \|u(t,x)\|_{L^2(\bP)}=\infty$ for $t>t_c$, and by applying Proposition~\ref{prop:p-2-ineq} with $p=2, q=r$ and $\lambda=r-1$, we have $\Vert u(t,x)\Vert_{L^r(\bP)}\le \Vert u_{r-1}(t,x)\Vert_{L^2(\bP)}<+\infty$ for $t>0$ small enough.

For $\rho\in(0,2)$, the argument above also applies. If an $L^p$-solution exists for some $p>1$, then there must be a local $L^2$-solution, which is a contradiction to Proposition \ref{prop:no-L^2}.
\end{remark}




\section{On \texorpdfstring{$L^1$}{}-Skorohod solution}\label{sec: L1-sol}

In this section, we  investigate the $L^1$-Skorohod solution to the SHE \eqref{e:she}. Throughout this section, we assume $\rho\in(1,2]$. 

\subsection{Local time of \texorpdfstring{$\rho$}{}-stable processes}\label{sec:local}

In preparation for the subsequent analysis, in this subsection we collect some preliminaries on the local time of the $\rho$-stable process $X$ for $\rho\in(1,2]$.

Let $L^x_s$ denote the local time of $X$ at level $x$ up to time $s$,  which is defined by 
\begin{equation}\label{e:local-time}
L_s^x:=\lim_{\e\to0^+}\int_0^s\delta_\e(X_r-x)\dd r=:\lim_{\e\to0^+} L_{s,\e}^x,
\end{equation}
where $\delta_\e$ is defined in \eqref{e:delta-e}. By \cite[Proposition 2 in Chapter V]{bertoin96}, the convergence is uniform in $t\in[0,T]$ in $L^2(\fP)$ if $\rho\in(1,2]$, i.e., 
\begin{equation}\label{e:u-c-L}
\lim_{\e\to0}\fE\bigg[\sup_{s\in[0,T]}\big|L_{s,\e}^x-L_s^x\big|^2\bigg]=0,
\end{equation}
and thus $L_s^x$ is $\fP$-a.s.\ a continuous  non-decreasing function in $t$.  In the following, we write $L_s=L^0_s$ when $x=0$ to simplify the notation.

For each $x\in \R$, 
\begin{equation}\label{e:nu-L}
\nu^x_L(A):=\int_0^T\mathbf{1}_{A}(s)\dd L^x_s,~ A\in \mathcal B([0,T])
\end{equation}
defines a  random measure on the Borel set $\mathcal{B}([0,T])$.

Denote 
\begin{equation}\label{e:nu_e}
\nu^x_{L,\e}(A):=\int_0^T\mathbf{1}_{A}(s)\dd L_{s,\e}^x,   ~ A\in \mathcal B([0,T]).
\end{equation}
We have the following weak convergence for $\nu_{L,\e}^x$. 
\begin{lemma}\label{lem:L-measure} For each $x\in \R$, as $\e\to0$, the sequence of random measures $\nu_{L,\e}^x$ converges weakly to $\nu_L^x$ in $L^2(\fP)$, i.e., 
for any $f\in C([0,T])$, 
\[\lim_{\e\to0}\fE \bigg|\int_0^T f(s)\dd \nu^x_{L,\e}-\int_0^T f(s)\dd \nu^x_L\bigg|^2=0.\]
\end{lemma}
\begin{proof}
Without loss of generality, we assume $x=0$ and omit the superscript $x$. We show that for any $\eta>0$, there exists $\e_0=\e_0(\eta)$, such that for any $\e\in(0,\e_0)$,
\begin{equation*}
\bigg\|\int_0^T f(s)\dd \nu_{L,\e}-\int_0^T f(s)\dd \nu_L\bigg\|_{L^2(\fP)}<\eta.
\end{equation*}

Given $f\in C([0,T])$, let $\{f_n\}_{n\in\bN}$ be a sequence of continuously differentiable functions, such that $\lim_{n\to\infty}\|f_n-f\|_\infty=0$ on $C([0,T])$. We have
\begin{align*}
\bigg\|\int_0^T f(s)\dd \nu_{L}-\int_0^T f_n(s)\dd \nu_{L}\bigg\|_{L^2(\fP)}\leq\sup_{s\in[0,T]}\big|f(s)-f_n(s)\big|\big\|\nu_L([0,T])\big\|_{L^2(\fP)},
\end{align*}
and similarly
\begin{align*}
\bigg\|\int_0^T f(s)\dd \nu_{L,\e}-\int_0^T f_n(s)\dd \nu_{L,\e} \bigg\|_{L^2(\fP)}\leq\sup_{s\in[0,t]}\big|f(s)-f_n(s)\big|\big\|\nu_{L,\e}([0,T])\big\|_{L^2(\fP)}.
\end{align*}
Note that $\sup_{\e>0}\|\nu_{L,\e}([0,T])\|_{L^2(\fP)}<\infty$ by \eqref{e:u-c-L}. For given $\eta>0$, there exists $n_0=n_0(\eta)\in\bN$, such that for $n\geq n_0$ and any $\e>0$,
\begin{align}\label{e:triangle_1}
\bigg\|\int_0^T f(s)\dd \nu_{L}-\int_0^T f_n(s)\dd \nu_{L}\bigg\|_{L^2(\fP{})}\leq\frac{\eta}{3}\quad\text{and}\quad\bigg\|\int_0^T f(s)\dd \nu_{L,\e}-\int_0^T f_n(s)\dd \nu_{L,\e}\bigg\|_{L^2(\fP)}\leq\frac{\eta}{3}.
\end{align}

Then, by integration by parts, we have
\begin{equation*}
\begin{aligned}
\bigg\|\int_0^T f_{n_0}(s)\dd \nu_L-\int_0^T f_{n_0}(s)\dd \nu_{L,\e}\bigg\|_{L^2(\fP)}&= \bigg\|f_{n_0}(T)(L_T-L_{T,\e})-\int_0^T  (L_s-L_{s,\e})f_{n_0}'(s)\dd s\bigg\|_{L^2(\fP)}
\\
&\leq \bigg(\big|f_{n_0}(T)\big|+\int_0^T \big|f_{n_0}'(s)\big|\dd s\bigg)\bigg\|\sup_{s\in[0,T]}\big|L_s-L_{s,\e}(s)\big|\bigg\|_{L^2(\fP)}.
\end{aligned}
\end{equation*}
By \eqref{e:u-c-L}, there exists $\e_0=\e_0(\eta, n_0)$ such that for all $\e\in(0,\e_0)$, 
\begin{equation}\label{e:smooth-con}
\bigg\|\int_0^t f_{n_0}(s)\dd \nu_L-\int_0^t f_{n_0}(s)\dd \nu_{L,\e}\bigg\|_{L^2(\fP)}\leq\frac{\eta}{3}.
\end{equation}

Combining \eqref{e:triangle_1} and \eqref{e:smooth-con}, by the triangle inequality we have that for all $\e\in(0,\e_0)$,  
\begin{equation*}
\begin{split}
\bigg\|\int_0^T f(s)\dd \nu_{L,\e}-\int_0^T f(s)\dd \nu_L\bigg\|_{L^2(\fP)}\leq&\bigg\|\int_0^T f(s)\dd \nu_L-\int_0^T f_{n_0}(s)\dd \nu_L\bigg\|_{L^2(\fP)}\\
&+\bigg\|\int_0^T f(s)\dd \nu_{L,\e}-\int_0^T f_{n_0}(s)\dd \nu_{L,\e} \bigg\|_{L^2(\fP)}\\
&+\bigg\|\int_0^T f_{n_0}(s)\dd \nu_L-\int_0^T f_{n_0}(s)\dd \nu_{L,\e}\bigg\|_{L^2(\fP)}\leq \eta.
\end{split}
\end{equation*}
\end{proof}

For each $t\in[0,T]$, note that $\fP_{(t,x)}\big((L_s)_{0\leq s\leq t}\in\cdot\big)=\fP\big((L^x_{t-s})_{0\leq s\leq t}\in\cdot\big)$ by a time reversal.
The following lemma provides formulas for computing the $\fE_{(t,x)}$-expectation of integrals with respect to $\dd L_s$. 

\begin{lemma}\label{lem:L-property}
Let $F_s$ be an adapted process  and $f(s)$ be a deterministic measurable  function. Then
\begin{align}
&\quad\mathbf{E}_{(t,x)}\left[\int_0^t F_s\dd L_s\right]=\int_0^t \mathbf{E}_{(s,0)}\left[F_s\right]g_\rho(t-s,x)\dd \mathbf{s},\label{1}
\\&\quad\mathbf{E}_{(t,x)}\left[\int_0^t f(s)\dd L_s\right]= \int_0^t f(s)g_\rho(t-s,x)\dd \mathbf{s},\label{2}
\end{align}
provided the right-hand sides exist.  For a deterministic measurable symmetric function $f(s_1,\cdots,s_n)$, we have 
 \begin{equation}\label{3}
\begin{aligned}
&\quad\mathbf{E}_{(t,x)}\left[\int_{[0,t]^n}f(s_1,\cdots
,s_n)\dd L_{s_1}\cdots\dd L_{s_n}\right]\\
=&n!\int_{[0,t]_<^n}f(s_1,\cdots
,s_n)g_\rho(s_2-s_1,0)\cdots g_\rho(t-s_n,x)\dd \mathbf{s},
\end{aligned}
\end{equation}
provided the right-hand side exists.
\end{lemma}
\begin{proof}
We prove only \eqref{1} and the proofs for \eqref{2} and \eqref{3} are similar. For a positive bounded adapted process $F_s$, we have  
\begin{align*}
\mathbf{E}_{(t,x)}\left[\int_0^t F_s\dd L_s\right]&= \mathbf{E}_{(t,x)}\left[\lim_{\e\to 0^+}\int_0^t F_s\delta_{\e}(X_s)\dd s\right]=\lim_{\e\to 0^+}\mathbf{E}_{(t,x)}\left[\int_0^t F_s\delta_{\e}(X_s)\dd s\right]
\\
&=\lim_{\e\to 0^+}\int_0^t \mathbf{E}_{(t,x)}\left[F_s\delta_{\e}(X_s)\right]\dd s=\lim_{\e\to 0^+}\int_0^t \int_\bR\mathbf{E}_{(s,y)}\left[F_s\right]\delta_{\e}(y)g_\rho(t-s,x-y)\dd y\dd s
\\
&=\int_0^t \lim_{\e\to 0^+}\int_\bR\mathbf{E}_{(s,y)}\left[F_s\right]\delta_{\e}(y)g_\rho(t-s,x-y)\dd y\dd s=\int_0^t \mathbf{E}_{(s,0)}\left[F_s\right]g(t-s,x)\dd s,
\end{align*}
where the second equality is due to the dominated convergence theorem, with the help of Lemma~\ref{lem:g-bound}, the third equality is due to Fubini's theorem, the fourth equality is due to the Markov property of the stable process and the adaptedness of $F_s$, and the last two equalities are also due to the dominated convergence theorem. This proves \eqref{1} for positive bounded $F_s$. The result extends to general positive adapted processes by the monotone convergence theorem, and the general case then follows from the decomposition $F=F^+-F^-$. 
\end{proof}

Recall that we have assumed $H\in(\frac12,1)$ throughout the paper. We have the following corollaries for the \emph{mutual energy} \eqref{e:al*} and the \emph{self-energy} \eqref{def:self_energy}, which will be used in Sections \ref{sec:Mut-e}--\ref{sec:al+H>1}.
\begin{corollary}\label{cor:XX}
We have that for $\rho\in(1,2]$,
$$\mathbf{E}_{(t,x)}\left[\int_0^t\int_0^t|r-s|^{2H-2}\dd L_r\dd L_s\right]<\infty,\quad\text{iff}~-\frac{1}{\rho}+2H>1.$$ 
\end{corollary}
\begin{proof} By \eqref{3} and $g_\rho(r-s,0)=c(r-s)^{-1/\rho}$, we get
\begin{align*}
\mathbf{E}_{(t,x)}\left[\int_0^t\int_0^t|r-s|^{2H-2}\dd L_r\dd L_s\right]=2 c\int_0^t\dd r\int_0^r(r-s)^{-\frac{1}{\rho}+2H-2}g_\rho (t-r,x)\dd s,
\end{align*}
where the inner integral is finite if and only if $-\frac{1}{\rho}+2H-2>-1\Longleftrightarrow-\frac{1}{\rho}+2H>1$, and the outer integral is finite since $\rho>1$ by Lemma \ref{lem:g-bound}. 
\end{proof}
\begin{corollary}\label{cor:XX'} Let $L'$ be an independent copy of $L$. We have that for $\rho\in(1,2]$ $$\mathbf{E}_{(t,0)}^{\otimes2}\left[\int_0^t\int_0^t|r-s|^{2H-2}\dd L_r\dd L'_s\right]<\infty,\quad\text{iff}~H>\frac{1}{\rho},$$ 
and for $x\neq 0$,
$$\mathbf{E}_{(t,x)}^{\otimes2}\left[\int_0^t\int_0^t|r-s|^{2H-2}\dd L_r\dd L'_s\right]<\infty,\quad\text{ for all }\, H\in\Big(\frac12,1\Big).$$
\end{corollary}
\begin{proof}
As in the proof of Corollary \ref{cor:XX}, when $x=0$, we have 
\begin{equation*}
    \begin{aligned}
\mathbf{E}_{(t,0)}^{\otimes2}\left[\int_0^t\int_0^t|r-s|^{2H-2}\dd L_r\dd L'_s\right]=&c\int_0^t\int_0^t(t-r)^{-\frac{1}{\rho}}(t-s)^{-\frac{1}{\rho}}|r-s|^{2H-2}\dd r\dd s\\
=& c\int_0^t\int_0^t r^{-\frac{1}{\rho}}s^{-\frac{1}{\rho}}|r-s|^{2H-2}\dd r\dd s
\leq C\bigg(\int_0^t r^{-\frac{1}{\rho H}}\dd r\bigg)^{2H}
\end{aligned}
\end{equation*}
by Lemma \ref{lem:1/H norm}, which is finite if $H>\frac{1}{\rho}$. On the other hand, note that
\begin{equation*}
\begin{split}
&\int_0^t\int_0^t r^{-\frac{1}{\rho}}s^{-\frac{1}{\rho}}|r-s|^{2H-2}\dd r\dd s
\geq\int_0^t r^{-\frac{1}{\rho}}\dd r\int_0^r s^{-\frac{1}{\rho}}(r-s)^{2H-2}\dd s\\
\geq&\int_0^t r^{-\frac{2}{\rho}}\dd r\int_0^r (r-s)^{2H-2}\dd s=\frac{1}{2H-1}\int_0^t r^{-\frac{2}{\rho}+2H-1}\dd r,
\end{split}
\end{equation*}
which is finite if $-\frac{2}{\rho}+2H-1>-1$, which is the same condition as $H>\frac{1}{\rho}$.

When $x\ne0$, we have
\begin{equation*}
\begin{aligned}
\mathbf{E}_{(t,x)}^{\otimes2}\left[\int_0^t\int_0^t|r-s|^{2H-2}\dd L_r\dd L'_s\right]&=\int_0^t\int_0^tg_\rho (t-r,x)g_\rho (t-s,x)|r-s|^{2H-2}\dd r\dd s\\
&=\int_0^t\int_0^tg_\rho (r,x)g_\rho (s,x)|r-s|^{2H-2}\dd r\dd s.
\end{aligned}
\end{equation*}
By Lemma \ref{lem:g-bound}, $g_\rho(t,x)\leq C|x|^{-(1+\rho)}t$. Hence, the integral is finite for all $H>\frac12$.
\end{proof}

\

\subsection{Projections of the equation}\label{sec:project_eq}
Recall the mild solution to the SHE \eqref{e:she} with $u_0(x)\equiv1$ from Definition~\ref{def: solution}, i.e., 
\begin{equation}\label{e:mild-solution}
u(t,x)=1+\sint g_{\rho}(t-\cdot,x)u(\cdot,0)\mathbf{1}_{[0,t]}(\cdot)\xi.
\end{equation}
Given that such a mild solution $u$ exists, we show in the following proposition that its conditional expectation $\E[u|\xi_1, \dots, \xi_n]$ solves the corresponding SHE driven by the finite-dimensional noise $\{\xi_1, \dots, \xi_n\}$.
\begin{proposition}\label{prop:finit n}
Let  $u\in L^1(\Xi)$ be a mild solution to \eqref{e:she} in the sense of Definition~\ref{def: solution}. Then $u_n:=\bE[u|\cF_n]=\bE[u|\xi_1,\dots,\xi_n]$ solves the projected equation (see Definition \ref{def:projection})
\begin{equation}\label{e:eq-u_n}
\begin{aligned}
u_n(t,x)=1+\sintn  \mathbf{1}_{[0,t]}(\cdot) g_{\rho}(t-\cdot,x)u_n(\cdot,0)\xi. 
\end{aligned}
\end{equation}
\end{proposition}
\begin{proof}
By Proposition~\ref{prop-conditional expectation},
    we have
 \begin{equation*}
\begin{aligned}
 \bE\bigg[\sint g_{\rho}(t-\cdot,x)u(\cdot,0)\mathbf{1}_{[0,t]}(\cdot)\xi\bigg|\mathcal{F}_n\bigg]&=\sintn\bE\big[g_{\rho}(t-\cdot,x)u(\cdot,0)\mathbf{1}_{[0,t]}\big|\mathcal{F}_n\big]\xi\\&=\sintn g_{\rho}(t-\cdot,x)
\mathbf{1}_{[0,t]}\bE\big[u(\cdot,0)\big|\mathcal{F}_n\big]\xi\\&=\sintn g_{\rho}(t-\cdot,x)
\mathbf{1}_{[0,t]}u_n(\cdot,0)\xi.
\end{aligned}   
 \end{equation*}
Taking the conditional expectation $\bE[\cdot|\cF_n]$ for \eqref{e:mild-solution}, we get  \eqref{e:eq-u_n}. 
\end{proof}

Next, we provide an explicit formula for $u_n$. Recall that $\{e_k(t)\}_{k\in\bN}$ is a bounded orthonormal basis of $\cH$. To simplify notations, throughout the rest of this paper we denote
\begin{equation}\label{def:tilde_e}
\tilde{e}_k(t):=\int_0^T e_k(r)|t-r|^{2H-2}\dd r.
\end{equation}
It is not hard to check that $\tilde{e}_k(t)$ is also a bounded function. We have the following proposition.
\begin{proposition}\label{prop:exis-uniq-u-n}
There exists a unique solution to \eqref{e:eq-u_n}, which is given by 
\begin{equation}\label{e:unt}
u_{n}(t,x)=Z_n(t,x):=\mathbf{E}_{(t,x)}[M_{n}(t)],
\end{equation}
where 
\begin{equation}\label{e:Mnt}
M_{n}(t):=\exp \left\{\sum_{k=1}^n\left(m_{k}(t)\xi_k-\frac{1}{2}m_{k}(t)^2\right)\right\}
\end{equation}
with $m_k(t)$ defined by 
\begin{equation}\label{e:mkt}
 m_{k}(t):=\int_0^t \tilde{e}_k(s)\dd L_s.
\end{equation}
\end{proposition}
\begin{proof}
Note that $m_k(t)=\int_0^t \tilde e_k(s)\dd L_s$ is well-defined $\fP_{(t,x)}$-a.s.\ (see Section \ref{sec:local}),
and thus $G_n(t):=\sum_{k=1}^n\left(m_{k}(t)\xi_k-\frac{1}{2}m_{k}(t)^2\right)$ is also well-defined $(\fP_{(t,x)}\times\bP)$-a.s. By \eqref{e:Mnt}, we have 
$$\begin{aligned}
M_n(t) &= 1 + \int_0^t M_n(s)\dd G_n(s)=1 + \int_0^t M_n(s) \sum_{k=1}^n \left( \xi_k - m_k(s) \right) \mathrm{d}m_k(s) \\
&= 1 + \sum_{k=1}^n \int_0^t \left( M_n(s)\xi_k - M_n(s)m_k(s) \right) \mathrm{d}m_k(s).
\end{aligned}$$
Then we get
\begin{equation}\label{e:u_n}
\begin{aligned}
u_n(t,x):=&\mathbf{E}_{(t,x)}\left[M_n(t)\right]=1+\mathbf{E}_{(t,x)}\left[ \sum_{k=1}^n \int_0^t \left( M_n(s)\xi_k - M_n(s)m_k(s) \right) \dd m_k(s) \right]\\
=&1+\mathbf{E}_{(t,x)}\left[ \sum_{k=1}^n \int_0^t \tilde{e}_k(s) \left( M_n(s)\xi_k - M_n(s)m_k(s) \right)\dd L_s \right]\\
=&1+\mathbf{E}_{(t,x)}\left[\sum_{k=1}^{n}\int_0^t\tilde{e}_k(s)\left(M_n(s)\xi_k-\partial_{\xi_k} M_n(s)\right)\dd L_s\right].
\end{aligned}
\end{equation}
Note that conditional on $\xi$, the term $\exp(-\frac12\sum_{k=1}^n m_k(s)^2)$ in $M_n(s)$ dominates the integrand in the above integral. Hence, $\sum_{k=1}^n \big|\tilde e_k(s)\left( M_n(s)\xi_k - \partial_{\xi_k}M_n(s)\right)\big|$ is uniformly bounded $\bP$-a.s., and we can apply Lemma~\ref{lem:L-property} to the second term on the right-hand side of \eqref{e:u_n} and get 
\begin{equation*}
\begin{aligned}
&\quad \mathbf{E}_{(t,x)}\left[\sum_{k=1}^{n}\int_0^t\tilde{e}_k(s)\Big(M_n(s)\xi_k-\partial_{\xi_k} M_n(s)\Big)\dd L_s\right]
\\&=\int_0^t\mathbf{E}_{(s,0)}\left[\sum_{k=1}^{n}\tilde{e}_k(s)\Big(M_n(s)\xi_k-\partial_{\xi_k} M_n(s)\Big)\right]g_{\rho}(t-s,x)\dd s.
\end{aligned}   
\end{equation*}

In view of \eqref{e:eq-u_n}, it remains to verify that
\begin{equation}\label{eq:verify_2nd}
\int_0^t\mathbf{E}_{(s,0)}\left[\sum_{k=1}^{n}\tilde{e}_k(s)\Big(M_n(s)\xi_k-\partial_{\xi_k} M_n(s)\Big)\right]g_{\rho}(t-s,x)\dd s=\sintn\mathbf{1}_{[0,t]}(\cdot) g_{\rho}(t-\cdot,x)u_n(\cdot,0)\xi.
\end{equation}
Recall the definition of $L^1$-Skorohod integral from Definitions~\ref{def-Skorohod} and  \ref{def:Skorohod-integral}. For a smooth function $F :\bR^n\to\bR$ with bounded derivatives, we claim 
\begin{equation}\label{ibp}
\begin{aligned}
&\bE\left[F(\xi)\sum_{k=1}^{n}\tilde{e}_k(s)\Big(M_n(s)\xi_k-\partial_{\xi_k} M_n(s\Big)\right]
=\bE\Big[\nabla F(\xi)\cdot\big(\tilde{e}_1(s),\cdots,\tilde{e}_n(s)\big)M_n(s)\Big],
\end{aligned}   
\end{equation}
where $F(\xi) := F(\xi_1, \dots, \xi_n)$. Denote $\mathcal{J}_{\hat{j}}:=\sigma\{\xi_i,i\neq j\}$ and $\hat{\xi}_k(x):=(\xi_1,\cdots,\xi_{k-1},x,\xi_{k+1},\cdots,\xi_n)$. We have that by integration by parts,
\begin{equation}\label{3.5}
\begin{aligned}
&\quad \bE\big[F(\xi)\tilde{e}_k(s)M_n(s)\xi_k \big|\mathcal{J}_{\hat{k}}\big]
\\&=\int_{\bR}  F\big(\hat{\xi}_k(x)\big)\tilde{e}_k(s)M_n\big(\hat{\xi}_k(x)\big)(s)\cdot x \frac{1}{\sqrt{2\pi}}e^{-\frac{x^2}{2}}\dd x
\\&=\int_{\bR}\frac{1}{\sqrt{2\pi}}e^{-\frac{x^2}{2}}\tilde{e}_k(s)\left(\frac{\partial F}{\partial x_k}\big(\hat{\xi}_k(x)\big) M_n\big(\hat{\xi}_k(x)\big)(s)+F\big(\hat{\xi}_k(x)\big)\partial_{\xi_k} M_n\big(\hat{\xi}_k(x)\big)(s)\right)\dd x
\\&=\bE\left[\frac{\partial F}{\partial x_k}(\xi)\tilde{e}_k(s) M_n(s)+F(\xi)\tilde{e}_k(s)\partial_{\xi_k} M_n(s)\Big|\mathcal{J}_{\hat{k}}\right].
\end{aligned}   
\end{equation}
Summing over $k$ and taking expectations in \eqref{3.5},  we prove \eqref{ibp}.

Then we have 
\begin{equation}\label{ibp2}
\begin{aligned}
&\bE\left[F(\xi)\int_0^t\fE_{(s,0)}\sum_{k=1}^{n}\tilde{e}_k(s)\Big(M_n(s)\xi_k-\partial_{\xi_k}M_n(s)\Big)g_{\rho}(t-s,x)\dd s\right]\\
=&\int_0^t\fE_{(s,0)}\bE\left[F\sum_{k=1}^{n}\tilde{e}_k(s)\Big(M_n(s)\xi_k-\partial_{\xi_k}M_n(s)\Big)\right]g_{\rho}(t-s,x)\dd s\\
\xlongequal{\eqref{ibp}}&\int_0^t\fE_{(s,0)}\bE\Big[\nabla F\cdot\big(\tilde{e}_1(s),\cdots,\tilde{e}_n(s)\big)M_n(s)\Big]g_{\rho}(t-s,x)\dd s\\
=&\int_0^t \bE\Big[\nabla F\cdot\big(\tilde{e}_1(s),\cdots,\tilde{e}_n(s)\big)\fE_{(s,0)}\big[M_n(s)\big]\Big]g_{\rho}(t-s,x)\dd s\\
=&\bE\left[\int_0^T \mathbf{1}_{[0,t]}(s)\nabla F\cdot\left(\tilde{e}_1(s),\cdots,\tilde{e}_n(s)\right)u_n(s,0)g_{\rho}(t-s,x)\dd s\right]\\
=&\bE\left[F(\xi)\sintn \mathbf{1}_{[0,t]}(\cdot)g_{\rho}(t-\cdot,x)u_n(\cdot,0)\xi\right],
\end{aligned}   
\end{equation}
which verifies \eqref{eq:verify_2nd}, where the last equality follows from Definitions~\ref{def:Skorohod-integral} and \ref{def:projection}, and the notation \eqref{def:tilde_e}. Therefore, $u_{n}(t,x)=Z_n(t,x):=\fE_{(t,x)}[M_{n}(t)]$ is a solution to \eqref{e:eq-u_n}.

Finally, we prove the uniqueness of the solution. Let $u$, $v$ be two solutions to \eqref{e:eq-u_n}. Consider Fourier type test functions $F(\xi)=\exp\left(i\sum_{k=1}^n \lambda_k\xi_k\right)$. Let $w:=\bE[(u-v)F(\xi)]$. By Definition~\ref{def:Skorohod-integral}, we have that
\begin{equation}\label{eq:w}
\begin{aligned}
w(t,x)&=\bE\left[F(\xi)\sintn \mathbf{1}_{[0,t]}(s) g_{\rho}(t-s,x)(u-v)(s,0) \xi\right]
\\&=\bE\left[ \sum_{j=1}^n i\lambda_j F(\xi) \langle \mathbf{1}_{[0,t]}(\cdot) g_{\rho}(t-\cdot,x)(u-v)(\cdot,0),e_j\rangle_{\mathcal{H}}\right]
\\&=\bE\left[ \sum_{j=1}^n \int_{0}^T i\lambda_j\mathbf{1}_{[0,t]}(s)g_{\rho}(t-s,x)(u-v)(s,0)F(\xi)\tilde{e}_j(s)\dd s\right]
\\&=:\int_0^t\tilde{h}(s) g_{\rho}(t-s,x)w(s,0)\dd s,
\end{aligned}
\end{equation}
where $\tilde{h}(s):=\sum_{j=1}^n i\lambda_j\tilde{e}_j(s)$. Thus, set $x=0$ and we have
\begin{align*}
|w(t,0)|&\leq C\int_0^t|\tilde{h}(s)|g_{\rho}(t-s,0)|w(s,0)|\dd s
\\&\leq C\Vert \tilde h\Vert_\infty\int_0^t (t-s)^{-\frac{1}{\rho}}|w(s,0)|\dd s.
\end{align*}
By Henry-Gr\"{o}nwall Inequality \cite[Lemma 7.1.1]{henry06}, we have $|w(t,0)|\equiv0, \forall t\in[0,T]$. Then by \eqref{eq:w}, $w(t,x)\equiv0$. Denote $\xi_{[n]}:=(\xi_1,\dots, \xi_n)$ and note that $w(t,x)$ is the Fourier transform of $(u-v)(t,x,\xi_{[n]})(2\pi)^{-n/2}\exp(-\frac{1}{2}\|\xi_{[n]}\|^2)$ at $(\lambda_1, \dots \lambda_n)$. As an $L^1(\bR^n)$-function is almost everywhere determined by its Fourier transform, it follows that, for each $(t,x)$, $(u-v)(t,x,\xi_{[n]})=0$  for almost all $\xi_{[n]}\in  \R^n$, and hence $(u-v)(t,x)=0$ $\mathbb P$-a.s. 
\end{proof}

\begin{remark}
We introduce the duplicate notation $Z_n$ for $u_n$, as we will later interpret $u_n$ as a partition function, for which the notation 
$Z_n$ is standard.
\end{remark}

It is not hard to check that $u_n(t,x)$ defined in \eqref{e:unt} is a nonnegative martingale with respect to $\cF_n=\sigma(\xi_1,\cdots,\xi_n)$. Hence, by the martingale convergence theorem, the limit $u(t,x):=\lim_{n\to\infty}u_n(t,x)$ exists $\bP$-a.s. If
this convergence can be strengthened in $L^1(\Xi)$, then $u_n(t,x)=\bE[u(t,x)|\cF_n]$, and in view of Proposition \ref{prop:finit n}, $u$ should be the solution to the SHE \eqref{e:she}. To be specific, we have the following result.

\begin{proposition}[Existence and uniqueness of the solution to \eqref{e:she}]\label{prop:un-u} Suppose that for almost everywhere $(t,x)\in[0,T]\times\bR$, the martingale $u_n(t,x)$ given in \eqref{e:unt} converges in $L^1(\Xi)$. Then the limit $u(t,x)$ is a solution to \eqref{e:she}. The solution is unique $\mathbb P$-a.s. for each $(t,x)$.
\end{proposition}
\begin{proof}
Note that $u_n$ is the unique solution to the projected equation \eqref{e:eq-u_n}. Suppose that the martingale $u_n$ converges in $L^1(\Xi)$. Then $u_n$ is uniformly integrable. By Proposition \ref{uniformly-Skorohod}, we have
$$\sint g_{\rho}(t-\cdot,x)u(\cdot,0)\mathbf{1}_{[0,t]}(\cdot)\xi=\lim_{n\to\infty}\sintn g_{\rho}(t-\cdot,x)u_n(\cdot,0)\mathbf{1}_{[0,t]}(\cdot)\xi.$$
Thus, we get
$$u(x,t)=\lim_{n\to\infty}u_n(x,t)=1+\sint g_{\rho}(t-\cdot,x)u(\cdot,0)\mathbf{1}_{[0,t]}(\cdot)\xi,$$
and hence $u(t,x)$ solves \eqref{e:she}. The uniqueness follows from Propositions \ref{prop:finit n} and \ref{prop:exis-uniq-u-n}.
\end{proof}
\subsection{Randomized shift and partition function}
In this subsection, we interpret the martingale $u_n(t,x)=Z_n(t,x)=Z_n$ as the partition function or total mass of a randomized shift (see Definition~\ref{def-Par}), and establish the criterion for its $L^1$-convergence.

Let $\mathcal{P}_n m$ denote the sequence $m$ where all coordinates beyond $n$ are set to zero. Note that in Proposition~\ref{prop:exis-uniq-u-n}, $Z_n$ is given by 
\begin{equation*} 
Z_n(m(t)):=Z_n(t,x)=\mathbf{E}_{(t,x)}\left[\exp \left\{\sum_{k=1}^n\left(m_{k}(t)\xi_k-\frac{1}{2}m_{k}(t)^2\right)\right\}\right].  
\end{equation*}
Recall that $\mathsf P_{(t,x)}=\fP_{(t,x)}\times\bP$ and $\mathsf E_{(t,x)}=\fE_{(t,x)}\times\bE$ are the joint probability and expectation on the product space $\mathfrak X\times\Xi$, respectively. It is straightforward to verify that $\bE[Z_n(m)F]=\mathsf{E}_{(t,x)}[F(\xi+ \mathcal{P}_nm)]$ for any bounded $F\in\mathcal{F}$, which is a Girsanov transformation, and hence by Definition~\ref{def-Par},  $Z_n$ is the partition function of the randomized shift $\mathcal{P}_n m$. For $A\in \mathcal G$, we denote  $\1_Am=(\1_Am_1, \1_A m_2, 
\dots)$ and denote
\begin{equation*} 
Z_n(\1_A m(t)):=\mathbf{E}_{(t,x)}\left[\exp \left\{\1_A\sum_{k=1}^n\left(m_{k}(t)\xi_k-\frac{1}{2}m_{k}(t)^2\right)\right\}\right].  
\end{equation*}

The following two results are taken from  Proposition~39 and Lemma~42 of \cite{qrv}, respectively. Note that we need an independent copy $m'$ of $m$, which can be obtained by replacing $X$ with an independent copy $X'$ in \eqref{e:mkt}.
\begin{proposition}\label{prop:Z_n-partition} 
Let $Z_n$ be given by \eqref{e:unt} and $\mathcal{P}_n$ be the projection to the first $n$ coordinates. Then for any $F \in \mathcal{F}$ with $\bE|Z_n F| < \infty$ or $F \geq 0$, we have that
\begin{equation*}
    \bE[Z_n F] = \mathsf{E}_{(t,x)} \left[F(\xi + \mathcal{P}_n m)\right].
\end{equation*}
In particular, the following results hold:
\begin{enumerate}
\item $\bE[Z_n]=1$;
\item $\bE[Z_{n}^2]=\fE_{(t,x)}^{\otimes2}[\exp\{\sum_{k=1}^n m_{k}m_{k}'\}]$;
\item For any $A\in\mathcal{G}$, $\bE|Z_n(m)-Z_n(\mathbf{1}_{A}m)|\leq 2\fP_{(t,x)}(A^c)$.
\end{enumerate}
\end{proposition}

\begin{lemma}\label{ab-con}
Suppose that there exists a random variable $Z$ such that $Z_n \to Z$ in $L^1(\Xi, \mathcal{F}, \bP)$, or equivalently $\{Z_n\}_{n\in\bN}$ is uniformly integrable. Then
the marginal law of $(\xi+m)$ on $(\Xi,\mathcal{F})$, given by
$$\hat{\mathbb{P}}_{(t,x)}(A):=\mathsf P_{(t,x)}(\xi+m\in A),\quad  \text{for } A\in \mathcal F,$$
is absolutely continuous with respect to $\bP$. Moreover, $Z=\dd \hat{\mathbb{P}}/\dd \bP$ is the Radon–Nikodym derivative.
\end{lemma}
 Next, we define the quantity
\begin{align}\label{e:mu-exp}
\mathfrak{e}(m):=\lim_{n\to\infty}\bE\big[Z_{n}^2\big]=\lim_{n\to\infty}\fE_{(t,x)}^{\otimes2}\left[\exp\sum_{k=1}^n m_{k}m_{k}'\right],
\end{align}
which plays a key role in establishing the $L^1$-convergence of the martingale $\{Z_n\}_{n\in\bN}$. The following result is standard for the $L^2$-convergence of a martingale.
\begin{proposition}\label{prop:Z-convege-L2}
If $\mathfrak{e}(m)<\infty$, then there exists a random variable $Z$, such that as $n\rightarrow\infty$, $Z_n\rightarrow Z$ almost surely and in $L^2(\Xi, \mathcal{F},\bP )$. Furthermore, $\E[Z^2]=\mathfrak{e}(m)$.
\end{proposition}
\begin{proof}
Note that $Z_n^2$ is a sub-martingale. Hence, $\sup_n\bE[Z_n^2]=\mathfrak{e}(m)<+\infty$. The desired result then follows by the $L^2$-martingale convergence theorem.
\end{proof}

For an event $A\in\mathcal{G}$, we denote
\begin{align}\label{e:mu-exp-A}
\mathfrak{e}(
\mathbf{1}_Am):=\lim\limits_{n\to\infty}\bE\big[Z_n(\mathbf{1}_Am)^2\big]=\lim_{n\to\infty}\fE_{(t,x)}^{\otimes2}\left[\exp\sum_{k=1}^n(\mathbf{1}_Am_{k})(\mathbf{1}_{A'}m_{k}')\right],
\end{align}
where $A'$ is an independent copy of $A$. We have the following criterion for the $L^1$-convergence of $Z_n$.

\begin{proposition}\label{prop:Z-convege}
If for any $p\in(0,1)$, there exists a set $A\in\mathcal{G}$ with $Q(A)>p$, such that $\mathfrak{e}(\mathbf{1}_A m)<\infty$, then as $n\rightarrow\infty$ $Z_n\rightarrow Z$ almost surely and in $L^1(\Xi, \mathcal{F},\bP)$.
\end{proposition}
\begin{proof}
The almost sure convergence immediately follows from the fact that $Z_n$ is a non-negative martingale. We then turn to the $L^1$-convergence.

The condition $\mathfrak{e}(\mathbf{1}_A m)<\infty$ ensures that $Z_n(\mathbf{1}_{A}m)$ is an $L^2$-bounded martingale, which converges in $L^2$. By Proposition \ref{prop:Z_n-partition}, we have that for any $n\in\bN$, $\bE|Z_n(m)-Z_n(\mathbf{1}_{A}m)|\leq 2(1-p)$. Then by the triangle inequality and Jensen's inequality, it follows that for any $k,n\in\bN$,
\begin{equation*}
\begin{split}
\bE|Z_k(m)-Z_n(m)|\leq&\bE|Z_k(m)-Z_k(\mathbf{1}_Am)|+\bE|Z_k(\mathbf{1}_Am)-Z_n(\mathbf{1}_Am)|+\bE|Z_n(\mathbf{1}_Am)-Z_n(m)|\\
\leq&4(1-p)+\|Z_k(\mathbf{1}_Am)-Z_n(\mathbf{1}_Am)\|_{L^2(\bP)},
\end{split}
\end{equation*}
which shows that $Z_n(m)$ is a Cauchy sequence in $L^1(\bP)$ and the proof is completed.
\end{proof}
Note that if $Z_n$ converges to $Z$ in $L^1(\bP)$, then $Z$ solves  the SHE~\eqref{e:she} by Proposition~\ref{prop:un-u}. In the following  subsections, we apply Proposition~\ref{prop:Z-convege} to establish the $L^1$-convergence of $Z_n$.



\subsection{Mutual energy of local times}\label{sec:Mut-e} 
In view of \eqref{e:mu-exp} and  \eqref{e:mu-exp-A}, a key ingredient is the quantity 
\begin{equation}\label{eq:alpha_n}
\mE_n(m(t),m'(t)) := \sum_{i=1}^n m_i(t) m_i'(t).
\end{equation}

For the sake of notational simplicity, we henceforth write $m$ for $m(t)$.
We may view $m$ as a random function on $\mathbb{N}$, defined by $m(i) = m_i$. The tensor product of $m$ on $\bN^k$ is denoted by $m^{\otimes k}$. Let
\begin{equation}\label{e:rho-km}
\rho_{k,m}(\mathbf i) := \fE_{(t,x)}\big[m^{\otimes k}(\mathbf i)\big],\quad\text{for}~\mathbf i=(i_1,\dots,i_k)\in\mathbb N^k.
\end{equation}
Then, $\rho_{k,m}$ is a deterministic function on $\mathbb N^k$, and we have 
\begin{equation}\label{al-rho}
\begin{aligned}
\fE_{(t,x)}^{\otimes2}\big[\mE_n(m,m')^k\big]&=\fE_{(t,x)}^{\otimes2}\left[\sum_{\mathbf{i}\in \{1,\dots,n\}^{k}}\prod_{j=1}^k m_{ i_j} m'_{ i_j}\right]\\ &=\sum_{\mathbf{i}\in \{1,\dots,n\}^{k}}\fE_{(t,x)}\left[\prod_{j=1}^km_{i_j}\right]\fE_{(t,x)}'\left[\prod_{j=1}^km'_{i_j}\right]=\big\Vert\mathbf{1}_{\{1,\dots,n\}^k}\rho_{k,m}\big\Vert_{l^2(\bN^k)}^2,
\end{aligned}
\end{equation}
where $l^2(\mathbb N^k)$ is the Hilbert space of square-summable  sequences with indices in $\bN^k$.

Recall $L_s$ and $\nu_L$ from \eqref{e:local-time} and \eqref{e:nu-L}, respectively.  We next establish the connections among $\mE_n$, $L_s$ and $\nu_L$. For i.i.d.\ local times $L$ and $L'$, define
\begin{equation}\label{e:al*}
\mE_*(\nu_L,\nu'_L):=\int_0^t\int_0^t|r-s|^{2H-2}\dd L_r\dd L'_s,
\end{equation}
which is a nonnegative random variable on $\mathfrak{X}^{\otimes2}$. 

In this subsection,  we aim to show that if $\mE_*(\nu_L,\nu'_L)<+\infty$, $\fP$-a.s., then $$\mE_*(\nu_L,\nu'_L)=\lim_{n\to\infty}\mE_n(m,m'):=\mE(m,m'),\quad \mathbf P \text{-a.s.},$$ 
which is called the \emph{mutual energy} of $L$ and $L'$ (see Definition~\ref{def:al} and Proposition~\ref{prop:al-mut-e} below). The analysis of this quantity is a crucial step in proving the $L^1$-convergence of $u_n(t,x)=Z_n(t,x)$.

Define a deterministic measure on $[0,t]$ by
\begin{align*}
\mu_L(A):=\mathbf{E}_{(t,x)}\left[\int_A \dd L_s\right]=\mathbf E_{(t,x)}\big[\nu_L(A)\big], 
\quad\text{for}~A\in \mathcal B([0,t]).
\end{align*}
Note that
\begin{equation*}
\frac{\mu_L(\dd s)}{\dd s}=\lim\limits_{\e\to0}\fE_{(t,x)}[\delta_\e(X_s)]=g(t-s,x)=:\rho_L(s),
\end{equation*}
which is the density of  $\mu_L$ with respect to the Lebesgue measure. By Corollary~\ref{cor:XX'},  We have that $\rho_L\in \mathcal{H}$ for $H>\frac{1}{\rho}$ if $x=0$ and for $H\in(\frac12,1)$ if $x\neq0$.

The above definition can be extended to the product space $[0,t]^k$ for $k\in\bN$, providing a connection to \eqref{e:rho-km}. For $A\in\mathcal B([0,t]^k)$, we have  that by Lemma \ref{lem:L-property},
\begin{equation*}
\begin{aligned}
\mathbf{E}_{(t,x)}\left[\nu_L^{\otimes k}(A)\right]&=\mathbf{E}_{(t,x)}\left[\int_{\boldsymbol{s
}\in [0,t]^k}\mathbf{1}_A(s_1,\cdots,s_k)\dd L_{s_1}\cdots\dd L_{s_k}\right]=\int_{\boldsymbol{s}\in [0,t]^k}\mathbf{1}_A(s_1,\cdots,s_k)\rho_{k,L}(\boldsymbol{s})\dd \boldsymbol{s},
\end{aligned}   
\end{equation*}
where 
\begin{equation}\label{e:rho-kL}
\rho_{k,L}(\boldsymbol{s}):=k!f_k(t,x;\boldsymbol{s})
\end{equation}
with $f_k$ given by \eqref{e:fn}, denoting the density of $\mathbf E_{(t,x)}[\nu_L^{\otimes k}]$. Meanwhile, recall $\rho_{k,m}$ from \eqref{e:rho-km}, and we obtain that
\begin{align*}
\rho_{k,m}(\mathbf i)&=\mathbf{E}_{(t,x)}\big[m^{\otimes k}(\mathbf i)\big]=\mathbf{E}_{(t,x)}\left[\int_{[0,t]^k} \tilde{e}_{i_1}\otimes\cdots\otimes \tilde{e}_{i_k}\dd\nu_L^{\otimes k}\right]=\int_{[0,t]^k}\prod_{j=1}^k\tilde{e}_{i_j}(s_j)\rho_{k,L}(\boldsymbol{s})\dd \boldsymbol{s},  
\end{align*}
which implies the following isometry
\begin{equation}\label{e:equal-inner}
\|\rho_{k,m}\|^2_{l^2(\bN^k)}=\|\rho_{k,L}\|^2_{\cH^{\otimes k}}.
\end{equation}

Similarly, for a bounded random variable $G$ on $\mathfrak{X}$ with $G\geq0$, we can define a deterministic measure
\begin{equation}\label{e:mu-GL}
\mu_{GL}(A):=\mathbf{E}_{(t,x)}\left[\int _{A}G\dd L_s\right], 
\quad\text{for}~A\in \mathcal B([0,t]).
\end{equation}
Clearly, $\mu_{GL}$ is absolutely continuous with respect to the Lebesgue measure, and its Radon-Nikodym derivative is denoted by
\begin{equation}\label{eq:rho_GL}
\rho_{GL}(s):=\frac{\dd}{\dd s}\mathbf{E}_{(t,x)}\left[\int_0^s G\dd L_r\right].
\end{equation}
Moreover,  under the same conditions for $\rho_L\in \mathcal H$, we have that $\rho_{GL}\in \mathcal{H}$, since $\Vert \rho_{GL}\Vert_{\mathcal{H}}\leq C\Vert\rho_{L}\Vert_{\mathcal{H}}$ by the boundedness of $G$. Equivalently,   $\rho_{GL}$ can be characterized by 
\begin{equation}\label{rho-FL}
\mathbf{E}_{(t,x)}\left[\int_0^t f(s) G\dd L_s\right]=\int_0^t f(s)\rho_{GL}(s)\dd s, \quad\text{for all bounded measurable functions $f$ on $\bR$}.
\end{equation}

With the above notations, we define the \emph{mutual energy} as follows.
\begin{definition}[Mutual energy]\label{def:al}
A random variable $\mE(\nu_L,\nu'_L)\in L^1(\mathfrak{X}^{\otimes 2})$ is called the \emph{mutual energy} of the local times $L$ and $L'$, if for any bounded non-negative random variables $F$ and $G$ on $\mathfrak{X}$, we have that
\begin{equation}\label{al}
\mathbf{E}_{(t,x)}^{\otimes2}\big[FG'\mE(\nu_L,\nu'_L)\big]
=\langle\rho_{FL}, \rho_{GL}\rangle_{\mathcal{H}}< \infty.
\end{equation}
\end{definition}

Recall the expressions for $m_k$ from \eqref{e:mkt} and $\tilde{e}_k$ therein from \eqref{def:tilde_e}. Then $\{m_k\}_{k\in\bN}$ is the projection of $\nu_L$ on the orthonormal basis $\{e_k\}_{k\in\bN}$ of $\cH$. We can view the random function $m=(m_1,m_2,\cdots): \bN\to\bR$ as the coordinates of the random measure $\nu_L$ with respect to $\{e_k\}_{k\in\bN}$. Define $\nu_{L,n}:=\sum_{k=1}^n m_k e_k$ and 
\begin{equation}\label{e:alpha-n}
\mE_{n}(\nu_L,\nu_L'):=\langle \nu_{L,n},\nu'_{L,n}\rangle_\mathcal{H}=\sum_{k=1}^n m_{k}m_{k}'=\mE_n(m,m'),
\end{equation}
which can be viewed as the truncations of $\nu_L$ and $\mE(\nu_L,\nu'_L)$ to the first $n$ dimensions, respectively. In the following proposition, we show that $\mE_n(\nu_L,\nu'_L)$ is indeed an approximation of $\mE(\nu_L,\nu'_L)$.

\begin{proposition}\label{al-converges}
For $k_0\in 2\bN$ and $\rho_{k_0,L}$ defined in \eqref{e:rho-kL}, if $\rho_{k_0,L}\in \mathcal{H}^{\otimes k_0}$, then for all integers $1\le k\le k_0$,
\begin{enumerate}

\item  As $n\to\infty$, $\mE_n(\nu_L,\nu_L')$ converges in $L^k(\mathbf{P}_{(t,x)}^{\otimes 2})$, and the limit is the mutual energy $\mE(\nu_L, \nu_L')$ defined in Definition~\ref{def:al};

\item $\Vert\mE(\nu_L,\nu_L')\Vert^k_{L^k(\mathbf{P}_{(t,x)}^{\otimes2})}=\Vert\rho_{k,L}\Vert^2_{\mathcal{H}^{\otimes k}}.$ 
\end{enumerate}
\end{proposition}
\begin{proof}
Recall \eqref{e:alpha-n}, and we have that $\mE_n(\nu_L,\nu'_L)=\mE_n(m,m')$. By \eqref{al-rho} and \eqref{e:equal-inner}, and the assumption $\rho_{k_0,L}\in\cH^{\otimes k_0}$, it follows that
$$\lim_{n\to\infty}\Vert\mE_n(m,m')\Vert^{k_0}_{L^{k_0}({\mathbf{P}}_{(t,x)}^{\otimes2})}=\lim_{n\to\infty}\big\Vert\mathbf{1}_{\{1,\dots,n\}^{k_0}}\rho_{k_0,m}\big\Vert_{l^2(\bN^{k_0})}^2=\Vert\rho_{k_0,m}\Vert^2_{l^2(\bN^{k_0})}=\Vert\rho_{k_0,L}\Vert^2_{\cH^{\otimes k_0}}<\infty.$$
Hence, part (2) is a direct consequence of part (1). Then we are devoted to proving part (1). We first show that as $n\to\infty$, $\mE_n(\nu_L,\nu_L')$ converges to a limit $\mE_\infty(\nu_L,\nu'_L)$ in $L^k(\mathbf{P}_{(t,x)}^{\otimes 2})$, and then verify that $\mE_\infty(\nu_L,\nu'_L)$ satisfies Definition \ref{def:al}.

Denote $\zeta(A):=\sum_{z\in A}\rho_{k_0,m}(z)^2$ for $A\subset \mathbb N^k$, which is a  measure on $\bN^k$.  Then for $n\ge l$, we have
\begin{equation*}
\begin{aligned}
\mathbf{E}_{(t,x)}^{\otimes2}\left[\mE_n(m,m')^{k_0}\right]=&\zeta\left(\{1,
\dots,n\}^{k_0}\right)\ge\zeta\left(\{1,
\dots,l\}^{k_0}\right)+\zeta\left(\{l+1,
\dots,n\}^{k_0}\right)\\
=&\mathbf{E}_{(t,x)}^{\otimes2}\left[\mE_l(m,m')^{k_0}\right]+\mathbf{E}_{(t,x)}^{\otimes2}\left[\big(\mE_n(m,m')-\mE_l(m,m')\big)^{k_0}\right].
\end{aligned}
\end{equation*}
Therefore, $\mE_n(\nu_L,\nu'_L)=\mE_n(m,m')$ is a Cauchy sequence in $L^{k}(\mathbf{P}_{(t,x)}^{\otimes2})$, for $1\leq k\leq k_0$.
We denote its limit by  $\mE_{\infty}(\nu_L,\nu_L')$.

To verify that $\mE_\infty(\nu_L,\nu'_L)$ is the mutual energy,  let $F, G\geq0$ be bounded random variables on $\mathfrak{X}$. Recalling \eqref{e:alpha-n} and \eqref{eq:rho_GL}, we have 
\begin{align*}
\mathbf{E}_{(t,x)}^{\otimes2}\big[FG'\mE_n(\nu_L,\nu_L')\big]&=\sum_{i=1}^n\mathbf{E}_{(t,x)}\left[Fm_i\right]\mathbf{E}_{(t,x)}\left[Gm_i\right]
\\&=\sum_{i=1}^n\mathbf{E}_{(t,x)}\left[\int _0^t F\tilde{e}_i(s) \dd L_s\right]\mathbf{E}_{(t,x)}\left[\int _0^t G\tilde{e}_i(s) \dd L_s\right]
\\&=\sum_{i=1}^n \langle\rho_{FL},\tilde{e}_i\rangle_{L^2([0,t])}\langle\rho_{GL},\tilde{e}_i\rangle_{L^2([0,t])}
\\&=\sum_{i=1}^n \langle\rho_{FL},e_i\rangle_{\mathcal{H}}\langle\rho_{GL},e_i\rangle_{\mathcal{H}},
\end{align*}
where the third equality follows from \eqref{rho-FL}.

By the boundedness of $F,G$, $\rho_{FL},\rho_{GL}\in\mathcal{H}$ and $FG'\mE_n(\nu_L,\nu_L')\to FG'\mE_{\infty}(\nu_L,\nu_L')$ in $L^{k}(\mathbf{P}_{(t,x)}^{\otimes2})$. Let $n\to\infty$, and we get
$$\mathbf{E}_{(t,x)}^{\otimes2}[FG'\mE_\infty(\nu_L,\nu_L')]=\langle\rho_{FL},\rho_{GL}\rangle_{\mathcal{H}},$$
which verifies that $\mE_{\infty}(\nu_L,\nu'_L)=\mE(\nu_L,\nu'_L)$ by Definition \ref{def:al}.
\end{proof}

Recall that $\rho=\frac{1}{1-\alpha}$, where $\alpha$ is the exponent of the renewal process introduced in \eqref{eq:pin_dist}. Since the Weinrib-Halperin prediction can be characterized by $\alpha$ more explicitly, from now on, we state our results with conditions on  $\rho$ and $\alpha$ simultaneously. To proceed, we need to check that the condition $\rho_{k,L}\in\cH^{\otimes k}$ in Proposition \ref{al-converges} can be satisfied. We have the following lemma.
\begin{lemma}\label{lem:rhoKL-inH}
For $\rho\in(\frac{1}{H},2]\Longleftrightarrow \alpha\in(1-H, \frac12]$, $\rho_{k,L}\in \mathcal{H}^{\otimes k}$ for all $k\geq 1$.
\end{lemma}
\begin{proof}
By \eqref{e:rho-kL}, we have that for $s_0:=0< s_1<\cdots<s_k<s_{k+1}:=t$,
\begin{equation*}
\begin{aligned}
\rho_{k,L}(\boldsymbol{s})
&=g_\rho (t-s_k,x)\prod_{i= 2}^{k}\frac{C_\rho}{(s_{i}-s_{i-1})^{\frac{1}{\rho  }}}.
\end{aligned}   
\end{equation*}
The desired result then follows by  applying Lemmas \ref{lem:1/H norm} and \ref{lem:g-bound}, that is,
\begin{equation*}
\begin{aligned}
\Vert\rho_{k,L}\Vert^2_{\mathcal{H}^{\otimes k}}\le C_H^k\Vert\rho_{k,L}\Vert^2_{L^{1/H} (\bR^k)}&\le C_H^k\left((k!)^{ 1-\frac{1}{H}}C^{k}\int_{s\in [0,t]^k_<} \prod_{i=1}^{k}\frac{1}{(s_{i+1}-s_{i})^{\frac{1}{\rho H }}}\dd s\right)^{2H}<\infty.
\end{aligned}   
\end{equation*}
\end{proof}

Next, we show that the mutual energy $\mE(\nu_L, \nu_L')$ of the local times $L$ and $L'$ defined in Definition~\ref{def:al} actually coincides with $\mE_*(\nu_L, \nu_L')$ in \eqref{e:al*},  which provides a more tractable and explicit formula for $\mE(L,L')$. 
\begin{proposition}\label{prop:al-mut-e}
For $\rho\in(\frac{1}{H},2]\Longleftrightarrow \alpha\in(1-H, \frac12]$,
 $\mE(\nu_L,\nu'_L)$ defined in Definition~\ref{def:al} exists finitely and equals to $\mE_\ast(\nu_L, \nu'_L)$ defined in~\eqref{e:al*}, $\fP_{(t,x)}$-almost surely.
\end{proposition}
\begin{proof}
By Lemma \ref{lem:rhoKL-inH} and Proposition \ref{al-converges}, $\mE(\nu_L,\nu'_L)$ is well-defined as the $L^{k}(\fP^{\otimes2}_{(t,x)})$-limit of $\mE_n(\nu_L,\nu'_L)$.
Recall \eqref{eq:rho_GL} and denote
$$\tilde{\rho}_{GL}(s):=\int_0^t\rho_{GL}(r)|r-s|^{2H-2}\dd r.$$
By Definition~\ref{def:al}, we have that for all bounded non-negative random variables $F$ and $G$ on $\mathfrak X$, 
\begin{equation*}
\mathbf{E}_{(t,x)}^{\otimes2}\big[FG'\mE(\nu_L,\nu'_L)\big]
=\langle\rho_{FL}, \rho_{GL}\rangle_{\mathcal{H}}=\langle\rho_{FL}, \tilde{\rho}_{GL}\rangle_{L^2([0,t])}.
\end{equation*}
Therefore, it suffices to prove
\begin{equation}\label{e:desired-eq}
\langle\rho_{FL}, \tilde{\rho}_{GL}\rangle_{L^2([0,t])}=\mathbf{E}_{(t,x)}^{\otimes2}\left[FG'\mE_\ast(\nu_L,\nu'_L)\right].
\end{equation}

Note that $\mathbf{E}_{(t,x)}^{\otimes2}\big[FG'\mE(\nu_L,\nu'_L)\big]\leq\|F\|_\infty\|G\|_\infty\|\rho_{1,L}\|_{\cH}^2<\infty$ by Proposition \ref{al-converges}. Hence,
\begin{align*}
\langle\rho_{FL}, \tilde{\rho}_{GL}\rangle_{L^2([0,t])}&=\lim_{N\to\infty}\langle\rho_{FL}, N\wedge\tilde{\rho}_{GL}\rangle_{L^2([0,t])}\\
&=\lim_{N\to\infty}\mathbf{E}_{(t,x)}\left[\int_0^tF(N\wedge\tilde{\rho}_{GL}(s))\dd L_s\right]\\
&=\mathbf{E}_{(t,x)}\left[\int_0^tF\tilde{\rho}_{GL}(s)\dd L_s\right]\\
&=\mathbf{E}_{(t,x)}\left[\int_0^t\int_0^tF\rho_{GL}(r)|r-s|^{2H-2}\dd r \dd L_s \right]
\\&=\mathbf{E}_{(t,x)}^{\otimes2}\left[\int_0^t\int_0^tFG'|r-s|^{2H-2}\dd L'_r\dd L_s \right]
\\&=\mathbf{E}_{(t,x)}^{\otimes2}\left[FG'\mE_\ast(\nu_L, \nu_L')\right]
,
\end{align*}
where the first and the third equalities follow from the  positivity of  $\rho_{FL}$ and $\tilde{\rho}_{GL}$ and the monotone convergence theorem, the second equality follows from \eqref{rho-FL}, the fifth equality follows from \eqref{rho-FL} again and Fubini's theorem. Then we conclude $\mE_*(\nu_L,\nu'_L)=\mE(\nu_L,\nu'_L)$.
\end{proof}

In Proposition~\ref{prop:al-mut-e}, we identify $\mE_{\ast}(\nu_{L},\nu_{L}')$ in \eqref{e:al*} with the mutual energy $\mE(\nu_{L},\nu_{L}')$ defined in Definition~\ref{def:al}.  Our approach is to approximate $\nu_L$ (or $\dd L_s$) by its finite-dimensional projection on the orthonormal basis $\{e_k\}_{k\in\bN}$ of $\cH$. Alternatively, note that by Lemma \ref{lem:L-measure}, we can approximate $\nu_L$ by $\nu_{L,\e}$ given in \eqref{e:nu_e}. Similar to \eqref{e:alpha-n}, we define
\begin{align}
m_{\e,k}(t)&:=\int_0^t\tilde{e}_k(s)\dd L_{s,\e}=\int_0^t\tilde{e}_k(s)\delta_{\e}(X_s)\dd s,\quad\forall k\in\bN,\label{def:m_ej}\\
\nu_{L,\e}^{(n)}&:=\sum\limits_{k=1}^n m_{\e,k}e_k,\label{def:nu_eLn}\\
\mE_n(\nu_{L,\e},\nu'_{L,\e})&:=\langle\nu_{L,\e}^{(n)},(\nu^{(n)}_{L,\e})'\rangle_{\cH}=\sum\limits_{k=1}^n  m_{\e,k}m'_{\e,k}=\mE_n(m_\e,m'_\e).\label{def:al_n_e}
\end{align}
By the same argument as in Proposition \ref{al-converges}, we can show that for any $\e>0$, $\mE_n(\nu_{L,\e},\nu'_{L,\e})$ converges to a limit $\mE(\nu_{L,\e},\nu'_{L,\e})$ in $L^2(\fP^{\otimes 2}_{(t,x)})$. On the other hand, by Lemma \ref{lem:L-measure} again, we have that for any $k,n\in\bN$, as $\e\to0$, $m_{\e,k}\to m_k$ and $\mE_n(\nu_{L,\e},\nu'_{L,\e})\to\mE_n(\nu_L,\nu'_L)=\mE_n(m,m')$ in $L^2(\fP^{\otimes 2}_{(t,x)})$. Then, it is natural to ask whether the classical mollification approach described above provides an equivalent approximation for the mutual energy $\mE(\nu_L,\nu'_L)=\mE_*(\nu_L,\nu'_L)$ by $\mE(\nu_{L,\e},\nu'_{L,\e})$. The following proposition provides an affirmative answer.

\begin{proposition}\label{al_e-to-al}
If $\rho\in(\frac{1}{H},2]\Longleftrightarrow \alpha\in(1-H, \frac12]$, then we have $\mE(\nu_{L,\e},\nu'_{L,\e})\to\mE(\nu_L,\nu'_L)$ in $L^2(\mathbf{P}^{\otimes 2}_{(t,x)})$ as $\e\to 0$.
\end{proposition}
\begin{proof}
We show that for any $\eta>0$, there exists an $\e_0>0$, such that for any $\e\in(0,\e_0)$, $\Vert\mE(\nu_L,\nu'_L)-\mE(\nu_{L,\e},\nu'_{L,\e})\Vert_{L^2(\fP^{\otimes2}_{(t,x)})}<\eta$. First, the triangle inequality yields
\begin{align*}
&\Vert\mE(\nu_L,\nu'_L)-\mE(\nu_{L,\e},\nu'_{L,\e})\Vert_{L^2(\mathbf{P}^{\otimes 2}_{(t,x)})}\\
\leq&\Vert\mE(\nu_L,\nu'_L)-\mE_n(\nu_L,\nu'_L)\Vert_{L^2(\mathbf{P}^{\otimes 2}_{(t,x)})}+\Vert\mE_n(\nu_L,\nu'_L)-\mE_n(\nu_{L,\e},\nu'_{L,\e})\Vert_{L^2(\mathbf{P}^{\otimes 2}_{(t,x)})} 
\\&+\Vert\mE_n(\nu_{L,\e},\nu'_{L,\e})-\mE(\nu_{L,\e},\nu'_{L,\e})\Vert_{L^2(\mathbf{P}^{\otimes 2}_{(t,x)})}\\
=&:\Delta_1+\Delta_2+\Delta_3.
\end{align*} 

For the term $\Delta_1$, Proposition~\ref{al-converges}  together with Lemma~\ref{lem:rhoKL-inH} imply that $\mE_n(\nu_L,\nu'_L)$ converges to $\mE(\nu_L,\nu'_L)$ in $L^k(\fP^{\otimes2}_{(t,x)})$ for all $k\in\bN$. Hence,  for any given $\eta>0$, there exists $n_1=n_1(\eta)\in\mathbb N$, such that for all $n\geq n_1$, we have $\Delta_1<\frac{\eta}{4}$.

For the term $\Delta_3$, we have that,
\begin{equation*}
\begin{aligned}
(\Delta_3)^2=&\fE_{(t,x)}^{\otimes2}\left[\left(\sum_{k=n+1}^{\infty} m_{\e,k}m'_{\e,k}\right)^2\right]\\
=&\fE_{(t,x)}^{\otimes2}\left[\sum_{k,j=n+1}^{\infty} (m_{\e,k}m_{\e,j})(m'_{\e,k}m'_{\e,j})\right]=\sum_{k,j=n+1}^\infty\Big(\fE_{(t,x)}\left[m_{\varepsilon,k}m_{\e,j}\right]\Big)^2.
\end{aligned}
\end{equation*}
A straightforward computation shows that
\begin{equation*}
\begin{aligned}
\fE_{(t,x)}\left[m_{\varepsilon,k}m_{\e,j}\right]&= \fE_{(t,x)}\left[\int_0^t\int_0^t \tilde e_k(r_1) \delta_\e (X_{r_1})\tilde e_j(r_2) \delta_\e (X_{r_2}) \dd r_1 \dd r_2\right]\\
&=\fE_{(t,x)}\left[\int_{[0,t]^4} e_k(s_1)e_j(s_2)\delta_\e(X_{r_1})\delta_\e(X_{r_2})|r_1-s_1|^{2H-2}|r_2-s_2|^{2H-2}\dd\boldsymbol{s}\dd\boldsymbol{r}\right]\\
&=\langle\rho_{2,\e},e_k\otimes e_j\rangle_{\mathcal{H}^{\otimes2}},
\end{aligned}
\end{equation*}
where $\rho_{2,\e}(s_1, s_2):=\mathbf{E}_{(t,x)}[\delta_{\e}(X_{s_1})\delta_{\e}(X_{s_2})]$. By the triangle inequality, $(a+b)^2\leq 2(a^2+b^2)$, and recalling $\rho_{2,L}$ from \eqref{e:rho-kL}, we get
\begin{equation}\label{eq:Delta3}
\begin{aligned}
(\Delta_3)^2=\sum_{k,j=n+1}^\infty \langle\rho_{2,\e},e_k\otimes e_j\rangle^2_{\mathcal{H}^{\otimes2}}\le 2\sum_{k,j=n+1}^\infty\Big(\langle\rho_{2,L},e_k\otimes e_j\rangle^2_{\mathcal{H}^{\otimes2}}+\langle\rho_{2,\e}-\rho_{2,L},e_k\otimes e_j\rangle^2_{\mathcal{H}^{\otimes2}}\Big).
\end{aligned}
\end{equation}
By Lemma \ref{lem:rhoKL-inH}, $\rho_{2,L}\in\mathcal{H}^{\otimes2}$ and hence $\sum_{k,j=n+1}^{\infty} \langle \rho_{2,L},e_k\otimes e_j\rangle_{\mathcal{H}^{\otimes2}}^2$ vanishes as $n\to\infty$. This implies that for any given $\eta>0$, we can find $n_2=n_2(\eta)\in\bN$, such that $\sum_{k,j=n+1}^{\infty} \langle \rho_{2,L},e_k\otimes e_j\rangle_{\mathcal{H}^{\otimes2}}^2<\frac{1}{16}\eta^2$ whenever $n\geq n_2$. For the second term above, observe that for $s_1>s_2$, \begin{equation*}
    \begin{aligned}
\rho_{2,\e}(s_1,s_2)&=\int_{\bR^2}\delta_{\e}(y_1)\delta_{\e}(y_2)g_\rho(t-s_1,y_1-x)g_\rho(s_1-s_2,y_1-y_2)\dd y_1\dd y_2\\
&\leq C\int_{\bR^2}\delta_{\e}(y_1)\delta_{\e}(y_2)(t-s_1)^{-\frac1\rho}(s_1-s_2)^{-\frac1\rho}\dd y_1\dd y_2=C(t-s_1)^{-\frac1\rho}(s_1-s_2)^{-\frac1\rho}
\end{aligned}
\end{equation*}
by Lemma \ref{lem:g-bound}, and thus for $s_1>s_2$, $\big|\rho_{2,L}(\boldsymbol{s})-\rho_{2,\e}(\boldsymbol{s})\big|\leq C(t-s_1)^{-1/\rho}(s_1-s_2)^{-1/\rho}$
by \eqref{e:rho-kL}. Then by Lemmas~\ref{lem:1/H norm} and~\ref{lem:a direct calculate}, we have that
\begin{equation*}
\begin{split}
\Vert\rho_{2,L}-\rho_{2,\e}\Vert_{\mathcal{H}^{\otimes2}}
\leq&C\left(\int_{[0,t]^2_<}(t-s_1)^{-\frac{1}{\rho H}}(s_2-s_1)^{-\frac{1}{\rho H}}\dd s_1\dd s_2\right)^{2H}<\infty
\end{split}
\end{equation*}
for $\rho\in(\frac{1}{H},2].$  Noting that $\rho_{2,\e}\to\rho_{2,L}$ almost everywhere as $\e\to0$ and by the dominated convergence theorem, we have that
\begin{equation*}
\Vert\rho_{2,L}-\rho_{2,\e}\Vert_{\mathcal{H}^{\otimes2}}=\sum\limits_{k,j=1}^\infty \langle\rho_{2,\e}-\rho_{2,L},e_k\otimes e_j\rangle^2_{\mathcal{H}^{\otimes2}}\longrightarrow0,\quad\text{as}~\e\to0.
\end{equation*}
Hence, there exists $\e_1$ such that for any $\e\in(0,\e_1)$, $\Vert\rho_{2,L}-\rho_{2,\e}\Vert_{\mathcal{H}^{\otimes2}}^2\leq \frac{1}{16}\eta^2$. Combining the analysis of the two terms in \eqref{eq:Delta3}, we have that for $\e\in(0,\e_1)$ and $n\geq n_2$,
\begin{equation*}
(\Delta_3)^2\leq\frac{\eta^2}{8}+2\|\rho_{2,L}-\rho_{2,\e}\|_{\cH^{\otimes2}}<\frac{\eta^2}{4}.
\end{equation*}



Finally, for the term $\Delta_2$, we have that for any fixed $n\in \mathbb N$, 
\begin{align*}
(\Delta_2)^2
=&\fE_{(t,x)}^{\otimes 2}\left[\left(\sum_{k=1}^n \big(m_{\e, k} m'_{\e, k}-m_{ k} m'_{ k}\big)\right)^2\right]\\
=&\sum_{k,j=1}^n\fE_{(t,x)}^{\otimes 2}\left[ \big(m_{\e, k} m'_{\e, k}-m_{ k} m'_{ k}\big)\big(m_{\e, j} m'_{\e, j}-m_{ j} m'_{ j}\big)\right]\\
=&\sum_{k,j=1}^n\left(\fE_{(t,x)}[m_{\e,k}m_{\e,j}]^2+\fE_{(t,x)}[m_{k}m_{j}]^2-\fE_{(t,x)}[m_{\e,k}m_{j}]^2-\fE_{(t,x)}[m_{k}m_{\e,j}]^2\right).
\end{align*}
Noting that for any fixed $k\in\bN$, $m_{\e,k}\to m_k$ in $L^2(\fP^{\otimes2}_{(t,x)})$ by Lemma \ref{lem:L-measure}. Hence, as $\e\to0$, $\fE_{(t,x)}[m_{\e,k}m_{\e,j}]\to\fE_{(t,x)}[m_k m_j]$ and $\fE_{(t,x)}[m_{\e,k}m_{j}]\to\fE_{(t,x)}[m_k m_j]$ for any fixed $k,j\in\bN$. Fix $n_3=n_1\vee n_2$ and then there exists $\e_2=\e_2(n_3,\eta)=\e_2(\eta)>0$ such that for any $\e\in(0,\e_2)$, $\Delta_2<\frac{\eta}{4}$.

Combining all the estimates above, we obtain that for any $\eta>0$, there exists $n_3\in\bN$, which only depends on $\eta$, and $\e_0=\e_1\wedge\e_2$, which depends on $n_3$ and $\eta$, such that for any $\e\in(0,\e_0)$,
\begin{align*}
&\Vert\mE(\nu_L,\nu'_L)-\mE(\nu_\e,\nu'_\e)\Vert_{L^2(\mathbf{P}^{\otimes 2}_{(t,x)})}\leq\Delta_1+\Delta_2+\Delta_3\le \frac{\eta}{4}+\frac{\eta}{4}+\frac{\eta}{2}=\eta,
\end{align*}
which concludes the results.
\end{proof}

\begin{remark}
As discussed in Remark~\ref{rmk:mollified_she}, an alternative approach to constructing the solution to the SHE \eqref{e:she} is to identify the limit $\lim_{\e\to0}u_\e$, where $u_\e$ in \eqref{e:FK-e} solves the mollified SHE~\eqref{e:she-e}. If the $L^1$-convergence of $u_\e$ can be established, then Proposition~\ref{al_e-to-al} ensures that $\lim_{\e\to0} u_\e$ coincides with the solution $u$ given by \eqref{e:FK} and \eqref{e:FK'}.

Moreover, Proposition~\ref{al_e-to-al} is of independent interest, since it establishes a bridge between the classical mollification approach and the projection approach in the study of mutual energy. For further discussion on mutual energy, see Section~\ref{sec:al+H>1}.
\end{remark}

\subsection{\texorpdfstring{$L^1$-convergence for $\rho\in(\frac{1}{2H-1},2]\Longleftrightarrow\alpha+2H>2$  and $\alpha\leq\frac12$}{}}\label{sec:al+2H>2}
In this section, we show the existence of an $L^1$-solution to our SHE \eqref{e:she} for $\rho\in(\frac{1}{2H-1},2]\Longleftrightarrow\alpha+2H>2$ and $\alpha\leq\frac12$. Recall the solution $u_n(t,x)=Z_n(t,x)$ to the projected equation \eqref{e:eq-u_n} from \eqref{e:unt}. The following result, together with the uniqueness in Proposition \ref{prop:un-u}, completes the proof of Theorem \ref{thm-well-posedness} (3).

\begin{proposition}\label{prop:Zn-Z}
If $\rho\in(\frac{1}{2H-1},2]\Longleftrightarrow\al+2H>2$ and $\alpha\leq\frac12$, then as $n\to\infty$, $Z_n(t,x)\to Z(t,x)$ in $L^1(\bP)$.
\end{proposition}
\begin{proof}
Recall the mutual energy $\mE(\nu_L,\nu'_L)=\mE_*(\nu_L,\nu'_L)$ from \eqref{al} and \eqref{e:al*}. Noting that
\begin{equation}\label{eq:decomp_int}
|r-s|^{2H-2}=C_H\int_\bR|r-u|^{-\frac{3-2H}{2}}|s-u|^{-\frac{3-2H}{2}}\dd u,
\end{equation}
we have
\begin{equation*}
\begin{aligned}
\mE(\nu_L,\nu_L')&=\int_0^t\int_0^t|r-s|^{2H-2}\dd L_r\dd L'_s\\
&=C_H\int_R\left(\int_0^t|r-u|^{-\frac{3-2H}{2}}\dd L_r\right)\left(\int_0^t|s-u|^{-\frac{3-2H}{2}}\dd L'_s\right)\dd u\\
&\leq C_H\left(\int_\bR\left(\int_0^t|r-u|^{-\frac{3-2H}{2}}\dd L_r\right)^2\dd u\cdot\int_\bR\left(\int_0^t|s-u|^{-\frac{3-2H}{2}}\dd L'_s\right)^2\dd u \right)^{\frac{1}{2}}\\
&\leq \frac{C_H}{2}\left(\int_\bR\left(\int_0^t|r-u|^{-\frac{3-2H}{2}}\dd L_r\right)^2\dd u+\int_\bR\left(\int_0^t|s-u|^{-\frac{3-2H}{2}}\dd L'_s\right)^2\dd u \right)\\
&=\frac{1}{2}\left(\int_0^t\int_0^t|r-s|^{2H-2}\dd L_r\dd L_s+\int_0^t\int_0^t|r-s|^{2H-2}\dd L'_r\dd L'_s\right),
\end{aligned}
\end{equation*}
where we have used Cauchy-Schwarz inequality and \eqref{eq:decomp_int} for the second and the last equalities. 

By Corollary \ref{cor:XX},
\begin{equation}\label{e:self-int}
\mathbf{E}_{(t,x)}\left[\int_0^t\int_0^t|r-s|^{2H-2}\dd L_r\dd L_s\right]<\infty,\quad \text{iff}~\rho\in\left(\frac{1}{2H-1},2\right]. 
\end{equation}
Then, define $$A_l:=\left\{\omega:\int_0^t\int_0^t|r-s|^{2H-2}\dd L_r\dd L_s<l\right\},$$ and $A_l^{'}$ is defined in the same way  with $L$ replaced by  $L'$. Due to \eqref{e:self-int}, for any $p\in(0,1)$, we can find large enough $l>0$ , such that $\mathbf{P}_{(t,x)}(A_l)>p$. The same computation as \eqref{al-rho} yields that for any $k\in\bN$,
$$\mathbf{E}_{(t,x)}^{\otimes2}\left[\mathbf{1}_{A_l\times A_l^{'}}\mE_{ n}^k\right]=\Vert\mathbf{1}_{\{1,\cdots,n\}^k}\rho_{k,\mathbf{1}_{A_l}m}\Vert^2_{l^2(\bN^k)}\le\Vert\rho_{k,\mathbf{1}_{A_l}m}\Vert^2_{l^2(\bN^k)}=\mathbf{E}_{(t,x)}^{\otimes2}\left[\mathbf{1}_{A_l\times A_l^{'}}\mE^k\right].$$

Finally, recall $\mathfrak{e}(\mathbf{1}_Am)$ from \eqref{e:mu-exp-A}. By Proposition \ref{prop:Z-convege}, to show that as $n\to\infty$, $Z_n$ converges in $L^1(\bP)$, it suffices to verify that $\mathfrak{e}(\mathbf{1}_Am)<+\infty$. We have that
\begin{equation}\label{e:mu-exp-ineq}
\begin{aligned}
\mathfrak{e}(\mathbf{1}_{A_l} m)&=\lim_{n\to\infty}\mathbf{E}_{(t,x)}^{\otimes2}\left[\exp\left\{\sum_{i=1}^n\mathbf{1}_{A_l\times A_l^{'}}m_{i}m_{i}'\right\}\right]=\lim_{n\to\infty}\mathbf{E}_{(t,x)}^{\otimes2}\left[\exp\left\{\mathbf{1}_{A_l\times A_l^{'}}\mE_{n}\right\}\right]\\
&\leq 2\lim_{n\to\infty}\mathbf{E}_{(t,x)}^{\otimes2}\left[\cosh\left( \mathbf{1}_{A_l\times A_l^{'}}\mE_n\right)\right]=2\lim_{n\to\infty}\mathbf{E}_{(t,x)}^{\otimes2}\left[\sum_{k=0}^\infty\frac{1}{(2k)!}\left(\mathbf{1}_{A_l\times A_l^{'}}\mE_{n}\right)^{2k}\right]\\
&=2\lim_{n\to\infty}\sum_{k=0}^\infty\frac{1}{(2k)!}\mathbf{E}_{(t,x)}^{\otimes2}\left[\mathbf{1}_{A_l\times A_l^{'}}\mE_{n}^{2k}\right]=2\sum_{k=0}^\infty\frac{1}{(2k)!}\mathbf{E}_{(t,x)}^{\otimes2}\left[\mathbf{1}_{A_l\times A_l^{'}}\mE^{2k}\right]\\
&=2\mathbf{E}_{(t,x)}^{\otimes2}\left[\cosh\left( \mathbf{1}_{A_l\times A_l^{'}}\mE\right)\right] <\infty,
\end{aligned}
\end{equation}
where we have used Fubini's theorem and the bounded convergence theorem in the third line, and Taylor expansion in  the second and the last lines, which completes the proof.
\end{proof}

\subsection{Discussion on the condition \texorpdfstring{$\rho\in(\frac{1}{2H-1},2]\Longleftrightarrow\alpha+2H>2$  and $\alpha\leq\frac12$}{}}\label{sec:al+H>1}

Recall the Weinrib--Halperin prediction from Section \ref{sec:pinning-model}. It conjectured that our SHE \eqref{e:she} or its discrete counterpart, the disorder pinning model \eqref{def:Gibbs}--\eqref{e:partition-dis} should be disorder relevant whenever $\alpha+H>1$ for $H\in(\frac12,1)$. This suggests that the SHE \eqref{e:she} should have an $L^1$-solution if $\rho>\frac{1}{H}\Longleftrightarrow\alpha+H>1$. In our analysis throughout Sections \ref{sec:local}--\ref{sec:al+2H>2}, the only place where we have relied on the stronger assumption $\rho\in(\frac{1}{2H-1},2]\Longleftrightarrow\alpha+2H>2$  and $\alpha\leq\frac12$ is in establishing the $L^1$-convergence of $Z_n(t,x)\to Z(t,x)$ as $n\to\infty$ in Section \ref{sec:al+2H>2}, which depends on $\alpha+2H>2$ via Corollary~\ref{cor:XX}. Otherwise, all other analyses remain valid under the weaker condition $\alpha+H>1$.

In the proof for Proposition \ref{prop:Zn-Z}, we applied a Cauchy-Schwarz inequality to decouple $L_s$ and $L'_s$ so that we can introduce the crucial event $A_l$, which may enlarge the mutual energy $\mE(\nu_L,\nu'_L)$ by introducing the \emph{self-energy}
\begin{equation}\label{def:self_energy}
I_t:=\int_0^t\int_0^t|r-s|^{2H-2}\dd L_r\dd L_s.
\end{equation}
However, apart from that, we emphasize that our other derivation is already optimal. To define the event $A_l$, we use the fact $\fE_{(t,x)}[I_t]<\infty$ iff $\rho\in(\frac{1}{2H-1},2]$  from Corollary \ref{cor:XX}, which is a sufficient condition for  $\fP_{(t,x)}(I_t=\infty)=0$. Surprisingly, it turns out that this condition is also necessary by the following result, whose proof is due to a delicate application of the Hausdorff $\varphi$-measure and the Riesz capacity.



\begin{lemma}\label{lem:infty}
If $\rho\in(\frac{1}{2H-1},2]\Longleftrightarrow\al+2H\leq2$  and $\alpha\leq\frac12$, then $\fP_{(t,x)}(I_t=\infty|L_t>0)=1$.
\end{lemma}
\begin{proof}
Define $\mathcal Z_t:=\{s\in[0,t]: X_s=0\}$ be the zero level set of $X$ on $[0,t]$. By \cite[Definition 8.4]{M99}, the Riesz capacity $\mathcal{C}_{2-2H}(\mathcal Z_t)\geq L_t^{2}I_t^{-1}$. Hence, to show that $\fP_{(t,x)}(I_t=\infty|L_t>0)=1$, it suffices to show that $\fP_{(t,x)}(\mathcal{C}_{2-2H}(\mathcal Z_t)=0)=1$, noting that $L_t<\infty$ $\fP_{(t,x)}$-a.s. Then by \cite[Theorem 8.9(1)]{M99}, it further suffices to show that the Hausdorff $\psi$-measure $\cH^{\psi}$ induced by $\psi(s):=s^{2-2H}$ has a finite measure on  $\mathcal Z_t$.

Recall that $\rho>1$. Then by \cite[Theorem~1]{TW66}, $\fP_{(t,x)}$-a.s., $\dd L_s$ is equivalent to a finite positive Hausdorff $\varphi$-measure $\mathcal{H}^\varphi$ on $\mathcal{Z}_t$ with 
$$\varphi(s): = s^{1-1/\rho}(\log |\log s|)^{1/\rho}=s^{\al}(\log |\log s|)^{1-\al}.$$
Note that by $\alpha+2H\leq2$ and $\alpha\leq\frac12$,
$$\lim_{s\to 0}\frac{\psi(s)}{\varphi(s)}=\lim_{s\to 0}\frac{s^{2-2H}}{s^{\alpha}(\log |\log s|)^{1-\al}}=0.$$
Hence, by the comparison principle \cite[Theorem 40]{R70}, the finiteness of $\cH^{\varphi}$ on $\mathcal Z_t$ implies that $\cH^{\psi}$ has zero measure on $\mathcal Z_t$, which concludes the result.
\end{proof}

From the lemma above, we conjecture that the condition $\alpha+2H=2$ may also mark a form of criticality, the nature of which remains unclear at present. Recall further that the Weinrib–Halperin prediction suggests $\alpha+H=1$ as the threshold between disorder relevance and irrelevance, while $\alpha=\frac12$ separates the $L^2$- and $L^1$-regimes. The regime
\begin{equation*}
\bigg\{(\alpha,H):\alpha+H>1, \alpha+2H\leq2, \alpha<\frac12\bigg\}
\end{equation*}
remains largely unexplored, and we believe that new approaches are needed to address this parameter range. The result below illustrates that, when $\alpha+2H>2$, our model indeed exhibits finer structural properties, providing further evidence in support of this conjecture.

\begin{proposition}\label{prop:delta-L2}
Under the condition $\rho\in(\frac{1}{2H-1},2]\Longleftrightarrow\al+2H>2$ and $\alpha\leq\frac12$, $\{\delta_\e(X_.)\}_{\e>0}$ is a Cauchy sequence in  $L^2(\mathfrak X; \mathcal H)$ as $\e\to 0$.
\end{proposition}

\begin{proof}

For any $\e,\eta>0$,
\begin{equation*}
\begin{aligned}
\mathbf{E}_{(t,x)}\big[\langle\delta_\e(X_\cdot),\delta_\eta(X_\cdot)\rangle_\mathcal{H}\big]&=\fE_{(t,x)}\left[\int_0^t\int_0^t\delta_\e (X_r)\delta_\eta(X_s)|s-r|^{2H-2}\dd r\dd s\right].
\end{aligned}
\end{equation*}
Taking $s<r$, we have that by Lemma \ref{lem:g-bound},
\begin{equation*}
\begin{aligned}
\mathbf{E}_{(t,x)}\left[\delta_\e (X_r)\delta_\eta(X_s)\right]=&\int_\bR\int_\bR\delta_\e(y_1)\delta_\eta(y_2)g_\rho(s-r,y_1-y_2)g_\rho(t-s,x-y_2)\dd y_1\dd y_2\\
\leq&C (s-r)^{-\frac{1}{\rho}}(t-s)^{-\frac{1}{\rho}}.
\end{aligned}
\end{equation*}
Thus, by $\rho\in(\frac{1}{2H-1},2]$, Fubini's theorem and Lemma \ref{lem:a direct calculate},
$$\mathbf{E}_{(t,x)}[\langle\delta_\e(X_.),\delta_\eta(X_.)\rangle_\mathcal{H}]\leq 2C\int_0^t\int_0^s(s-r)^{2H-2-\frac{1}{\rho}}(t-s)^{-\frac{1}{\rho}}\dd r\dd s<+\infty.$$
Hence, $\delta_\e(X_.)\in\mathcal{H}$, $\fP_{(t,x)}$-a.s.\ for any fixed $\e>0$.

Then by the dominated convergence theorem, we have
\begin{equation}\label{eq:domin-con}
\begin{aligned}
\lim_{\e,\eta\rightarrow0}\mathbf{E}_{(t,x)}[\langle\delta_\e(X_.),\delta_\eta(X_.)\rangle_\mathcal{H}]&=2\int_0^t\int_0^sg_\rho(s-r,0)g_\rho(t-s,x)|s-r|^{2H-2}\dd r\dd s<\infty.
\end{aligned}
\end{equation}
Therefore, 
\begin{equation*}
\begin{aligned}
\mathbf{E}_{(t,x)}\left[\Vert\delta_\e(X_.)-\delta_\eta(X_.)\Vert_\mathcal{H}^2\right]=&\mathbf{E}_{(t,x)}\left[\langle\delta_\e(X_.)-\delta_\eta(X_.),\delta_\e(X_.)-\delta_\eta(X_.)\rangle_\mathcal{H}\right]
\\=&\mathbf{E}_{(t,x)}\left[\|\delta_\e(X_.)\|^2_\mathcal{H}\right]+\mathbf{E}_{(t,x)}\left[\|\delta_\eta(X_.)\|^2_\mathcal{H}\right]-2\mathbf{E}_{(t,x)}\left[\langle\delta_\e(X_.),\delta_\eta(X_.)\rangle_\mathcal{H}\right],
\end{aligned}
\end{equation*}
which tends to $0$ as $\e,\eta\to0$ by \eqref{eq:domin-con}. This shows that $\{\delta_\e(X_\cdot)\}_{\e>0}$ is a Cauchy sequence in $L^2(\mathfrak X;\mathcal H)$ and the proof is completed.
\end{proof}


 \begin{remark} \label{rem:self-energy}
Recalling \eqref{e:local-time} and \eqref{e:u-c-L},   when $\rho\in(\frac{1}{2H-1},2]$, by Proposition \ref{prop:delta-L2}, we can write $\dd L_s=\delta_0(X_s)\dd s$ where $\delta_0(X_\cdot)=\lim_{\e\to0}\delta_\e(X_\cdot) \in L^2(\mathfrak X; \mathcal H)$.    Then, we have that
\begin{align*}
&\sum_{k=1}^\infty m_k^2=\sum_{k=1}^\infty\langle e_i,\delta_0(X)\rangle_{\mathcal{H}}^2=\langle \delta_0(X),\delta_0(X)\rangle_{\mathcal{H}}=\int_0^t\int_0^t|r-s|^{2H-2}\dd L_r\dd L_s    
\end{align*}
is $L^1(\mathbf P_{t,x})$-integrable (see also Corollary~\ref{cor:XX}) and hence  is finite $\fP_{(t,x)}$-a.s.  On the other hand, when $\al+2H\le 2$,  Lemma \ref{lem:infty} yields that $\sum_{k=1}^\infty m_k^2=\infty$, $\fP_{(t,x)}$-a.s., given $L_t>0$. 
\end{remark}

\subsection{Revisiting the case \texorpdfstring{$\rho =2$}{}}   In Section \ref{Sec: Lp}, we showed that for $\rho=2$, our SHE \eqref{e:she} has a local $L^p$-solution for any $p\geq1$ by the classical Wiener chaos expansion. In this subsection, we provide a more explicit and concise expression than \eqref{e:second_moment} for the second moment of the solution to \eqref{e:she}, with the help of the mutual energy $\mE(\nu_L,\nu'_L)=\mE_*(\nu_L,\nu'_L)$. Our tool is Proposition \ref{prop:Z-convege-L2}, and we first need the following preliminary result.

\begin{lemma}\label{lem:mut-exp-eq}
Recall $\mathfrak{e}(m)$ from \eqref{e:mu-exp}. If $\fE^{\otimes2}_{(t,x)}[e^{\mE(\nu_L,\nu'_L)}]<\infty$, then $\mathfrak{e}(m)=\fE^{\otimes2}_{(t,x)}[e^{\mE(\nu_L,\nu'_L)}]$.
\end{lemma}

\begin{proof}
The proof is similar to \eqref{e:mu-exp-ineq}. Note that $\mE(\nu_L,\nu'_L)$ is non-negative by \eqref{e:al*}, and for $x>0$, $\exp(x)/\cosh(x)\in[1,2]$. It then suffices to show $\fE_{(t,x)}^{\otimes2}[\cosh\mE(\nu_L,\nu'_L)]<+\infty$ and then apply the dominated convergence theorem.

Recall $\mE_n(\nu_L,\nu'_L)$ from \eqref{eq:alpha_n}. By Taylor expansion and Fubini's Theorem, we get
\begin{align*}
\mathbf{E}_{(t,x)}^{\otimes2}\left[\cosh\mE_n(\nu_L,\nu'_L)\right]=\mathbf{E}_{(t,x)}^{\otimes2}\left[\sum_{k=0}^\infty\frac{1}{(2k)!}\mE_{n}(\nu_L,\nu'_L)^{2k}\right]=\sum_{k=0}^\infty\frac{1}{(2k)!}\mathbf{E}_{(t,x)}^{\otimes2}\left[\mE_{n}(\nu_L,\nu'_L)^{2k}\right].
\end{align*}
By \eqref{al-rho}, we have that for any $k\in\bN$,
$$\mathbf{E}_{(t,x)}^{\otimes2}\left[\mE_{ n}(\nu_L,\nu'_L)^k\right]=\Vert\mathbf{1}_{\{1,\cdots,n\}^k}\rho_{k,m}\Vert^2_{l^2(\bN^k)}\le\Vert\rho_{k,m}\Vert^2_{l^2(\bN^k)}=\mathbf{E}_{(t,x)}^{\otimes2}\left[\mE(\nu_L,\nu'_L)^k\right].$$
Then by our assumption, it follows that
\begin{align*}
\mathbf{E}_{(t,x)}^{\otimes2}\left[\cosh \mE(\nu_L,\nu'_L)\right]\le\sum_{k=0}^\infty\frac{1}{k!}\mathbf{E}_{(t,x)}^{\otimes2}\left[\mE(\nu_L,\nu'_L)^k\right]
=\mathbf{E}_{(t,x)}^{\otimes2}\left[\exp\mE(\nu_L,\nu'_L)\right]<\infty.
\end{align*}
Hence, we can apply the dominated convergence theorem to get
\begin{align*}
\lim_{n\to\infty}\mathbf{E}_{(t,x)}^{\otimes2}\left[\cosh\mE_n(\nu_L,\nu'_L)\right]&=\lim_{n\to\infty}\sum_{k=0}^\infty\frac{1}{(2k)!}\mathbf{E}_{(t,x)}^{\otimes2}\left[\mE_{n}(\nu_L,\nu'_L)^{2k}\right]\\
&=\sum_{k=0}^\infty\frac{1}{(2k)!}\mathbf{E}_{(t,x)}^{\otimes2}\left[\mE(\nu_L,\nu'_L)^{2k}\right]=\mathbf{E}_{(t,x)}^{\otimes2}\left[\cosh \mE(\nu_L,\nu'_L)\right].
\end{align*}
Finally, we have that by the dominated convergence theorem again, 
\begin{equation*}
\begin{aligned}
\mathfrak{e}(m)&=\lim_{n\to\infty}\mathbf{E}_{(t,x)}^{\otimes2}\left[\exp\mE_n(\nu_L,\nu'_L)\right]
=\mathbf{E}_{(t,x)}^{\otimes2}\left[\exp\mE(\nu_L,\nu'_L)\right].
\end{aligned}
\end{equation*}
\end{proof}
The following proposition is the main result of this subsection.
\begin{proposition}\label{prop:FK-moment} When $\rho=2$ and  $t>0$ is small enough,  for the second-moment of the mild solution to \eqref{e:she}, we have 
\begin{equation}\label{e:FK_2}
\bE\big[u(t,x)^2\big]=\fE_{(t,x)}^{\otimes2}\left[\exp \sum_{k=1}^\infty m_k m'_k\right]=\fE_{(t,x)}^{\otimes2}\left[\exp{\left\{\int_0^t\int_0^t|r-s|^{2H-2}\dd L_r\dd L'_s \right\}}\right]<\infty.   
\end{equation}
\end{proposition}
\begin{proof}
The first equality in \eqref{e:FK_2} follows by a direct second moment computation with the help of the Feynman-Kac formula \eqref{e:FK}. For the second equality, by Proposition \ref{prop:Z-convege-L2}, it suffices to show that $\fE^{\otimes2}_{(t,x)}[e^{\mE(\nu_L,\nu'_L)}]<\infty$, and then the result follows from Lemma \ref{lem:mut-exp-eq} and the monotone convergence theorem.


By Taylor expansion and Tonelli's theorem,
\begin{align*}
\fE^{\otimes2}_{(t,x)}\left[\exp\mE(\nu_L,\nu'_L)\right]&=1+\sum\limits_{n=1}^\infty\frac{1}{n!}\fE^{\otimes2}_{(t,x)}\left[\mE(\nu_L,\nu'_L)^n\right]\\
&=1+\sum_{n=1}^\infty\frac{1}{n!}\fE_{(t,x)}^{\otimes2}\left[\left(\int_0^t\int_0^t|r-s|^{2H-2}\dd L_r\dd L'_s \right)^n\right]\\
&=1+\sum_{n=1}^\infty \frac{1}{n!}\int_{[0,t]^n}\int_{[0,t]^n}\prod_{i=1}^n|s_i-s_i'|^{2H-2}\rho_{n,L}(\boldsymbol{s})\rho_{n,L'}(\boldsymbol{s}')\dd \boldsymbol{s}\dd \boldsymbol{s}'\\
&=1+\sum_{n=1}^\infty n!\int_{[0,t]^n}\int_{[0,t]^n}\prod_{i=1}^n|s_i-s_i'|^{2H-2}f_n(t,x;\boldsymbol{s})f_n(t,x;\boldsymbol{s}')\dd \boldsymbol{s}\dd \boldsymbol{s}' 
\\
&=1+\sum_{n=1}^\infty n!\Vert f_n(t,x)\Vert_{\mathcal{H}^{\otimes n}}^2<+\infty,
\end{align*}
where the fourth equation follows from \eqref{e:rho-kL}, and the finiteness is due to Proposition \ref{prop:local-L^2} for $t\in(0,t_c)$.
\end{proof}

\begin{remark}
For $\rho\in(\frac{1}{2H-1},2]$, the Feynman-Kac formula \eqref{e:FK} of the solution to~\eqref{e:she}  holds for all $t>0$. Thus, the first equality in \eqref{e:FK_2} also holds for all $t>0$, which though may be infinite. 
\end{remark}

\subsection{Discussion on Stratonovich solution} 


Suppose that a Stratonovich solution $u(t,x)$ of \eqref{e:she} exists. Then, the Feynman-Kac formula suggests
$$u(t,x)=\fE_{(t,x)}\left[\exp{\left\{\int_0^t \delta_0(X_s)\xi(t-s)\dd s\right\}}\right].$$

For simplicity, we may assume $x=0$. By Taylor expansion and Tonelli's theorem,
\begin{equation}\label{e:stra_lower}
\begin{split}
\E[u(t,0)]
&=\fE_{(t,0)}  \left[\exp{\left\{\frac12\int_0^t \int_0^t \delta_0(X_r) \delta_0(X_s) |s-r|^{2H-2} \dd r\dd s\right\}}\right]\\
&=\sum_{n=0}^\infty \frac{1}{2^n n!} \fE_{(t,0)}\left[\left(\int_0^t \int_0^t \delta_0(X_r) \delta_0(X_s) |s-r|^{2H-2} \dd r\dd s\right)^n\right]\\
&=\sum_{n=0}^\infty\frac{1}{2^n n!}\sum_{\sigma\in\mathscr{P}_{2n}}\int_{[0,t]^{2n}_<}C^{2n}s_1^{-\frac{1}{\rho}}\prod\limits_{k=2}^{2n}(s_k-s_{k-1})^{-\frac{1}{\rho}}\prod\limits_{i=1}^n\frac{\mathrm{d}s_{2i-1}\mathrm{d}s_{2i}}{|s_{\sigma(2i)}-s_{\sigma(2i-1)}|^{2-2H}}\\
&\geq\sum_{n=0}^\infty\frac{t^{(2H-2)n}}{2^n n!}\sum_{\sigma\in\mathscr{P}_{2n}}\int_{[0,t]^{2n}_<}C^{2n}s_1^{-\frac{1}{\rho}}\prod\limits_{k=2}^{2n}(s_k-s_{k-1})^{-\frac{1}{\rho}}\dd\boldsymbol{s}\\
&=\sum_{n=0}^\infty\frac{t^{2n(\alpha+H-1)n}C^{2n}(2n)!}{2^n n!\Gamma(2n\al+1)},
\end{split}
\end{equation}
where we have also used Lemma~\ref{lem:L-property} in the third equality and Lemma \ref{lem:a direct calculate} in the last equality.  By Stirling's approximation, the right-hand side, which is a lower bound of $\E [u(t,0)]$  is finite if $\alpha>\frac12$ or $\alpha=\frac12$ with $t>0$ sufficiently small.

On the other hand, by the same strategy in \cite[Lemma 4.1]{lswz}, we can adjust the kernel $|s_{\sigma(2i)}-s_{\sigma(2i-1)}|^{2H-2}$ in the third line  of \eqref{e:stra_lower} so that the integrand therein is bounded above by the product that consists of a factor $C^{2n}$, $n$ factors of the form $|s_k-s_{k-1}|^{-1/\rho}$, and another $n$ factors of the form $|s_k-s_{k-1}|^{2H-1/\rho-2}$. Then by Lemma \ref{lem:a direct calculate} again, we have that
\begin{equation*}
\bE[u(t,0)]\leq\sum\limits_{n=0}^\infty\frac{C^{2n}(2n)!}{2^n n!\Gamma(2n(\alpha+H-1))},
\end{equation*}
which is finite if $\alpha+H>\frac32$ by Stirling's approximation. Note that $\alpha+H>\frac32$ is stronger than $\alpha+2H>2$ since $H\in(\frac12,1)$.

In light of the above analysis, under our assumptions $\alpha\leq\frac12$ and $\alpha+2H>2$, an $L^1$-Stratonovich solution can only be considered in the special case $\alpha=\frac12$, and even then only for sufficiently small time $t>0$. We therefore do not pursue this case, as it falls outside the main scope of this paper, which is already quite lengthy.

\section{0-1 law and strict positivity of the Skorohod solution}\label{sec:0-1}
Recall $Z_n=Z_n(t,x)$ from \eqref{e:unt}. The limit $Z=Z(t,x):=\lim_{n\to\infty}Z_n$ always exists since $Z_n$ is a nonnegative martingale. In this section, we show that $Z$ satisfies a 0-1 law, i.e., $\bP(Z>0)\in\{0,1\}$. In particular, by Theorem \ref{thm-well-posedness}, $Z=u(t,x)$ is an $L^1$-solution to the SHE \eqref{e:she} when $\alpha+2H>2$, and we show that in this case $\bP(Z>0)=1$, that is, the solution is non-trivial.


\begin{proposition}[0--1 law]\label{prop:0-1} For any $(t,x)\in(0,\infty)\times \mathbb R$, we have that $\bP(Z(t,x)>0)\in\{0,1\}$.
\end{proposition}
\begin{proof}
Our strategy is to show that $\{Z=0\}$ differs from a tail event $A$ (see \eqref{def:tail_event}) with respect to the filtration $\{\cF_n\}_{n\in\bN}$ by at most a set of zero $\bP$-probability, where $\cF_n=\sigma(\xi_1,\cdots,\xi_n)$. Note that $Z(t,x)=\lim_{n\to\infty}\fE_{(t,x)}[M_k Y_{k,n}]$, where $k\in\bN$ is fixed and arbitrary, $M_k$ is given in \eqref{e:Mnt}, and 
\begin{equation}\label{def:M_Y}
Y_{k,n}=\exp\Bigg\{\sum_{i=k+1}^n\Bigg(m_{i}\xi_i-\frac{m_{i}^2}{2}\Bigg)\Bigg\}.
\end{equation}

For any fixed $k$ and $\bP$-a.s.\ noise $\xi$, $C=C_{k,\xi}:=\sup_{m_1,\cdots,m_k}M_k<+\infty$ by $\exp(x\xi-\frac12x^2)<+\infty$. Hence, by the dominated convergence theorem, we have that $\bP$-a.s.,
\begin{equation}\label{eq:one_side_bound}
\lim_{n\to\infty}\fE_{(t,x)}[Y_{k,n}]=0\implies \lim_{n\to\infty}\fE_{(t,x)}[CY_{k,n}]=0\implies Z=\lim_{n\to\infty}\fE_{(t,x)}[M_kY_{k,n}]=0.
\end{equation}
Denote
\begin{equation}\label{def:tail_event}
\Xi_1:=\Big\{\xi:\sup_{m_1,\cdots,m_k}M_k<+\infty, \forall k\in\bN\Big\}\quad\text{and}\quad A=\bigcap_{k=1}^\infty\Big\{\lim_{n\to\infty}\fE_{(t,x)}[Y_{k,n}]=0 \Big\},
\end{equation}
where that $\bP(\Xi_1)=1$ and $A$ is a tail event with respect to $\cF_n$. Then \eqref{eq:one_side_bound} implies that $A\cap\Xi_1\subset\{Z=0\}\cap\Xi_1$.

Hence, if we can show that for some $\Xi_2$ with $\bP(\Xi_2)=1$, $\{Z=0\}\cap\Xi_2\subset A\cap\Xi_2$, then we can conclude the result by $\bP(Z=0)=\bP(A)\in\{0,1\}$, thanks to Kolmogorov's 0--1 law. It suffices to show that on $\Xi_2$, for any fixed $k\in\bN$,
\begin{equation}\label{e:converse}
Z=\lim\limits_{n\to\infty}\fE_{(t,x)}[M_k Y_{k,n}]=0\implies\lim_{n\to\infty}\fE_{(t,x)}[Y_{k,n}]=0.
\end{equation}

To proceed, we first establish the following lemma.
\begin{lemma}\label{exchange}
Let
\begin{equation*}
\begin{split}
a_{m,n}&=a_{m,n}^{(k)}:=\fE_{(t,x)}\Big[\mathbf{1}_{\{\frac1m\leq M_k<\frac{1}{m-1}\}} Y_{k,n}\Big],\\
S_{M,n}&=S_{M,n}^{(k)}:=\sum_{m=M+1}^\infty a_{m,n}^{(k)}=\fE_{(t,x)}\Big[\mathbf{1}_{\{0< M_k<\frac{1}{M}\}} Y_{k,n}\Big].
\end{split}
\end{equation*}
Then we have that $\bP$-a.s., 
\begin{equation}\label{e:interchange-limit}  \lim_{M\to\infty}\lim_{n\to\infty}\sum_{m=1}^M a_{m,n}=\lim_{n\to\infty}\lim_{M\to\infty}\sum_{m=1}^M a_{m,n}.
\end{equation}
\end{lemma}
\begin{proof}
Note that \eqref{e:interchange-limit} is a stochastic version of switching the order of limits. By Lemma \ref{lem:change-order}, it suffices to show that $\bP$-a.s., $\lim_{M\to\infty}\sup_{n\geq k+1} S_{M,n}=0$ for all $k\in\bN$, which is the classical condition of uniform convergence.

For fixed $k,M\in\bN$, $S_{M,n}$ is a martingale for $n\geq k+1$ with respect to the filtration $\cF^{(k)}_n:=\sigma(\xi_{k+1},\cdots,\xi_n)$. Hence, for any $N\geq k+1$ and $\varepsilon>0$, by Doob's maximum inequality,
\begin{equation*}
\bP \Big(\sup\limits_{k+1\leq n\leq N}S_{M,n}>\e\Big)\leq\frac{1}{\e}\bE [S_{M,N}]=\frac{1}{\e}\bE [S_{M,k+1}]=\frac{1}{\e}\fE_{(t,x)}\Big[\bP\Big(0<M_k<\frac{1}{M}\Big)\Big].
\end{equation*}
Let $N\to\infty$ and by the monotone convergence theorem, we have that
\begin{equation*}
\bP \Big(\sup\limits_{n\geq k+1}S_{M,n}>\e\Big)\leq\frac{1}{\e}\fE_{(t,x)}\Big[\bP\Big(0<M_k<\frac{1}{M}\Big)\Big].
\end{equation*}
Recalling the expression of $M_k$ from \eqref{e:Mnt} and by the well-posedness of $m_i$ therein (see Section~\ref{sec: L1-sol}), let $M\to\infty$ and we have that $\sup_{n\geq k+1}S_{M,n}$ converges to $0$ in $\bP$-probability for any fixed $k\in\bN$. Also note that $S_{M,n}$ is decreasing in $M$ and so is $\sup_{n\geq k+1}S_{M,n}$. Hence, by the monotone convergence theorem, $\lim_{M\to\infty}\sup_{n\geq k+1}S_{M,n}=0$, $\bP$-a.s.
\end{proof}
Now we prove \eqref{e:converse}. Note that for all $m\geq 1$,
\begin{equation*}
\fE_{(t,x)}[M_k Y_{k,n}]\geq\fE_{(t,x)}\Big[\mathbf{1}_{\{\frac1m\leq M_k<\frac{1}{m-1}\}}M_k Y_{k,n}\Big]\geq\frac{1}{m}\fE_{(t,x)}\Big[\mathbf{1}_{\{\frac1m\leq M_k<\frac{1}{m-1}\}} Y_{k,n}\Big],
\end{equation*}
where we use the convention $1/0:=+\infty$. Denote by $\Xi_3$ with $\mathbb P(\Xi_3)=1$ the set where the equality \eqref{e:interchange-limit} holds. Then on $\{Z=0\}\cap\Xi_3$, we have, for all $m\geq1$,
\begin{equation*}
\lim_{n\to\infty } \fE_{(t,x)}\Big[\mathbf{1}_{\{\frac1m\leq M_k<\frac{1}{m-1}\}} Y_{k,n}\Big]=0,
\end{equation*}
which implies that for all $M\geq 1$,
\begin{equation*}
0=\sum_{m=1}^M\lim_{n\to\infty } \fE_{(t,x)}\Big[\mathbf{1}_{\{\frac1m\leq M_k<\frac{1}{m-1}\}} Y_{k,n}\Big]=\lim_{n\to\infty}\fE_{(t,x)}\Big[\mathbf{1}_{\{\frac1M\leq M_k<\infty\}} Y_{k,n}\Big].
\end{equation*}
By  Lemma~\ref{exchange} and the monotone convergence theorem, we get 
\begin{equation*}
\begin{split}
0&=\lim\limits_{M\to\infty}\lim_{n\to\infty}\fE_{(t,x)}\Big[\mathbf{1}_{\{\frac1M\leq M_k<\infty\}} Y_{k,n}\Big]=\lim_{n\to\infty}\lim\limits_{M\to\infty}\fE_{(t,x)}\Big[\mathbf{1}_{\{\frac1M\leq M_k<\infty\}}Y_{k,n}\Big]\\
&=\lim\limits_{n\to\infty}\fE_{(t,x)}\Big[\mathbf{1}_{\{0< M_k<\infty\}} Y_{k,n}\Big].
\end{split}
\end{equation*}
Finally, as in the proof for Lemma \ref{exchange}, the well-posedness of $m_i$ in $M_k$ yields $\fP_{(t,x)}(M_k=0)=\fP_{(t,x)}(M_k=\infty)=0$, $\bP$-a.s. Hence, there exists a set $\Xi_4$ with $\bP(\Xi_4)=1$, such that
\begin{equation}
0=\lim\limits_{n\to\infty}\fE_{(t,x)}\Big[\mathbf{1}_{\{0< M_k<\infty\}} Y_{k,n}\Big]=\lim\limits_{n\to\infty}\fE_{(t,x)}[Y_{k,n}].
\end{equation}
Take $\Xi_2=\Xi_3\cap\Xi_4$ and we conclude that $\{Z=0\}\cap\Xi_2\subset A\cap\Xi_2$, which completes the proof.
\end{proof}

A direct corollary is the following proposition.

\begin{proposition}[Strict positivity]\label{prop: positivity} If $\rho\in(\frac{1}{2H-1},2]\Longleftrightarrow\alpha+2H>2$ and $\alpha\leq\frac12$, then for all $(t,x)\in(0,\infty)\times \mathbb R$, we have $\bP (Z>0)=1$.
\end{proposition}
\begin{proof}
The desired result follows from the fact $\E [Z]=1$ when $\rho\in(\frac{1}{2H-1},2]$ and Proposition~\ref{prop:0-1}.
\end{proof}

Recall $M_k$ and $m_i$ from \eqref{e:Mnt} and \eqref{e:mkt}. Note that for any fixed path of $X_t$, as long as $|m_i|<\infty$ for all $i\in\bN$, $M_k$ is a non-negative martingale with respect to $\cF_k=\sigma(\xi_1,\cdots,\xi_k)$. By the martingale convergence theorem, the limit $J=J(X_{\cdot}):=\lim_{k\to\infty}M_k$ exists $\fP_{(t,x)}$-a.s.. We have the following result for the fractional moments of $J$ when $\alpha+2H\leq2$. 
\begin{proposition}\label{prop:J=0}
 Assume $\rho\in(1,\frac{1}{2H-1}]\Longleftrightarrow\alpha+2H\le 2$. Then,  given that $L_t>0$,  we have $\bE \left[J^p\right]=0$, $\fP_{(t,x)}$-a.s.,  for all $p\in(0,1)$.
\end{proposition}
\begin{proof}
Denote $A_k:=\frac12\sum_{i=1}^km_i^2$. Given that $L_t>0$, Remark~\ref{rem:self-energy} yields  $\lim_{k\to\infty}A_k=\infty$, $\fP_{(t,x)}$-a.s., and we get that  $\fP_{(t,x)}$-a.s., 
\begin{equation*}
\begin{aligned}
\bE\big[(M_k)^p\big]=\bE\bigg[\exp\bigg\{\sum_{i=1}^kpm_{i}\xi_i-pA_k\bigg \}\bigg]=\exp\left\{\left(p^2-p\right)A_k\right\}\overset{k\to\infty}{\longrightarrow}0.
\end{aligned}
\end{equation*}
Then by Fatou's Lemma, $\lim_{k\to\infty}\bE [(M_k)^p]\geq \bE [\lim_{k\to\infty}(M_k)^p]=\bE[J^p]$. Hence, $\bE[J^p]=0$ by noting that $J$ is nonnegative.
\end{proof}

\begin{remark}

Noting that $L_t>0$, $\mathbf P_{(t,0)}$-a.s., Proposition~\ref{prop:J=0} yields $J=0$, $\mathbb P\times \mathbf P_{(t,0)}$-a.s.\ when $\alpha+2H\le 2$.
We emphasize that $J=\lim_{k\to\infty}M_k$ is not the solution $Z(t,x)=\lim_{k\to\infty}\fE_{(t,x)}[M_k]$ to the SHE~\eqref{e:she}. It is  tempting to conjecture that $\fE_{(t,x)}[J]=Z(t,x)$, which would in turn imply that $\mathbf E_{(t,0)}[J]=0$, $\mathbb P$-a.s., under the condition $\alpha+2H\le 2$. However, this would contradict the Weinrib–Halperin prediction, according to which $\mathbf E_{(t,0)}[J]$ should be nontrivial whenever $\alpha+H>1$.
Nevertheless, when $\alpha+2H>2$, Theorem~\ref{thm-gmc} ensures that $\fE_{(t,x)}[J]=Z$.
\end{remark}

\subsection*{Acknowledgment}
J.\ Song is partially supported by NSFC (No.\ 12471142) and the Fundamental Research Funds for the Central Universities. R.\ Wei is supported by NSFC (No.\ 12401170) and Xi'an Jiaotong-Liverpool University Research Development Fund (No.\ RDF-23-01-024).

\appendix
\section{Miscellaneous results}\label{sec:appendix}

The following result establishes the connection between the hitting times to $0$ for a heavy-tailed random walk $S=\{S_n, n\in\bN_0\}$ and the renewal process \eqref{eq:pin_dist}.
\begin{lemma}\label{lem:random walk}
Suppose that a heavy-tailed random walk $S=\{S_n, n\in\bN_0\}$
belongs to the domain of normal attraction of a symmetric $\rho$-stable distribution with $\rho\in(1,2]$. Then the return time to zero of $S$ is a renewal process \eqref{eq:pin_dist} with $\al=1-\frac{1}{\rho}\in(0,\frac12]$.
\end{lemma}
\begin{proof}
Let $\phi(t):=\bE_S[\exp\{itS_1\}]$ be the characteristic function of the increment of $S$. It is well-known that $\lim_{n\to\infty}\phi(\frac{t}{n^{1/\rho}})^n=\exp(-c_\rho |t|^\rho)$, which implies that $\log\phi(\frac{t}{n^{1/\rho}})\sim-\frac{1}{n}c_\rho |t|^\rho$ as $n\to\infty$. Note that $\lim_{n\to\infty}\phi(\frac{t}{n^{1/\rho}})=1$. It follows that $-\frac{1}{n}c_\rho |t|^\rho\sim\log\phi(\frac{t}{n^{1/\rho}})\sim\phi(\frac{t}{n^{1/\rho}})-1$ as $n\rightarrow\infty$. Denote $x=\frac{t}{n^{1/\rho}}$, and we have that $\phi(x)-1\sim-c_\rho |x|^\rho$ as $x\rightarrow0$. Then by \cite[Theorem A.11]{Gia07}, we complete the proof.
\end{proof}

The following lemma is a direct consequence of  \cite[Theorem 7.11]{R76}.
\begin{lemma}\label{lem:change-order}
Consider the double-indexed series $\{a_{m,n}\}_{m\geq1,n\geq1}$. If
$$\lim_{M\to\infty}\sup_{n}\sum_{m>M}|a_{n,m}|=0,$$
then
$$\lim_{M\to\infty}\lim_{n\to\infty}\sum_{m=1}^M a_{m,n}=\lim_{n\to\infty}\lim_{M\to\infty}\sum_{m=1}^M a_{m,n}.$$
\end{lemma}
The following inequality is borrowed from  \cite{mmv01}.
\begin{lemma}\label{lem:1/H norm}
For $H\in(\frac{1}{2},1]$, the following inequality holds:
$$\int_{\bR^m}\int_{\bR^m}f(\boldsymbol{t})f(\boldsymbol{s})\prod_{i=1}^m|t_i-s_i|^{2H-2}d\boldsymbol{t}d\boldsymbol{s}\leq C_H^m \|f\|^2_{ L^{1/H}(\bR^m)},  $$
where $C_H$ is a constant depending on H, and we denote $\boldsymbol{t}=(t_1,\cdots, t_m),~\boldsymbol{s}=(s_1,\cdots,s_m).$
\end{lemma}

The following equality can be obtained by a direct computation. 
\begin{lemma}\label{lem:a direct calculate}
Suppose $\al_i<1,~i=1,\cdots,m$, and $\al:=\sum_{i=1}^m \al_i$, then
$$\int_{[0=r_0<r_1<r_2<\cdots<r_m\leq t]}\prod_{i=1}^m(r_i-r_{i-1})^{-\al_i}d\boldsymbol{r}=\frac{\prod_{i=1}^m\Gamma(1-\al_i)}{\Gamma(m-\al+1)}t^{m-\al},  $$
where $\Gamma(x)=\int_0^\infty t^{x-1} e^{-t}dt$ is the Gamma function.
\end{lemma}


The following lemma provides uniform upper and lower bounds for the transition density of $\rho$-stable processes.

\begin{lemma}\label{lem:g-bound}
For the transition density $g_\rho(t,x)$ of a $1$-dimensional $\rho$-stable process, we have the following uniform upper and lower bounds:
\begin{equation}\label{eq:g-bound}
\frac{C_1t}{\left(t^{\frac{1}{\rho}}+|x|\right)^{1+\rho}}\le g_\rho(t,x) \le \frac{C_2t}{\left(t^{\frac{1}{\rho}}+|x|\right)^{1+\rho}}\leq C_2t^{-\frac{1}{\rho}},\quad\forall~x\in\bR~\text{and}~t>0,
\end{equation}
where $C_1,C_2$ are positive constants depending only on $\rho$. Moreover, for $t\leq T$, the lower bound can be simplified by $g_\rho(t,x)\geq \frac{C_1t}{\left(T^{1/\rho}+|x|\right)^{1+\rho}}$. 

\end{lemma}
    \begin{proof}
        By \cite[Theorem 1.2 in Section 1.5]{nolan20},
\begin{equation*} 
g_{\rho}(1,x)\sim c_\rho(1+|x|)^{-(1+\rho)},\quad\text{ as } x\to\infty.
\end{equation*}
Then there exists $A>0$, such that, $$\frac{1}{2}c_\rho(1+|x|)^{-(1+\rho)})\le g_\rho(1,x)\le 2c_\rho(1+|x|)^{-(1+\rho)}, \quad \forall~|x|\geq A.$$
Note that $g_{\rho}(1,x)$ is positive and bounded from below on $[-A,A]$, and so is $(1+|x|)^{-(1+\rho)}$. Hence, there exist $c_1, c_2>0$, such that
\begin{equation*}
c_1 (1+|x|)^{-(1+\rho)}\le g_{\rho}(1,x)\leq c_2 (1+|x|)^{-(1+\rho)},\quad\forall~|x|\leq A.
\end{equation*}
Then we can take $C_1=\min\{c_1,\frac{1}{2}c_\rho\}$ and $C_2=\max\{c_2, 2c_\rho\}$, such that for any $x\in\bR$,
$$C_1(1+|x|)^{-(1+\rho)})\le g_\rho(1,x)\le C_2(1+|x|)^{-(1+\rho)}.$$
Finally, by the scaling property, we have that
\begin{equation*}
\frac{C_1t}{\left(t^{\frac{1}{\rho}}+|x|\right)^{1+\rho}}\le g_\rho(t,x) =t^{-\frac{1}{\rho}}g_{\rho}(1,t^{-\frac{1}{\rho}}x)\le \frac{C_2t}{\left(t^{\frac{1}{\rho}}+|x|\right)^{1+\rho}}\leq C_2t^{-\frac{1}{\rho}}.
\end{equation*}
\end{proof}

\bibliographystyle{plain}
 \bibliography{ref}

\end{document}